\documentclass[a4paper,11pt]{article}

\usepackage{amsmath,amssymb}
\usepackage[a4paper,margin=2.5cm]{geometry}
\usepackage{verbatim}

\usepackage{ntheorem}
\usepackage{thm-restate}
\usepackage{tikz}
\usetikzlibrary{positioning,shapes,shadows}
\usetikzlibrary{arrows,automata}
\usepackage{xcolor}
\usepackage{tikz-cd}

\usepackage[only,llbracket,rrbracket]{stmaryrd}

\usepackage{enumitem}

\usepackage[skip=0pt]{caption}
\usepackage{hyperref}
\hypersetup{colorlinks=true, linkcolor={red!50!black}, citecolor={green!50!black}}

\title{Type systems and maximal subgroups of Thompson's group~$V$}
\author{James Belk, Collin Bleak, Martyn Quick, and Rachel Skipper}

\newcommand{\AddrSet}{\Omega}
\newcommand{\branchPartQ}{\Part[Q]^{\dagger}}
\newcommand{\Cant}{\mathfrak{C}}
\newcommand{\card}[1]{\mathopen{|}#1\mathclose{|}}
\newcommand{\cofsubset}{\subseteq_{cf}}
\newcommand{\coker}{\operatorname{coker}}
\newcommand{\cone}[1]{#1\mathfrak{C}}
\newcommand{\coCF}{$\mathrm{co}\kern .07em\mathcal{CF}$}
\newcommand{\emptyword}{\varepsilon}
\newcommand{\Fix}[2]{\operatorname{Fix}_{#1}(#2)}
\newcommand{\supt}[1]{\operatorname{Supt}(#1)}

\newcommand{\graph}[1]{\Gamma(#1)}
\newcommand{\Ideal}{I}  % The ideal in the semigroup for the nested branching type
\newcommand{\infPart}{\Part^{\ast}} % The set of infinite classes in \Part
\newcommand{\length}[1]{\mathopen{|}#1\mathclose{|}}
\newcommand{\letters}{\{0,1\}}
\newcommand{\Nat}{\mathbb{N}}
\newcommand{\Part}[1][P]{\mathcal{#1}}  % Symbol for a partition  The default \Part typesets a caligraphic P, while \Part[Q] will produce a Q, etc.
\newcommand{\prefix}{\preccurlyeq}
\newcommand{\properprefix}{\prec}
\newcommand{\nprefix}{\not\preccurlyeq}
\newcommand{\presentation}[2]{\langle\,#1\mid#2\,\rangle}
\newcommand{\rhorel}{\mathrel{\rho}}
\newcommand{\sd}[1]{\mathrm{sd}(#1)}  % stable depth - as in Def 5.3
\newcommand{\seq}[1]{\mbox{\boldmath$#1$}}
\newcommand{\set}[2]{\{\,#1\mid#2\,\}}
\newcommand{\sgp}[1]{S(#1)}
\newcommand{\smalladdr}[1]{\Sigma_{#1}}
\newcommand{\Stab}[2]{\operatorname{Stab}_{#1}(#2)}
\newcommand{\swap}[2]{(#1\;\;#2)}
\newcommand{\threecycle}[3]{(#1\;\;#2\;\;#3)}
\newcommand{\stype}[1]{\operatorname{s-type}(#1)}
\newcommand{\stypep}[2]{\operatorname{s-type}_{#1}(#2)}
\newcommand{\addrs}[1]{\operatorname{Addr}(#1)}
\newcommand{\Zint}{\mathbb{Z}}

\renewcommand{\emptyset}{\varnothing}
\renewcommand{\leq}{\leqslant}
\renewcommand{\geq}{\geqslant}

\newcommand{\AND}{\qquad\text{and}\qquad}
\newcommand{\IFF}{\qquad\text{if and only if}\qquad}
\newcommand{\nbd}{\nobreakdash-}
\newcommand{\spc}{\vspace{\baselineskip}}
\newcommand{\viceversa}{\emph{vice versa}}

\newcommand{\qed}{\hspace*{\fill}$\square$}
\newenvironment{proof}{%
  \begin{trivlist}
  \item\textsc{Proof:}}{\qed\end{trivlist}}

\theorembodyfont{\slshape}
\newtheorem{thm}{Theorem}[section]
\newtheorem{lemma}[thm]{Lemma}
\newtheorem{cor}[thm]{Corollary}
\newtheorem{prop}[thm]{Proposition}
\theorembodyfont{\normalfont}
\newtheorem{qu}{Question}
\newtheorem{defn}[thm]{Definition}
\newtheorem{notate}[thm]{Notation}
\newtheorem{rem}[thm]{Remark}
\newtheorem{example}[thm]{Example}

\renewcommand{\labelenumi}{{\normalfont\theenumi}}

\newcommand{\labelcaret}[3]{
  \begin{tikzpicture}
    \node[draw,circle] (a) at (0,0) {$\mathsf{#1}$};
    \node[draw,circle] (b) at (-1,-1.5) {$\mathsf{#2}$};
    \node[draw,circle] (c) at (1,-1.5) {$\mathsf{#3}$};
    \draw (a) -- (b);
    \draw (a) -- (c);
  \end{tikzpicture}}

\begin{document}

\maketitle

\begin{abstract}
We introduce the concept of a type system~$\Part$, that is, a partition on the set of finite words over the alphabet~$\{0,1\}$ compatible with the partial action of Thompson's group~$V$, and associate a subgroup~$\Stab{V}{\Part}$ of~$V$.  We classify the finite simple type systems and show that the stabilizers of various simple type systems, including all finite simple type systems, are maximal subgroups of~$V$.  We also find an uncountable family of pairwise non-isomorphic maximal subgroups of~$V$.  These maximal subgroups occur as stabilizers of infinite simple type systems and have not been described in previous literature: specifically, they do not arise as stabilizers in $V$ of finite sets of points in Cantor space.  Finally, we show that two natural conditions on subgroups of $V$ (both related to primitivity) are each satisfied only by $V$ itself, giving new ways to recognise when a subgroup of $V$ is not actually proper.
%Finally, we consider two conditions for subgroups of~$V$ that could be viewed as related to the concept of primitivity.  We show that in fact only $V$~itself satisfies either of these conditions.

\end{abstract}  

\section*{Introduction}

In this article, we investigate the maximal subgroups of Richard J. Thompson's group~$V$. The group $V$ was first introduced in 1965 \cite{Thompson} as part of a triple of groups $F < T < V$ where it was shown that these groups are all infinite, finitely presented groups, and that $T$ and $V$ are simple.  These groups were then the first  examples of infinite, finitely presented groups that were known to be simple. Indeed, until the arrival of the finitely presented infinite simple groups of Burger and Mozes in 1997 \cite{BurgerMozes}, all known examples of infinite, finitely presented simple groups were in some way modelled on, or built by extending, the constructions of $T$ or $V$.

The Thompson groups in general have had a wide impact, appearing in the theory of many fields of mathematics, and playing host as examples of or counterexamples to numerous questions and conjectures about the existence of groups with various unusual properties.   Some examples of where these groups arise include the word problem for groups \cite{thompsonmckenzie, Thompson76}, in homotopy and shape theory \cite{Dydak1,Dydak2,FreydHeller}, in the algebra of string rewriting \cite{Coh93}, in the theory of diagram groups over semigroup presentations \cite{GS97}, in group cohomology \cite{brown3,brown4,browngeoghegan1}, and even in dynamical systems and analysis \cite{GhysSergiescu}.  There are now hundreds of research articles strongly connected to these groups, and all three groups have difficult questions associated with them (e.g., the question of whether or not $F$ is amenable is a particularly well-known open question). 

A particular area of abiding mystery for R.\ Thompson's group $V$ is its subgroup structure: both in terms of determining its lattice of subgroups as well as trying to understand the isomorphism types of its subgroups.  Since all finite symmetric groups and hence all finite groups embed into~$V$, as well as $\Zint^{n}$ and all free products of finite groups, one might guess that most naturally occurring finitely generated groups embed into~$V$. However, the paper \cite{bleak-salazar1} shows that even the group $\Zint^{2}*\Zint$ fails to embed as a subgroup of $V$. Furthermore, it is also known that the Baumslag--Solitar groups~$B(m,n)$, for $m\neq \pm n$, do not embed into~$V$, as cyclic groups are undistorted in~$V$ (see \cite{BBGGHMS, BMN16,BCR18}).  On the positive side, many known \coCF\ groups embed in $V$ and a form of Lehnert's Conjecture is that a group is a \coCF\ group if and only if it embeds as a finitely generated subgroup of~$V$ (see, e.g., \cite{lehnertthesis,BMN16,FarleyFSS}).  In terms of the lattice structure of the subgroups of~$V$, a first observation is that, as $V$~is finitely generated, every proper subgroup of~$V$ is contained in some maximal proper subgroup.

A standard way of representing the group $V$ is as a subgroup of the automorphism group of the standard binary Cantor space $\Cant$.  It is a folklore result that the stabilizer in $V$ of a point in $\Cant$ is a maximal subgroup of~$V$ (see Corollary~\ref{cor:Stab-tailclass}), with an explicit proof of this result for $F$ in \cite{Savchuk}. In 2017 Golan and Sapir produced the first maximal subgroups of $F$ that do not arise as point stabilizers \cite{Golan-Sapir17a} or as index $p$ subgroups of $F$ for $p$ a prime, and this has begun an active area of research for~$F$ \cite{Golan-Sapir17b,Golan-Sapir17c, Aiello-Nangnibeda21a21a,Aiello-Nangnibeda21a21b}. The groups arising in the above research can be interpreted as groups preserving certain sub-families of dyadic rationals.  In contrast to this, our approach involves partitioning the set of finite prefixes of points in Cantor space and looking at the subgroups of $V$ which preserve the given partition. This is a wholly new way of constructing maximal subgroups for Thompson-type groups.

It is well known that the partial action of~$V$ on the Cantor algebra of clopen subsets of~$\Cant$ shares some characteristics with the natural action of a finite symmetric group, see, for example,~\cite{BQ17}.  Our study of maximal subgroups of~$V$ further exploits this connection.  Recall that a subgroup of the symmetric group of degree~$n$ is either intransitive (in which case it is contained in a direct product of smaller symmetric groups), transitive but imprimitive (in which case it is contained in a wreath product), or primitive.  As a consequence, one has a description of the maximal subgroups of the symmetric group~$S_{n}$ and the famous O'Nan--Scott Theorem~\cite{AS85} is often used to provide further information on the possibilities for primitive maximal subgroups.  The subgroups that we introduce can be viewed as analogous to the intransitive and imprimitive maximal subgroups of~$S_{n}$.  Let $\AddrSet$ be the set of finite (perhaps empty) strings on the alphabet $\{0,1\}$.  We think of these as being addresses of basic clopen sets in $\Cant$. We say that a partition $\Part$ of $\AddrSet$ is a \emph{type system} if the corresponding equivalence relation $\sim$ satisfies
\[
\alpha\sim \beta \quad\Longleftrightarrow\quad \text{$\alpha 0\sim \beta 0$ and $\alpha 1 \sim \beta 1$}
\]
for all $\alpha,\beta\in\AddrSet$.  The elements of $\Part$ are called \emph{types}.  Such a type system is \textit{simple} if the partition $\Part$ is not a refinement of any other non-trivial type system.  Every type system has an associated subgroup $\Stab{V}{\Part}$ that we define in Definition~\ref{def:basic}\ref{i:Fix&Stab}.
This can be viewed as the set of transformations in~$V$ that preserve a partition on the set of cones in~$\Cant$ and is therefore the analogue of the wreath product from the permutation groups setting.  Indeed, we prove that under natural conditions, $\Stab{V}{\Part}$~is a maximal subgroup of Thompson's group~$V$.

\begin{restatable*}{thm}{FiniteSimpleThm}
\label{thm:finitesimple}
Let $\Part$~be a finite simple type system on~$\AddrSet$.  Then $\Stab{V}{\Part}$~is a maximal subgroup of~$V$.
\end{restatable*}

A major step in proving the theorem is to first fully classify which finite type systems can be simple.  To do so, we consider an associated graph, called the \emph{type graph}, which has one vertex for each type with directed edges from the type of $\alpha$ to the types of $\alpha 0$ and $\alpha 1$ for all $\alpha\in\AddrSet$.  We refer to a strongly connected component of this graph as a \emph{nucleus}, and we classify simple type systems based on the number and characteristics of the nuclei.

\begin{restatable*}{thm}{FiniteTypes}
\label{thm:finitetypes}
A finite simple type system on~$\AddrSet$ is either nuclear, atomic binuclear, or atomic quasinuclear.
\end{restatable*}

Theorem \ref{thm:FixIsMatui} shows that the maximal subgroups of $V$ associated with finite simple nuclear type systems are isomorphic to full groups of irreducible, branching subshifts of finite type (or occasionally, they are finite index overgroups of such a group).  These groups were introduced by Matsumoto in \cite{Matsumoto} and studied extensively by Matui in \cite{Matui}.

The taxonomy of simple type systems is somewhat complex, and also includes systems which need not be finite but which nonetheless lead to maximal subgroups.

\begin{restatable*}{thm}{BranchType}
\label{thm:branchtype}
Let $\Part$~be a simple type system on~$\AddrSet$ that is either multinuclear or atomic branching quasinuclear.  Then $\Stab{V}{\Part}$~is a maximal subgroup of~$V$.
\end{restatable*}

It is easy to specify the point stabilizer of any given ``irrational'' point in Cantor space as the stabilizer of an infinite simple type system. Thus,  constructing uncountably many distinct maximal subgroups of $V$ is not a difficult task.  However, it was not known if $V$ admitted uncountably many pairwise non-isomorphic maximal subgroups.  Also, all but one (up to automorphic image) of the previously known maximal subgroups arose as stabilizers of finite sets of points with the same infinite tail class (with the exception being $T$ and its automorphic images).  In contrast to these examples, our technique (see Section \ref{sec:uncountIso}) provides for the construction of an uncountable family of pairwise non-isomorphic maximal subgroups of $V$.  Furthermore, these maximal subgroups are all ``new'': they have never been described in the literature, none of these subgroups stabilize any finite set of points in $\Cant$, and finally, they are not related in any obvious way to  automorphic images of $T$ (specifically, they do not preserve any circular order on a dense set of points in $\Cant$).

\begin{restatable*}{thm}{UncountablyManyIso}
\label{thm:manynonIso}
There is an uncountable family $\mathcal{M}$ of pairwise non-isomorphic maximal subgroups of\/~$V$.  Moreover, every $G$ in~$\mathcal{M}$ arises as a stabilizer of a simple nuclear type system, does not preserve any circular order on any dense set of points in $\Cant$, and does not stabilize any finite set of points in~$\Cant$.

\end{restatable*}

%\begin{restatable*}{thm}{UncountablyMany}
%\label{thm:manyStabs}
%There are uncountably many maximal subgroups of\/~$V$ arising as stabilizers of simple nuclear type systems and where, moreover, none of these subgroups stabilize a finite set of points in~$\Cant$.
%\end{restatable*}

As alluded to above, there is an obvious family of (maximal) subgroups of $V$ which do not preserve any type system, namely Thompson's group $T$ and its set of images under automorphisms of $V$.  If we were to view stabilizers of type systems as analogues of imprimitive groups, then these last groups would play the role analogous to the primitive subgroups of $V$.  We do not know if these are the only ``primitive'' maximal subgroups of $V$, and we mention various questions around this in Section \ref{sec:questions}.

A second goal of this article is to provide theorems which can be used to detect when a subgroup of $V$ is actually proper.  In this vein, we consider notions in the setting of the partial action of~$V$ on the cones of~$\Cant$ that relate to transitivity and to an alternative version of primitivity.

\begin{restatable*}{thm}{TwoFoldTrans}
\label{thm:2foldtrans}
Let $G$~be a subgroup of Thompson's group~$V$ that acts $2$\nbd fold transitively on the set of proper cones of $\Cant$.  Then $G = V$.
\end{restatable*}

We call our alternative form of primitivity (in the setting of partial actions on cones) \emph{swap-primitivity} and prove the following:

\begin{restatable*}{thm}{swapPrimitiveIsV}
\label{thm:swap-primitiveisV}
Let $G$~be a subgroup of~$V$ and suppose that $G$~is swap-primitive.  Then $G = V$.
\end{restatable*}

Section \ref{sec:questions} discusses several questions that we have not been able to address.  The first of which is below.

\begin{restatable}{qu}{QMain}\label{ques:main}
Suppose $G$ is a maximal subgroup of $V$ which does not preserve a type system. Is $G$ an automorphic image of $T$? 
\end{restatable}

\subsection{Acknowledgements}
The authors would like to thank Peter Cameron and James Hyde for helpful conversations and suggestions.

The first and second authors have been partially supported by EPSRC grant EP/R032866/1 during the creation of this paper.  The first author is also grateful for support from the National Science Foundation under Grant No.\ \mbox{DMS-185436}. The fourth author was partially supported by the GIF grant I-198-304.1-2015, ``Geometric exponents of random walks and intermediate growth groups", LMS Grant 21806, NSF DMS--2005297 ``Group Actions on Trees and Boundaries of Trees", and the European Research Council (ERC) under the European Union’s Horizon 2020 research and innovation program (grant agreement No.725773).

\section{Background}

In this section, we define Thompson's group $V$ and set up the notation for the rest of the paper. 

First, let $\Cant$ denote the Cantor set consisting of the infinite words over the alphabet $\letters$.  Let $\AddrSet$ be the set of all finite words in the same alphabet $\letters$.  We refer to the set $\AddrSet$ as the set of \emph{addresses}. Any address $\alpha$ determines a \emph{cone}~$\cone{\alpha}$, which is one of the basic clopen subsets of Cantor space~$\Cant$:
\[
\cone{\alpha} = \set{ \alpha w }{w \in \Cant }.
\]
We write~$\length{\alpha}$ for the \emph{length} of the address~$\alpha$ as a word in $\letters$.  If $\alpha$~is any finite word, then $\overline{\alpha}$~denotes the element $\alpha\alpha\dots$ of~$\Cant$ obtained by repeating~$\alpha$ infinitely many times.

We say a set of cones $\cone{\alpha_{1}}$,~$\cone{\alpha_{2}}$, \dots,~$\cone{\alpha_{n}}$ (respectively, a set of addresses $\alpha_{1}$,~$\alpha_{2}$, \dots,~$\alpha_{n}$) has \emph{small support} if $\cone{\alpha_{1}}\cup \cone{\alpha_{2}} \cup \dots \cup \cone{\alpha_{n}}\neq \Cant$.  We say two cones $\cone{\alpha}$ and $\cone{\beta}$ (respectively, two addressses $\alpha$ and $\beta$) are \emph{incomparable} if $\cone{\alpha} \cap \cone{\beta} = \emptyset$.  We write $\alpha \perp \beta$ to denote that the addresses $\alpha$~and~$\beta$ are incomparable.  Note that two cones are either incomparable or there is an inclusion of one into the other. If $\cone{\alpha} \supseteq \cone{\beta}$, we will write $\alpha \prefix \beta$ for the addresses.  This is equivalent to the word~$\alpha$ being a prefix of the word~$\beta$.  For this reason, for $\alpha \in \AddrSet$ and $w \in \Cant$, we shall also write $\alpha \prefix w$ if $\alpha$~is a prefix of~$w$.

Let $\cone{\alpha_{1}}$,~$\cone{\alpha_{2}}$, \dots,~$\cone{\alpha_{n}}$ and $\cone{\beta_1}$,~$\cone{\beta_{2}}$, \dots,~$\cone{\beta_{n}}$ be two partitions of $\Cant$ into disjoint cones.  We can then define a map $\Cant \to \Cant$ given by $\alpha_{i}w \mapsto \beta_{i}w$ for $i = 1$,~$2$, \dots,~$n$ and all $w \in \Cant$.  Such a map will be called a \emph{prefix substitution map}.  If $\alpha$~and~$\beta$ are incomparable addresses, we denote by~$\swap{\alpha}{\beta}$ the prefix substitution map given by $\alpha w \mapsto \beta w$ and $\beta w \mapsto \alpha w$ for $w \in \Cant$, and with all other points of~$\Cant$ fixed.  We shall refer to such prefix substitution maps as \emph{transpositions}.  We extend this in the natural way to define disjoint cycle notation for prefix substitution maps that arise as finite permutations of disjoint cones.

We are now ready to define Thompson's group $V$.

\begin{defn}
\emph{Thompson's group}~$V$ is the set of all prefix substitution maps on~$\Cant$ with multiplication given by function composition.
\end{defn}

Note that Thompson's group~$V$ has an alternative definition using paired tree diagram. We direct the reader to \cite{CFP} for a detailed description of this viewpoint.

The action of $V$ on the Cantor set induces a partial right action of~$V$ on the set~$\AddrSet$ of addresses.

\begin{defn}
Let $g \in V$ and $\alpha$~be an address.  If $(\cone{\alpha})^{g} = \cone{\beta}$ for some address~$\beta$ and $g$~is given by the prefix substitution $\alpha \mapsto \beta$ on~$\cone{\alpha}$, then we define $\alpha^{g} = \beta$. If there is no address~$\beta$ for which this holds, then $\alpha^{g}$~is undefined.
\end{defn}

It follows immediately from the definition that if $g \in V$ then $\alpha^{g}$~is defined for all but finitely many addresses~$\alpha \in \Omega$.  The following observations are straightforward and will be used without comment in what follows.

\begin{lemma}
\label{lem:partialact}
Let $g$~and~$h$ be elements of~$V$ and $\alpha$~and~$\beta$ be addresses in~$\AddrSet$.
\begin{enumerate}
\item If $\alpha^{g}$~is defined and $\alpha \prefix \beta$, then $\beta^{g}$~is defined.  Moreover, if $\beta = \alpha \eta$ for some word~$\eta$, then $\beta^{g} = (\alpha^{g})\eta$.
\item Suppose that $\alpha^{g}$~is defined.  Then $(\alpha^{g})^{h}$~is defined if and only if $\alpha^{gh}$~is defined and, when it is, $(\alpha^{g})^{h} = \alpha^{gh}$.
\end{enumerate}
\end{lemma}

%\begin{proof}
%(i),~(ii), and~(iii) are clear. 
%
%(iv)~Assume $\alpha^{g}$~is defined.  Thus, on the cone~$\cone{\alpha}$, the map $g$~is given by $\alpha w \mapsto \beta w$ for some $\beta \in \AddrSet$ and then $\alpha^{g} = \beta$ by definition.
%
%Suppose first that $(\alpha^{g})^{h}$~is defined; that is, $\beta^h$~is defined, say $\beta^{h} = \gamma$.  Then $h$~is given by $\beta w \mapsto \gamma w$ on~$\cone{\beta}$.  Now $gh$~is given by $\alpha w \mapsto \gamma w$ on~$\cone{\alpha}$.  Thus $\alpha^{gh}$~is defined and $\alpha^{gh} = \gamma = (\alpha^{g})^{h}$, as required.
%
%Conversely, suppose $\alpha^{gh}$~is defined and equals~$\gamma$.  Thus $gh$~is given by $\alpha w \mapsto \gamma w$ on~$\cone{\alpha}$.  Now the effect of~$h$ on the cone~$\cone{\beta}$ is given by $(\beta w)^{h} = (\alpha w)^{gh} = \gamma w$.  Thus $\beta^{h}$~is defined and equals~$\gamma$.  Hence $(\alpha^{g})^{h}$~is defined and $(\alpha^{g})^{h} = \alpha^{gh}$.
%\end{proof}

%Thus we have a partial action of~$V$ on the set of addresses and as a consequence also a partial action on the set of cones.

\section{Stabilizers of finite sets of fixed tail class}\label{subs:stabilizersoffinitetails}

As a starting point for our study of maximal subgroups of $V$, we will prove that the stabilizer of any finite set of points in $\Cant$ which all have the same tail class is a maximal subgroup. To do this, we will first apply a criterion for primitive group actions given in \cite{DixonMort}.  Given a group~$G$ acting transitively on a set~$S$ and points $a,b \in S$, denote by $\Delta = \set{(a,b)^{g}}{g \in G}$ the orbit of the pair~$(a,b)$ and by~$\Gamma(a,b)$ the associated \emph{orbital graph}; that is, $\Gamma(a,b)$~has vertex set~$S$ and edge set~$\Delta$.  The required criterion is:

\begin{prop}[\protect{\cite[Theorem~3.2A]{DixonMort}}]
\label{prop:DM-prim}
Let $G$~be a group acting transitively on a set~$S$.  Then the action of~$G$ on~$S$ is primitive if and only if the orbital graph $\Gamma(a,b)$ is connected for every distinct $a,b\in S$.
\end{prop}

\begin{thm}
\label{thm:2k-trans->primitive}
Let $G$~be a group that acts $2k$\nbd transitively on a set~$X$ with $\card{X} \geq 2k+1$ for some positive integer~$k$.  Then $G$~acts primitively on the set of subsets of~$X$ with cardinality~$k$.
\end{thm}
\begin{proof}
Let $S$~denote the set of subsets of~$X$ with cardinality~$k$.  Since $G$~is, in particular, $k$\nbd transitive on~$X$, it is the case that $G$~acts transitively on~$S$.  Let $A$~and~$B$ be a pair of distinct subsets of~$X$ with $\card{A} = \card{B} = k$, let $\Delta = \set{(A,B)^{g}}{g \in G}$ be the orbit of the pair~$(A,B)$ and $\Gamma$~be the orbital graph as above.  %; that is, $\Gamma$~has vertex set~$S$ and edge set~$\Delta$.
We shall show that $\Gamma$~is connected as an undirected graph.

Let $C$~be any subset of~$X$ with $\card{C} = k$.  Take $x \in X \setminus C$ and $y \in C$.  Define $C'$~to be the union of~$C \setminus \{y\}$ with~$\{x\}$.  Since $\card{X} \geq 2k+1$, there is a subset~$D$ of~$X$ such that $\card{C \cap D} = \card{A \cap B}$, \ $\card{D} = k$, and $x,y \notin D$.  As $G$~acts $2k$\nbd transitively, there exists $g \in G$ such that $(A,B)^{g} = (C,D)$; that is, $(C,D) \in \Delta$.  Now as $D$~does not contain $x$~or~$y$, it is also the case that $\card{C' \cap D} = \card{C \cap D}$.  Hence, using $2k$\nbd transitivity again, there exists $h \in G$ such that $(C,D)^{h} = (C',D)$.  Thus $(C',D) \in \Delta$.  This shows that there are (undirected) edges joining $C$~and~$D$ and joining $C'$~and~$D$ in~$\Gamma$.  Hence we may travel from the subset~$C$ to a subset obtained by removing one element and replacing it by another.  By induction there is a path in~$\Gamma$ from~$C$ to any other subset of cardinality~$k$.  The result now follows by Proposition~\ref{prop:DM-prim}.
\end{proof}

Recall that two points $u$~and~$v$ in~$\Cant$ have the same \emph{tail class} if $u = \alpha w$ and $v = \beta w$ for some finite words $\alpha,\beta \in \AddrSet$ and some tail~$w \in \Cant$.  Fix a tail class~$X$.  Then since a prefix substitution preserves the tail class, Thompson's group~$V$ acts on~$X$.  Moreover, it is straightforward to show that this action is highly transitive; that is, $V$~is $n$\nbd transitive for every positive integer~$n$.  In particular, we may apply Theorem~\ref{thm:2k-trans->primitive} for any value of~$k$.  Hence $V$~acts primitively on the set of subsets of~$X$ of cardinality~$k$.  The corresponding stabilizer is therefore a maximal subgroup of~$V$ and we have established the following:

\begin{cor}
\label{cor:Stab-tailclass}
Let $D$~be a non-empty finite subset of~$\Cant$ consisting of elements of the same tail class. Then $\Stab{V}{D}$~is a maximal subgroup of~$V$. \qed
\end{cor}

\section{\boldmath Labelling subgroups of~\texorpdfstring{$V$}{V}}

In this section, we use the partial action of~$V$ on~$\AddrSet$ to define a collection of subgroups of~$V$ that will provide many of our examples of maximal subgroups.

\begin{defn}\label{defn:typesystem}
Let $\sim$~be an equivalence relation on the set~$\AddrSet$ of addresses.
\begin{enumerate}
\item We say that $\sim$~is \emph{coherent} if, for all addresses $\alpha$~and~$\beta$, \ $\alpha \sim \beta$ implies that both $\alpha0 \sim \beta0$ and $\alpha1 \sim \beta1$.
\item We say that $\sim$~is \emph{reduced} if it has the property that if $\alpha$~and~$\beta$ are addresses such that both $\alpha0 \sim \beta0$ and $\alpha1 \sim \beta1$, then $\alpha \sim \beta$. 
\item If $\sim$~is an equivalence relation on~$\AddrSet$ that is both coherent and reduced, we refer to the resulting collection~$\Part$ of equivalence classes on~$\AddrSet$ as a \emph{type system}.
\item If $\Part$ is a type system, we refer to the equivalence classes in~$\Part$ as \emph{types}.  For any address~$\alpha$, if $\alpha$~belongs to the equivalence class $P \in \Part$ then we say that $\alpha$~has \emph{$\Part$\nbd type}~$P$.  When it is clear which type system is under consideration, we shall simply refer to the \emph{type} of~$\alpha$.
\end{enumerate}
\end{defn}

The definition of coherent and reduced for an equivalence relation~$\sim$ on~$\AddrSet$ ensures that type systems interact well with the partial action of Thompson's group~$V$.
Notationally, if $P$~is a type in some type system~$\Part$, then we shall denote by $P_{0}$~the type of~$\alpha0$ where $\alpha$~is any address whose type is~$P$.  Analogously we denote~$P_{1}$ to be the type of~$\alpha1$.  The assumption that $\sim$~is coherent ensures that $P_{0}$~and~$P_{1}$ are well-defined types and hence the free monoid with basis~$\{0,1\}$ has an action on the type system~$\Part$.  

\begin{rem}
Under the assumption that $\sim$~is coherent and reduced, the equivalence relation actually determines a congruence on a free Cantor algebra. Then $\Part$~embeds in a natural way in the quotient Cantor algebra determined by this congruence.  We shall, however, not make explicit use of such Cantor algebras here and instead will explore this in a future article.
\end{rem}

We now introduce certain subsets of~$V$ determined by an equivalence relation on the addresses (or, equivalently, on the set of cones in~$\Cant$).

\begin{defn}
\label{def:basic}
Let $\sim$~be an equivalence relation on the set~$\AddrSet$ of addresses.
\begin{enumerate}
\item We say that an element~$g$ of~$V$ \emph{fixes the parts} of~$\sim$ if for all but finitely many $\alpha \in \AddrSet$, \ $\alpha^g$~is defined and $\alpha^g \sim \alpha$. 

\item We say that $g \in V$ \emph{stabilizes}~$\sim$ if there is a subset $\AddrSet_{0} \subseteq \AddrSet$ with $\card{\AddrSet \setminus \AddrSet_{0}} < \infty$ such that $\alpha^g$~is defined for all $\alpha \in \AddrSet_{0}$ and, for $\alpha,\beta \in \AddrSet_{0}$, \ $\alpha \sim \beta$ if and only if $\alpha^{g} \sim \beta^g$.  (That is, for all but finitely many addresses, $g$~preserves equivalence of addresses under~$\sim$.)

\item \label{i:Fix&Stab}
Let $\Part$~denote the partition of~$\AddrSet$ determined by the equivalence relation~$\sim$.  We define
\[
\Fix{V}{\Part} = \set{g \in V}{\text{$g$~fixes the parts of~$\sim$}},
\]
i.e., the set of elements of~$V$ that \emph{fix the parts in~$\Part$}, and
\[
\Stab{V}{\Part} = \set{g \in V}{\text{$g$~stabilizes~$\sim$}},
\]
i.e., the set of elements of~$V$ that \emph{stabilize the partition~$\Part$}.
\end{enumerate}
\end{defn}

By definition, $\Fix{V}{\Part} \subseteq \Stab{V}{\Part}$ and we shall now observe that they are indeed subgroups of~$V$.  The lemma also indicates one of the gains of assuming our relation is coherent: one has greater control over when an element belongs to~$\Fix{V}{\Part}$.  In fact, there is little lost by assuming that $\Part$~is a type system as we shall observe in Lemma~\ref{lem:reduction}.  To aid our arguments, we shall use the notation $\AddrSet_{0} \cofsubset \AddrSet$ to indicate that $\AddrSet_{0}$~is a cofinite subset of~$\AddrSet$; that is, when $\card{\AddrSet \setminus \AddrSet_{0}} < \infty$.

\begin{lemma}%\label{lem:membershiptest}
\label{lem:basic}
Let $\sim$~be an equivalence relation on the set~$\AddrSet$ of addresses and let $\Part$~be its associated partition.  Then:
\begin{enumerate}
\item \label{i:Fix-subgp}
$\Fix{V}{\Part}$~is a subgroup of~$V$.
\item \label{i:Stab-subgp}
$\Stab{V}{\Part}$~is a subgroup of~$V$.
\item \label{i:Fix-normal}
$\Fix{V}{\Part}$~is a normal subgroup of\/~$\Stab{V}{\Part}$.
\item \label{i:Fix-form}
Suppose that $\sim$~is coherent and let $g \in V$.  Then $g \in \Fix{V}{\Part}$ if and only if there exists a partition $\alpha_1\Cant,\ldots,\alpha_n\Cant$ of $\Cant$ into cones so that each $\alpha_i^g$ is defined and $\alpha_i^g \sim \alpha_i$. 
\end{enumerate}
\end{lemma}

\begin{proof}
\ref{i:Fix-subgp}~Certainly the identity element belongs to~$\Fix{V}{\Part}$.
Suppose $g,h \in \Fix{V}{\Part}$.  Then there exist $\AddrSet_{1}, \AddrSet_{2} \cofsubset \AddrSet$ such that $\alpha^{g}$~is defined and satisfies $\alpha^{g} \sim \alpha$ for all $\alpha \in \AddrSet_{1}$ and such that $\alpha^{h}$~is defined and satisfies $\alpha^{h} \sim \alpha$ for all $\alpha \in \AddrSet_{2}$.
It is straightforward to observe $\beta^{g^{-1}} \sim \beta$ for all $\beta \in \AddrSet_{1}^{g}$ and $\AddrSet_{1}^{g} \cofsubset \AddrSet$, so $g^{-1} \in \Fix{V}{\Part}$.
Let $\AddrSet_{0} = \AddrSet_{1} \cap (\AddrSet_{2} \cap \AddrSet_{1}^{g})^{g^{-1}}$.  Then $\AddrSet_{0} \cofsubset \Omega$ and if $\alpha \in \AddrSet_{0}$, then $\alpha^{g}$~is defined, belongs to~$\AddrSet_{2}$ and satisfies $\alpha^{g} \sim \alpha$.  Hence $\alpha^{gh} = (\alpha^{g})^{h} \sim \alpha^{g} \sim \alpha$.  This shows that $gh \in \Fix{V}{\Part}$.  We conclude that $\Fix{V}{\Part}$~is indeed a subgroup of~$V$.

\ref{i:Stab-subgp}~Suppose $g,h \in \Stab{V}{\Part}$. Then there exist subsets $\AddrSet_{1}, \AddrSet_{2} \cofsubset \AddrSet$ such that $\alpha^{g}$~is defined for all $\alpha \in \AddrSet_{1}$, \ $\alpha^g \sim \beta^g$ if and only if $\alpha \sim \beta$ whenever $\alpha, \beta \in \AddrSet_{1}$, \ $\alpha^{h}$~is defined for all $\alpha \in \AddrSet_{2}$ and $\alpha^{h} \sim \beta^{h}$ if and only if $\alpha \sim \beta$ whenever $\alpha, \beta \in \AddrSet_{2}$.
As in~\ref{i:Fix-subgp}, it is straightforward to observe $\gamma^{g^{-1}} \sim \delta^{g^{-1}}$ if and only if $\gamma \sim \delta$ for $\gamma,\delta \in \AddrSet_{1}^{g}$ and hence $g^{-1} \in \Stab{V}{\Part}$.
Similarly, for $\AddrSet_{0} = \AddrSet_{1} \cap (\AddrSet_{2} \cap \AddrSet_{1}^g)^{g^{-1}} \cofsubset \AddrSet$, it is the case that $\alpha^{gh}$~is defined for all $\alpha \in \AddrSet_{1}$.  Let $\alpha, \beta \in \AddrSet_{0}$ and suppose $\alpha \sim \beta$.  Then $\alpha^{g} \sim \beta^{g}$ since $\alpha,\beta \in \AddrSet_{1}$ and then $\alpha^{gh} \sim \beta^{gh}$ since $\alpha^{g}, \beta^{g} \in \AddrSet_{2}$.  This argument reverses and hence $gh \in \Stab{V}{\Part}$.
We deduce that $\Stab{V}{\Part}$~is a subgroup of~$V$.

\ref{i:Fix-normal}~Let $h \in \Fix{V}{\Part}$ and $g \in \Stab{V}{\Part}$. Then there exist $\AddrSet_{1}, \AddrSet_{2} \cofsubset \AddrSet$ such that $\alpha^{h}$~is defined and $\alpha^{h} \sim \alpha$ for all $\alpha \in \AddrSet_{1}$ and such that $\alpha^{g}$~is defined for all $\alpha \in \AddrSet_{2}$ and, for $\alpha,\beta \in \AddrSet_{2}$, \ $\alpha \sim \beta$ if and only if $\alpha^{g} \sim \beta^{g}$.
Then $\bigl( ( \AddrSet_{1} \cap \AddrSet_{2} )^{h} \cap \AddrSet_{2} \bigr)^{h^{-1}}$ is a well-defined cofinite subset of~$\AddrSet$ contained in~$\AddrSet_{1}$.  Hence we can set $\AddrSet_{0} = \bigl( (\AddrSet_{1} \cap \AddrSet_{2})^{h} \cap \AddrSet_{2} \bigr)^{h^{-1}g} \cofsubset \AddrSet$.
If $\alpha \in \AddrSet_0$, then $\alpha^{g^{-1}h} \in \AddrSet_{2}$ and so $\alpha^{g^{-1}hg}$~is defined.  Furthermore, $\alpha^{g^{-1}} \in \bigl( ( \AddrSet_{1} \cap \AddrSet_{2} )^{h} \cap \AddrSet_{2} \bigr)^{h^{-1}} \subseteq \AddrSet_{1}$ and so $\alpha^{g^{-1}h} \sim \alpha^{g^{-1}}$.  Since both $\alpha^{g^{-1}}, \alpha^{g^{-1}h} \in \AddrSet_{2}$, we deduce $\alpha^{g^{-1}hg} \sim \alpha^{g^{-1}g} = \alpha$.  This shows $g^{-1}hg \in \Fix{V}{\Part}$.
  
\ref{i:Fix-form}~Suppose that $g$~is given by a prefix replacement map with domain cones $\cone{\alpha_{1}}$,~$\cone{\alpha_{2}}$, \dots,~$\cone{\alpha_{k}}$ and that $\alpha_{i}^g \sim \alpha_{i}$ for $i = 1$,~$2$, \dots,~$k$.
Let $\AddrSet_{0}$~be the set of addresses having some~$\alpha_{i}$ as a prefix, so that $\AddrSet_{0} \cofsubset \AddrSet$.  If $\alpha \in \AddrSet_{0}$, say $\alpha = \alpha_{i}\eta$ for some word~$\eta$, then by repeated use of the fact that $\sim$~is coherent, we conclude that $\alpha^{g} = (\alpha_{i}^{g})\eta \sim \alpha_{i}\eta = \alpha$.  Hence $g \in \Fix{V}{\Part}$.
  
Conversely, suppose $g \in \Fix{V}{\Part}$.  Let $\AddrSet_{0} \cofsubset \AddrSet$ such that $\alpha^{g}$~is defined for all $\alpha \in \AddrSet_{0}$ and satisfies $\alpha^{g} \sim \alpha$.  By repeated subdivision, we can express~$g$ as a prefix replacement map with all domain cones~$\cone{\alpha_{i}}$ indexed by addresses~$\alpha_{i}$ in~$\AddrSet_{0}$.  Then $\alpha_{i}^{g} \sim \alpha_{i}$ for all~$i$, as required.
\end{proof}

We shall view an equivalence relation~$\sim$ as providing a ``labelling'' of the addresses according to which equivalence classes they belong to; that is, in the case of a type system, according to the type of the addresses.  For a coherent equivalence relation~$\sim$ with corresponding partition~$\Part$, we shall call $\Fix{V}{\Part}$ and $\Stab{V}{\Part}$ the \emph{labelling subgroups} for~$\Part$.  

If $\Part$~has only finitely many finite classes, $\Fix{V}{\Part}$ is in fact the kernel of the action of $\Stab{V}{\Part}$ on the collection of infinite equivalence classes as we shall demonstrate (see Lemma \ref{lem:P*-action}\ref{i:Fix=ker} below).  Let $\sim$~be an equivalence relation on~$\AddrSet$ and $\Part$~be its associated partition.  If $g \in \Stab{V}{\Part}$ then there exists some subset $\AddrSet_{0} \cofsubset \AddrSet$ such that $\alpha^{g}$~is defined for all $\alpha \in \AddrSet_{0}$ and, for $\alpha,\beta \in \AddrSet_{0}$, \ $\alpha \sim \beta$ if and only if $\alpha^{g} \sim \beta^{g}$.  As a consequence, if $P \in \Part$ is an \emph{infinite} equivalence class, then $\alpha^{g} \sim \beta^{g}$ for all $\alpha,\beta \in P \cap \AddrSet_{0}$ and hence $(P \cap \AddrSet_{0})^{g} \subseteq Q$ for some, necessarily infinite, class~$Q \in \Part$.  We can accordingly define an action of~$\Stab{V}{\Part}$ on the infinite classes in~$\Part$.

\begin{defn}
If $P$~is an infinite class in~$\Part$ and $g \in \Stab{V}{\Part}$, define~$P^{g}$ to be the infinite class $Q \in \Part$ such that $\alpha^{g} \in Q$ for all but finitely many $\alpha \in P$.
\end{defn}

\begin{lemma}
\label{lem:P*-action}
Let $\sim$~be an equivalence relation on~$\AddrSet$ and let $\Part$~be the associated partition.  Suppose that the collection of infinite classes in~$\Part$, denoted $\infPart$, is non-empty.  Then:
\begin{enumerate}
\item \label{i:Stab-acts}
$\Stab{V}{\Part}$~acts on~$\infPart$ via the above definition and $\Fix{V}{\Part}$~is contained in the kernel of this action.
\item \label{i:Fix=ker}
If $\Part$~contains only finitely many finite classes, then $\Fix{V}{\Part}$~is the kernel of the action of\/~$\Stab{V}{\Part}$ on~$\infPart$.
\end{enumerate}
\end{lemma}

\begin{proof}
For part~\ref{i:Stab-acts} the proof is straightforward.  
For part~\ref{i:Fix=ker}, suppose that $\Part$~contains only finitely many finite classes.  Let $g$~belong to the kernel of the action of~$\Stab{V}{\Part}$ on~$\infPart$.  There exists some $\AddrSet_{0} \cofsubset \AddrSet$ such that $\alpha^{g}$~is defined for all $\alpha \in \AddrSet_{0}$ and, for $\alpha,\beta \in \AddrSet_{0}$, \ $\alpha \sim \beta$ if and only if $\alpha^{g} \sim \beta^{g}$.  If $P \in \infPart$, then $P^{g} = P$ by definition, so $\alpha^{g} \in P$ for all but finitely many $\alpha \in P$.  In particular, there exists some $\alpha \in P \cap \AddrSet_{0}$ such that $\alpha^{g} \in P$ and then $\beta^{g} \sim \alpha^{g}$ for all $\beta \in P \cap \AddrSet_{0}$.  This shows $(P \cap \AddrSet_{0})^{g} \subseteq P$; that is, $\alpha^{g} \sim \alpha$ for all $\alpha \in P \cap \AddrSet_{0}$.

Take $\AddrSet_{1} = \AddrSet_{0} \cap \bigcup \infPart$; that is, the set of addresses in~$\AddrSet_{0}$ that belong to an infinite class.  Since $\Part$~has only finitely many finite classes, $\AddrSet_{1}$~is a cofinite subset of~$\AddrSet$.  Then $\alpha^{g}$~is defined for all $\alpha \in \AddrSet_{1}$ and, by the previous paragraph, $\alpha^{g} \sim \alpha$ for all $\alpha \in \AddrSet_{1}$.  This shows $g \in \Fix{V}{\Part}$ and completes the proof of part~\ref{i:Fix=ker}.
\end{proof}

\begin{lemma}
\label{lem:reduction}
Let $\sim$~be an equivalence relation on~$\Omega$ and $\Part$~be its associated partition.  Define, for $\alpha,\beta \in \AddrSet$, \ $\alpha \approx \beta$ if and only if $\alpha\eta \sim \beta\eta$ for all but finitely many words~$\eta$.  Then $\approx$~is a coherent and reduced equivalence relation on~$\AddrSet$ such that
\begin{enumerate}
\item \label{i:reduction-quot}
if $\sim$~is coherent and $\alpha \sim \beta$ for some $\alpha,\beta \in \AddrSet$, then $\alpha \approx \beta$;
\item \label{i:reductiongps}
if $\Part[Q]$~is the type system determined by~$\approx$, then $\Fix{V}{\Part[Q]} = \Fix{V}{\Part}$ and $\Stab{V}{\Part} \leq \Stab{V}{\Part[Q]}$.
\end{enumerate}
\end{lemma}

\begin{proof}
It is straightforward to verify that $\approx$~is an equivalence relation on~$\AddrSet$ and that it is both coherent and reduced.  If we assume that $\sim$~is coherent, then part~\ref{i:reduction-quot} also follows directly.

For part~\ref{i:reductiongps}, suppose that $g \in \Fix{V}{\Part}$.  Then there exists $\AddrSet_{0} \cofsubset \AddrSet$ such that $\alpha^{g}$~is defined and $\alpha^{g} \sim \alpha$ for all $\alpha \in \AddrSet_{0}$.  If $\alpha \in \AddrSet_{0}$, then $\alpha\eta \in \AddrSet_{0}$ for all but finitely many words~$\eta$ and $\alpha^{g}\eta = (\alpha\eta)^{g} \sim \alpha\eta$ for all such~$\eta$.  Hence $\alpha^{g} \approx \alpha$ for $\alpha \in \AddrSet_{0}$ and we deduce $g \in \Fix{V}{\Part[Q]}$.  Conversely, suppose $g \in \Fix{V}{\Part[Q]}$ and let $\AddrSet_{0} \cofsubset \AddrSet$ such that $\alpha^{g}$~is defined and $\alpha^{g} \approx \alpha$ for all $\alpha \in \AddrSet_{0}$.  There exist incomparable addresses $\beta_{1}$,~$\beta_{2}$, \dots,~$\beta_{k}$ in~$\AddrSet_{0}$ such that $\Cant$~is the disjoint union of the cones~$\cone{\beta_{i}}$.  Now, $\beta_{i}^{g} \approx \beta$ for each~$i$, so there exists $\Delta_{i} \cofsubset \AddrSet$ such that $(\beta_{i}^{g})\eta \sim \beta_{i}\eta$ for all $\eta \in \Delta_{i}$.  Set $\AddrSet_{1} = \set{\beta_{i}\eta}{\eta \in \Delta_{i}, \; 1 \leq i \leq k}$.  Then $\AddrSet_{1} \cofsubset \AddrSet$ and $\alpha^{g} \sim \alpha$ for all $\alpha \in \AddrSet_{1}$; that is, $g \in \Fix{V}{\Part}$.  This demonstrates that $\Fix{V}{\Part[Q]} = \Fix{V}{\Part}$.

Finally, let $g \in \Stab{V}{\Part}$.  There exists $\AddrSet_{2} \cofsubset \AddrSet$ such that $\alpha^{g}$~is defined for all $\alpha \in \AddrSet_{2}$ and, for $\alpha,\beta \in \AddrSet_{2}$, \ $\alpha \sim \beta$ if and only if $\alpha^{g} \sim \beta^{g}$.  By removing at most finitely many addresses, we may assume that $\alpha0,\alpha1 \in \AddrSet_{2}$ for all $\alpha \in \AddrSet_{2}$.  Let $\alpha,\beta \in \AddrSet_{2}$.  If $\alpha \approx \beta$ then $\alpha\eta \sim \beta\eta$ for all but finitely many words~$\eta$.  All such $\alpha\eta$~and~$\beta\eta$ are in~$\AddrSet_{2}$ by our assumption and hence $(\alpha\eta)^{g} \sim (\beta\eta)^{g}$ for all but finitely many~$\eta$; that is, $\alpha^{g}\eta \sim \beta^{g}\eta$.  Thus $\alpha^{g} \approx \beta^{g}$.

Conversely suppose $\alpha^{g} \approx \beta^{g}$.  Then $(\alpha\eta)^{g} = \alpha^{g}\eta \sim \beta^{g}\eta = (\beta\eta)^{g}$ for all but finitely many words~$\eta$.  Since $\alpha\eta,\beta\eta \in \AddrSet_{2}$, we deduce $\alpha\eta \sim \beta\eta$ for such~$\eta$.  Hence $\alpha \approx \beta$.  In conclusion, we have shown that $g \in \Stab{V}{\Part[Q]}$ and established the inclusion $\Stab{V}{\Part} \leq \Stab{V}{\Part[Q]}$.
\end{proof}

Lemma~\ref{lem:reduction} shows that we lose little by assuming that $\Part$~is a type system rather than working with some arbitrary equivalence relation.  Accordingly, we shall work with type systems from this point forward.  The following gives a characterization of when an element belongs to~$\Stab{V}{\Part}$ under this hypothesis.

\begin{lemma}
\label{lem:reducedStab}
Let $\Part$~be a type system on~$\AddrSet$ and $\sim$~be its associated equivalence relation.  Let $g \in V$ and let $\Gamma$ be any set of addresses with finite complement such that $\alpha^{g}$~is defined for all $\alpha \in \Gamma$.
Then $g \in \Stab{V}{\Part}$ if and only if, for all $\alpha, \beta \in \Gamma$, \ $\alpha \sim \beta$ if and only if $\alpha^{g} \sim \beta^{g}$.
\end{lemma}

\begin{proof}
Certainly if the condition holds then $g \in \Stab{V}{\Part}$ when we take $\AddrSet_{0} = \Gamma$ where $\AddrSet_0$ is as in Definition~\ref{def:basic}.
Conversely, suppose $g \in \Stab{V}{\Part}$.  Then there exists $\AddrSet_{0} \cofsubset \AddrSet$ such that $\alpha^{g}$~is defined for all $\alpha \in \AddrSet_{0}$ and, for $\alpha, \beta \in \AddrSet_{0}$, \ $\alpha \sim \beta$ if and only if $\alpha^g \sim \beta^g$.  As $\card{\AddrSet \setminus \AddrSet_{0}}$ is finite, there is a positive integer~$n$ such that $\delta \in \AddrSet_{0}$ whenever $\length{\delta} \geq n$.
  
Let $\alpha, \beta \in \Gamma$ and suppose $\alpha \sim \beta$.  If $\eta$~is any word of length~$n$, then $\alpha\eta, \beta\eta \in \AddrSet_{0}$ and $\alpha\eta \sim \beta\eta$ since $\sim$~is coherent.  Therefore $(\alpha^{g})\eta = (\alpha\eta)^{g} \sim (\beta\eta)^{g} = (\beta^{g})\eta$ for all such words~$\eta$.  Since $\sim$~is reduced, we now conclude $\alpha^{g} \sim \beta^{g}$.
A similar argument shows that if $\alpha^{g} \sim \beta^{g}$, then $\alpha \sim \beta$.
%Conversely, if $\alpha^{g} \sim \beta^{g}$, then, since $\sim$~is coherent, $(\alpha\gamma)^{g} = (\alpha^{g})\gamma \sim (\beta^{g})\gamma = (\beta\gamma)^{g}$ for all words~$\gamma$ of length~$n$.  Therefore $\alpha\gamma \sim \beta\gamma$ for all such~$\gamma$ and we conclude $\alpha \sim \beta$ using the fact that $\sim$~is reduced.
This establishes the required condition.
\end{proof}

\begin{example}\label{ex:stabilizerZero}Define a labelling of $\Omega$ by the alphabet $\{\mathsf{A},\mathsf{B}\}$ inductively as follows:
\begin{enumerate}
\item The empty word $\emptyword$ has label $\mathsf{A}$.
\item If $\alpha\in\Omega$ has label $\mathsf{A}$, then $\alpha 0$ has label $\mathsf{A}$, and $\alpha 1$ has label $\mathsf{B}$.
\item If $\alpha\in\Omega$ has label $\mathsf{B}$, then both $\alpha 0$ and $\alpha 1$ have label $\mathsf{B}$.
\end{enumerate}
Conditions (ii) and (iii) are summarized in Figure~\ref{fig:stabilizerZero}, which we refer to as the \textit{label diagram} for this labelling.  In general, a label diagram consists of one caret for each label and the labelling is completely determined by the label diagram together with the information that the empty word has label~$\mathsf{A}$.

\begin{figure}
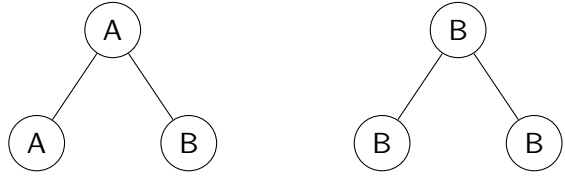

\begin{center}
  \labelcaret{A}{A}{B}
  \qquad\qquad
  \labelcaret{B}{B}{B}
\end{center}
  \caption{Label diagram for Example~\ref{ex:stabilizerZero}}
  \label{fig:stabilizerZero}
\end{figure}

This labelling defines a partition $\Part$ of $\AddrSet$ and this is an example of a type system.  In general, the equivalence relation determined by a label diagram is always coherent, and is reduced (and so defines a type system) as long as no two carets have the same labels on their leaves.

For this particular labelling, the addresses labeled~$\mathsf{A}$ are precisely the finite prefixes of the infinite word $\overline{0} = 000\cdots$.  Use of Lemma~\ref{lem:basic}\ref{i:Fix-form} and Lemma~\ref{lem:reducedStab} shows that the resulting labelling subgroups $\Fix{V}{\mathcal{P}}$ and $\Stab{V}{\mathcal{P}}$ are both equal to the stabilizer in~$V$ of the point $\overline{0} = 0000\cdots$.
\end{example}

\begin{example}\label{ex:stabilizerSet}
Let $S\subseteq \Cant$ and, for each $\alpha\in\Omega$, let
\[
S[\alpha] = \set{w\in \Cant}{\alpha w\in S}.
\]
Observe that $S[\alpha] = 0\bigl(S[\alpha0]\bigr) \cup 1\bigl(S[\alpha1]\bigr)$ for each address~$\alpha$ and hence the partition~$\Part$ determined by the equivalence relation~$\sim$ on~$\Omega$ defined by
\[
\alpha \sim \beta \IFF S[\alpha]=S[\beta]
\]
is a type system.  If $g \in V$ and $\alpha$~is any address such that $\alpha^{g}$~is defined, then one checks $S^{g}[\alpha^{g}] = S[\alpha]$.  From this one deduces that $\Fix{V}{\Part}$~equals the stabilizer in~$V$ of the subset~$S$.

For many sets~$S$ the subgroup $\Stab{V}{\Part}$ is equal to $\Fix{V}{\Part}$, but there are exceptions.  For example, if there are any elements of~$V$ that switch~$S$ with its complement (e.g., if $S$ is a clopen set), then these will lie in~$\Stab{V}{\Part}$ but not~$\Fix{V}{\Part}$.
\end{example}

\begin{example}\label{ex:SlopeFour}
Let $\Part$~be the type system determined by the label diagram in Figure~\ref{fig:SlopeFour}, where the empty word has label~$\mathsf{A}$.  Observe that an address $\alpha\in \Omega$ is labeled~$\mathsf{A}$ if and only if the length of~$\alpha$ is even. In this case,
\[
\Fix{V}{\Part} = \set{g\in V }{\text{$\length{\alpha^g} - \length{\alpha}$ is even for all $\alpha\in\AddrSet$ such that $\alpha^{g}$~is defined}}
\]
and
\[
\Stab{V}{\Part} = \Fix{V}{\Part} \cup \set{g\in V}{ \text{$\length{\alpha^g} - \length{\alpha}$ is odd for all $\alpha \in \AddrSet$ such that $\alpha^{g}$~is defined}}.
\]
In particular, $\Fix{V}{\Part}$ has index two in $\Stab{V}{\Part}$.

\begin{figure}
\begin{center}
  \labelcaret{A}{B}{B}
  \qquad\qquad
  \labelcaret{B}{A}{A}
\end{center}
  \caption{Label diagram for Example~\ref{ex:SlopeFour}}
  \label{fig:SlopeFour}
\end{figure}

\end{example}

\begin{defn}
\label{def:typeprops}
\mbox{}
\begin{enumerate}
\item A type system~$\Part[Q]$ (determined by the equivalence relation~$\approx$) is said to be a \emph{quotient} of the type system~$\Part$ (determined by the equivalence relation~$\sim$) if, for all $\alpha,\beta \in \AddrSet$, \ $\alpha \sim \beta$ implies $\alpha \approx \beta$.
\item The \emph{trivial relation} on~$\AddrSet$ is that given by equality: two addresses $\alpha$~and~$\beta$ are related under this relation only when $\alpha = \beta$.  The \emph{universal relation} is given by $\alpha \sim \beta$ for all $\alpha,\beta \in \AddrSet$.  The partitions determined by these two equivalence relations are both examples of type systems.
\item A type system~$\Part$ is said to be \emph{simple} if its only proper quotient is the one determined by the universal relation.
\item A type system~$\Part$ is said to be \emph{finite} if it consists of only finitely many equivalence classes on~$\AddrSet$.
\end{enumerate}
\end{defn}

The language of ``quotient'' here is appropriate since if $\Part[Q]$~is a quotient of the type system~$\Part$, then there is an associated surjective map $\Part \to \Part[Q]$.  Also observe that every type system is a quotient of the one determined by the trivial relation, while the type system for the universal relation is a quotient of every type system.  We now observe that for most type systems, $\Stab{V}{\Part}$ and $\Fix{V}{\Part}$ are in fact proper subgroups of~$V$.

\begin{lemma}
\label{lem:Stab-proper}
Let $\Part$~be a type system on~$\AddrSet$ with associated equivalence relation~$\sim$ that is neither the trivial nor the universal relation. Then $\Stab{V}{\Part} \neq V$.
\end{lemma}

\begin{proof}
Since $\sim$~is not the trivial relation, there are distinct addresses $\alpha$~and~$\beta$ such that $\alpha \sim \beta$.  If $\alpha$~were a prefix of~$\beta$, then there is a choice of letter~$x \in \letters$ such that $\alpha x \perp \beta x$.  Hence, by coherence, there is a pair of incomparable addresses related under~$\sim$.  If it were the case that all triples of incomparable addresses were related under~$\sim$, then any pair of addresses of length at least~$2$ would be related under~$\sim$ and, as it is reduced, we would conclude $\sim$~is the universal relation.  Hence there exist pairwise incomparable addresses $\alpha$,~$\beta$ and~$\gamma$ such that $\alpha \sim \beta \nsim \gamma$.  Then $\swap{\alpha}{\gamma} \notin \Stab{V}{\Part}$ by Lemma~\ref{lem:reducedStab} and so $\Stab{V}{\Part}$~is a proper subgroup of~$V$.
\end{proof}

We shall associate two additional structures to a type system.  This first is a graph which encodes how the type changes as one descends the binary rooted tree whose vertices are labelled by the addresses in~$\AddrSet$.  The second is a semigroup that will enable us to extend the notion of ``type'' of a cone to any clopen subset of~$\Cant$.  

\begin{defn}
\label{def:typegraph}
Let $\Part$~be a type system with $\sim$~its associated equivalence relation on~$\AddrSet$.  We associate a directed graph~$\graph{\Part}$, called the \emph{type graph}, as follows.
The vertices of~$\graph{\Part}$ are the types in~$\Part$.  If $P$~is a type, say with some address~$\alpha$ of type~$P$, then there is an edge labelled~$0$ from~$P$ to~$Q$ if the address~$\alpha0$ has type~$Q$.  Similarly there is an edge labelled~$1$ from~$P$ to~$R$ if the address~$\alpha1$ has type~$R$.
\end{defn}

Note that the fact that $\sim$~is coherent ensures that the definition of the edges in~$\graph{\Part}$ is independent of the choice of representative~$\alpha \in P$.  Hence this graph is well-defined.

\begin{notate}
Let $\Part$~be a type system and $P \in \Part$, say with some representative~$\alpha$.  We write~$P_{\gamma}$ for the type of the address~$\alpha\gamma$ for each $\gamma \in \AddrSet$.
\end{notate}

The coherence of~$\sim$ ensures this is well-defined independent of the choice of~$\alpha$.  Furthermore, if one starts at vertex~$P$ in the graph~$\graph{\Part}$ and follows the path labelled by the word~$\gamma$, then this path terminates at the vertex~$P_{\gamma}$.  In particular, $P_{0}$~and~$P_{1}$ are the two vertices adjacent to~$P$ and correspond to the types of the two descendants of~$\alpha$ in the binary rooted tree labelled by the addresses in~$\AddrSet$.

Observe that if $P$~is an infinite class in~$\Part$ then both $P_{0}$~and~$P_{1}$ are also infinite classes.  Hence there is a subgraph~$\Gamma^{\ast}$ of~$\graph{\Part}$ whose vertices are those members of~$\Part$ that are infinite classes.  We have observed in Lemma~\ref{lem:P*-action} above that $\Stab{V}{\Part}$~acts upon the set~$\infPart$ of these infinite classes and consequently it also acts, via automorphisms, on this subgraph~$\Gamma^{\ast}$.

\subsection{The semigroup associated to the equivalence relation}\label{subs:semigroup}

The semigroup~$\sgp{\Part}$ that we shall associate to a type system~$\Part$ will be a homomorphic image of the semigroup~$\sgp{\Cant}$ of continuous functions from~$\Cant$ to the set~$\Nat$ of natural numbers (endowed with the discrete topology) that are not identically zero.  If $U$~is a subset of~$\Cant$, we write~$\chi_{U}$ for the characteristic function of~$U$; that is, the function that takes the value~$1$ on~$U$ and the value~$0$ elsewhere.  To simplify notation, we use~$\chi_{\alpha}$, for $\alpha \in \AddrSet$, as an abbreviation for the characteristic function of the cone~$\cone{\alpha}$.  These characteristic functions generate~$\sgp{\Cant}$ and we shall exploit the following presentation for this semigroup.

\begin{lemma}
\label{lem:S(C)-pres}
The semigroup~$\sgp{\Cant}$ of non-zero continuous functions $\Cant \to \Nat$ is isomorphic to the additively-written commutative semigroup with presentation consisting of generators~$x_{\alpha}$ for each~$\alpha \in \AddrSet$ and relations $x_{\alpha} = x_{\alpha0} + x_{\alpha1}$ for each~$\alpha \in \AddrSet$.
\end{lemma}

\begin{proof}
Define $T$~be the commutative semigroup with presentation
\[
\presentation{ \text{$x_{\alpha}$ (for $\alpha \in \AddrSet$)} }{ \text{$x_{\alpha} = x_{\alpha0} + x_{\alpha1}$ for each $\alpha \in \AddrSet$}}.
\]
Since the characteristic functions~$\chi_{\alpha}$ satisfy the same relations, we observe that there is a surjective homomorphism $\phi \colon T \to \sgp{\Cant}$ given by $x_{\alpha} \mapsto \chi_{\alpha}$.  Suppose that $f$~and~$g$ are elements of~$T$ that satisfy $f\phi = g\phi$.  Express $f$~and~$g$ as sums of generators~$x_{\alpha}$ and let $N$~be the length of the longest address~$\alpha$ involved in these sums.  By repeated use of the relations $x_{\alpha} = x_{\alpha0} + x_{\alpha1}$, we can write $f = \sum_{\gamma \in C} c_{\gamma}x_{\gamma}$ and $g = \sum_{\gamma \in C} d_{\gamma} x_{\gamma}$ where $C$~is the set of addresses of length~$N$ and $c_{\gamma}$~and~$d_{\gamma}$ are non-negative integers.  The images of $f$~and~$g$ under~$\phi$ take the values $c_{\gamma}$~and~$d_{\gamma}$ respectively on the (disjoint) cones~$\cone{\gamma}$ and the hypothesis then implies that $c_{\gamma} = d_{\gamma}$ for each $\gamma \in C$.  It follows that $\phi$~is injective and hence $\sgp{\Cant} \cong T$, as claimed.
\end{proof}

Fix a type system~$\Part$ determined by some coherent and reduced equivalence relation~$\sim$ on~$\AddrSet$.  We shall associate a semigroup~$\sgp{\Part}$ with generators that correspond to the types in~$\Part$.  We reserve the lowercase letters $p$,~$q$ and~$r$ for the generators that correspond to the types in~$\Part$ denoted by $P$,~$Q$ and~$R$.  This will be extended to include any decoration that is applied to one these symbols.  Thus, if $P \in \Part$, then  $p_{\gamma}$~will denote the generator of~$\sgp{\Part}$ that corresponds to the type~$P_{\gamma}$, this being the type that is reached by following the path in the graph~$\graph{\Part}$ labelled by the word~$\gamma$.  In particular, $p_{0}$~and~$p_{1}$ will be the generators that correspond to the types $P_{0}$~and~$P_{1}$, these being the two vertices of~$\graph{\Part}$ that are adjacent to~$P$ (attached by edges with labels $0$~and~$1$, respectively).  Similarly, when we write $p^{(1)}$,~$p^{(2)}$, \dots,~$p^{(m)}$, these will be the generators of~$\sgp{\Part}$ that correspond to some family $P^{(1)}$,~$P^{(2)}$, \dots,~$P^{(m)}$ of types in~$\Part$.

\begin{defn}\label{defn:semigroup}
Suppose that $\Part$~is a type system on~$\AddrSet$.  Define~$\sgp{\Part}$ to be the commutative semigroup (whose operation we write additively) with generators~$p$ in one-one correspondence with the types~$P$ in~$\Part$ subject to all relations $p = p_{0} + p_{1}$ for each $P \in \Part$.
\end{defn}

\begin{rem}
In \cite{glu21}, the authors introduce the notion of an arboreal semigroup presentation, connecting it to a labelled planar forest diagram. The main difference between our semigroup and theirs is that we are taking our semigroup to be commutative and working with a single binary rooted tree instead of a more general forest. 
\end{rem}

Recall that the assumption of coherence ensures that the types $P_{0}$~and~$P_{1}$ are well-defined independent of the choice of representative for~$P$.  Hence the above definition for~$\sgp{\Part}$ makes sense in this context.  We denote the generators of~$\sgp{\Part}$ with a lowercase letter since it is common for the relations assumed to imply that $p = q$ for particular distinct types $P$~and~$Q$ from~$\Part$.  We are therefore able to distinguish between equality within this semigroup and genuine equality between types from~$\Part$.  As a first observation, and linking back to the graph~$\graph{\Part}$ defined above, if $P,Q \in \Part$ and there is a path from~$P$ to~$Q$ in~$\graph{\Part}$ then, upon applying the defining relation at each step of the path, we observe
\[
p = q + x
\]
for some $x \in \sgp{\Part}$ (and where, according to our convention, $p$~and~$q$ are the generators of~$\sgp{\Part}$ associated to $P$~and~$Q$).

\begin{defn}
\label{def:s-type}
By Lemma~\ref{lem:S(C)-pres}, there is a surjective homomorphism $\psi \colon \sgp{\Cant} \to \sgp{\Part}$ given by $\chi_{\alpha} \mapsto p$ whenever the address~$\alpha$ belongs has type~$P$ in~$\Part$.  If $U$~is a non-empty clopen subset of~$\Cant$, we define the \emph{s-type of~$U$ (under $\Part$)}, denoted~$\stypep{\Part}{U}$ (or $\stype{U}$ if $\Part$ is understood), to be the image of the characteristic function~$\chi_{U}$ under this homomorphism.
\end{defn}

The idea of this definition is to extend the concept of ``type'' from addresses (which correspond to cones) to arbitrary clopen subsets of~$\Cant$.  However, one should note that it is possible for two addresses $\alpha$~and~$\beta$ to have different types in~$\Part$ but that the defining relations result in the corresponding cones $\cone{\alpha}$~and~$\cone{\beta}$ having the same s-type.  (This reflects our earlier comment that we might have types satisfying $P \neq Q$ but the corresponding generators $p$~and~$q$ being equal in~$\sgp{\Part}$.)  Despite this possible collapse in s-types, this extension of ``type'' to clopen sets is useful and the following is the key property that we shall use:

\begin{lemma}
\label{lem:clopen-type}
Let $\Part$~be a type system with associated equivalence relation~$\sim$ and $\sgp{\Part}$~be the commutative semigroup it determines.  Suppose that $U$~and~$U'$ are non-empty clopen subsets of~$\Cant$ that have the same s-type.  Then there exist decompositions $U = \bigcup_{i=1}^{n} \cone{\gamma_{i}}$ and $U' = \bigcup_{i=1}^{n} \cone{\delta_{i}}$ into the same number of disjoint cones such that $\gamma_{i} \sim \delta_{i}$ for $i = 1$,~$2$, \dots,~$n$.
\end{lemma}

\begin{proof}
The s-type of a non-empty clopen subset~$U$ is its image of the characteristic function~$\chi_{U}$ under the homomorphism $\psi \colon \sgp{\Cant} \to \sgp{\Part}$ and so we seek to identify the kernel of~$\psi$ as a congruence on~$\sgp{\Cant}$.  Define a relation~$\rhorel$ on~$\sgp{\Cant}$ by $f \rhorel g$ if there exist decompositions $f = \sum_{i=1}^{m} \chi_{\alpha_{i}}$ and $g = \sum_{i=1}^{m} \chi_{\beta_{i}}$ where $\alpha_{i} \sim \beta_{i}$ for each~$i$.  This relation is certainly reflexive and symmetric.  Suppose that $f \rhorel g$ and $g \rhorel h$.  Then there are decompositions
\begin{equation}
f = \sum_{i=1}^{m} \chi_{\alpha_{i}}, \qquad 
g = \sum_{i=1}^{m} \chi_{\beta_{i}} = \sum_{j=1}^{n} \chi_{\gamma_{j}}, \qquad
h = \sum_{j=1}^{n} \chi_{\delta_{j}}
\label{eq:rho-decomp}
\end{equation}
such that $\alpha_{i} \sim \beta_{i}$ for each~$i$ and $\gamma_{j} \sim \delta_{j}$ for each~$j$.  Let $N$~be the maximum length of the addresses $\beta_{i}$~and~$\gamma_{i}$.  Repeatedly exploit the relation $\chi_{\alpha} = \chi_{\alpha0} + \chi_{\alpha1}$ to express each~$\chi_{\beta_{i}}$ as a sum of characteristic functions~$\chi_{\beta_{i}\zeta}$ where $\length{\zeta} = N - \length{\beta_{i}}$.  There is an analogous decomposition of~$\chi_{\alpha_{i}}$ as a sum of terms~$\chi_{\alpha_{i}\zeta}$ and here $\alpha_{i}\zeta \sim \beta_{i}\zeta$ since $\sim$~is coherent.  We can do the same thing for the terms corresponding to the relationship $g \rhorel h$.  Hence we can assume that in Equation~\eqref{eq:rho-decomp} that $\length{\beta_{i}} = \length{\gamma_{j}} = N$ for all $i$~and~$j$.  However, this means that any pair of addresses from the~$\beta_{i}$ and the~$\gamma_{j}$ are either equal or incomparable and so the two decompositions for the function~$g$ must be essentially the same.  Consequently $m = n$ in this case and we may relabel the addresses so that $\beta_{i} = \gamma_{i}$ for each~$i$.  It then follows that $\alpha_{i} \sim \delta_{i}$ and hence $f \rhorel h$.  We conclude that $\rhorel$~is an equivalence relation on~$\sgp{\Cant}$.  It is also a congruence on this semigroup since if $f \rhorel g$ then it immediately follows that $(f + \chi_{\gamma}) \rhorel (g + \chi_{\gamma})$ for any address~$\gamma$.

If $f \rhorel g$, then by definition $f\psi = g\psi$ in~$\sgp{\Part}$.  On the other hand, it follows from Lemma~\ref{lem:S(C)-pres} and the definition of~$\sgp{\Part}$ that the kernel of~$\psi$ is the congruence on~$\sgp{\Cant}$ generated by the condition $(\chi_{\alpha},\chi_{\beta}) \in \ker\psi$ whenever $\alpha \sim \beta$.  Since $\chi_{\alpha} \rhorel \chi_{\beta}$ when $\alpha \sim \beta$, it follows that $\ker\rho$~is contained in the congruence~$\rhorel$.  Therefore, the above congruence~$\rhorel$ is the kernel of~$\psi$.

Now consider the clopen subsets $U$~and~$U'$ as in the statement of the lemma.  As they have the same s-type, $\chi_{U} \rhorel \chi_{U'}$; that is, there are decompositions $\chi_{U} = \sum_{i=1}^{n} \chi_{\gamma_{i}}$ and $\chi_{U'} = \sum_{i=1}^{n} \chi_{\delta_{i}}$ such that $\gamma_{i} \sim \delta_{i}$ for each~$i$.  The definition of the characteristic function then implies that $U$~is the disjoint union of the cones~$\cone{\gamma_{i}}$ while $U'$~is the disjoint union of the cones~$\cone{\delta_{i}}$.  This completes the proof.
\end{proof}

This semigroup~$\sgp{\Part}$ will be used to produce elements of~$V$ that act (partially) on the set~$\AddrSet$ of addresses in some specific manner.  For example, Lemmas~\ref{lem:multinuke-sgp} and~\ref{lem:adjunct-sgp} below consider the structure of this semigroup under certain hypotheses on the type system~$\Part$ (with equivalence relation~$\sim$).  The last part of each of these two lemmas observes that, for appropriate addresses $\alpha$,~$\beta$, $\alpha'$ and~$\beta'$ with $\alpha \sim \alpha'$ and $\beta \sim \beta'$, we can show that $U = \Cant \setminus (\cone{\alpha} \cup \cone{\beta})$ and $U' = \Cant \setminus (\cone{\alpha'} \cup \cone{\beta'})$ have the same s-type.  We then may apply Lemma~\ref{lem:clopen-type} to $U$~and~$U'$ and hence construct two decompositions of~$\Cant$ into disjoint cones that can be paired such that corresponding cones have the same type.  With use of Lemma~\ref{lem:basic}\ref{i:Fix-form}, our conclusion is that there is an element~$g \in \Fix{V}{\Part}$ satisfying $\alpha^{g} = \alpha'$ and $\beta^{g} = \beta'$.

%%%%%%%%%%%%%%%%%%%%%%%%%%%%%%%%%%%%%%%%%%%%%%%%%%
\section{Simple type systems}
\label{sec:simpletypes}

In this section, we classify finite simple type systems.  Recall from Definition~\ref{def:typeprops} that a simple type system is one whose only proper quotient consists of a single type and that a type system is finite when it has finitely many equivalence classes.  

Our first task is to develop useful language to describe type systems and subsets of type systems.

\begin{defn}
Let $\Part$~be a type system on~$\AddrSet$ and let $\Part[S]$~be a subset of~$\Part$.
\begin{enumerate}
\item We say that $\Part[S]$~is \emph{child-closed} if whenever $P \in \Part[S]$ then also $P_{0}$~and~$P_{1}$ belong to~$\Part[S]$.
\item If $P\in  \Part[S]$ and $\gamma$~is a  word, we say \emph{$\gamma$~determines a path through~$\Part[S]$ from~$P$} if $P_{\zeta} \in \Part[S]$ for all prefixes~$\zeta$ of~$\gamma$.
\item We say that $\Part[S]$~is \emph{strongly connected} if for all $P,Q \in \Part[S]$ there exists a non-empty word~$\gamma$ so that $\gamma$ determines a path through $\Part[S]$ with $P_\gamma=Q$.
\item We say that $S$~is a \emph{nucleus} if $S$~is child-closed and strongly connected.
\end{enumerate}
\end{defn}

We begin by introducing a number of special categories of type systems that will arise within the final classification of finite simple type systems. 

\begin{defn}
\label{def:nuclear}
A type system~$\Part$ on~$\Omega$ is called \emph{nuclear} if there exists a positive integer~$t$ and a subset~$\Part[Q]$ of~$\Part$ such that:
\begin{enumerate}
\renewcommand{\labelenumi}{(\alph{enumi})}
\item if $\alpha$~is an address with $\length{\alpha} \geq t$, then the type of~$\alpha$ is a member of~$\Part[Q]$;
\item the subset $\Part[Q]$~is a nucleus (that is, child-closed and strongly connected).
\end{enumerate}
\end{defn}

The following are variations on the theme of a nuclear type system.  Indeed, a nuclear type system is a special case of a multinuclear type system, but with precisely one nucleus.

\begin{defn}
\label{def:nuketypes}
Let $\Part$~be a type system on~$\AddrSet$.
\begin{enumerate}
\item \label{i:multinuke}
We shall say that $\Part$~is \emph{multinuclear} if there exist a positive integer~$t$ and disjoint subsets $\Part[Q]^{(1)}$,~$\Part[Q]^{(2)}$, \dots,~$\Part[Q]^{(k)}$ of~$\Part$ such that
\begin{enumerate}
\item if $\alpha \in \AddrSet$ with $\length{\alpha} \geq t$ then the type of~$\alpha$ is a member of one of the~$\Part[Q]^{(i)}$; and
\item for each index $i$, the subset  $\Part[Q]^{(i)}$ is a nucleus.
\end{enumerate}
The subsets~$\Part[Q]^{(i)}$ will be called the \emph{nucleii} of the type system.

\item We say that $\Part$~is \emph{binuclear} if it is multinuclear with precisely two nucleii $\Part[Q]^{(1)}$~and~$\Part[Q]^{(2)}$.

\item \label{i:quasi}
We say that $\Part$~is \emph{quasinuclear} if there exist a positive integer~$t$, and non-empty disjoint subsets~$\Part[Q],\Part[R] \subseteq \Part$  such that
\begin{enumerate}[ref=(\alph*)]
\item \label{i:q-depth}
if $\alpha \in \AddrSet$ with $\length{\alpha} \geq t$ then the type of~$\alpha$ is in~$\Part[Q]\cup \Part[R]$;
\item \label{i:q-nucleus}
the set $\Part[R]$ is a nucleus;
\item \label{i:q-strong}
the set $\Part[Q]$ is strongly connected;
\item \label{i:q-child}
the set $\Part[R]\cup\Part[Q]$ is child-closed; and
\item \label{i:q-reach}
for every $Q\in \Part[Q]$ and $R\in \Part[R]$, there is a word $\gamma$ so that $Q_\gamma=R$.
\end{enumerate}
Given such a setup, we define $\branchPartQ$~to be the set of equivalence classes~$Q \in \Part[Q]$ such that $Q_{0}$~and~$Q_{1}$ are both in~$\Part[Q]$.  If $\branchPartQ$~is non-empty, then we say that $\Part$~is \emph{branching quasinuclear}.  If $\Part$ is quasinuclear, we refer to the set $\Part[Q]\cup\Part[R]$ as the \emph{terminal set of $\Part$} and to~$\Part[R]$ as the \emph{nucleus of~$\Part$}.
\end{enumerate}
\end{defn}

The terms appearing in Definitions~\ref{def:nuclear} and~\ref{def:nuketypes} appear in our classification of finite simple type systems.  In order to aid our proof of this classification, we introduce some additional terminology and establish the lemmas that we need. 

\begin{defn}
Let $\Part$~be a type system on~$\AddrSet$ and let $\Part[S]$~be a subset of~$\Part$.
\begin{enumerate}
\item We say that $\Part[S]$~contains a \emph{cycle} if there exists $P \in \Part[S]$ and a non-empty word~$\gamma$ that determines a path through~$\Part[S]$ at $P$ with $P_{\gamma} = P$.  Such a path will be called a \emph{cycle (in~$\Part[S]$) based at~$P$}.
\item We say that $\Part[S]$~is \emph{stable} if for every $Q \in \Part[S]$, there exists some $P \in \Part[S]$ such that there is a cycle in~$\Part[S]$ based at $P$ and also a word~$\gamma$ that determines a path through~$\Part[S]$ from~$P$ with $P_{\gamma} = Q$.
\end{enumerate}
\end{defn}

Strongly connected subsets are necessarily stable, but stable subsets are more general and help us to encode the nature of binuclear and quasinuclear type systems.  For example, with the notation of Definition~\ref{def:nuketypes}\ref{i:quasi}, if $\Part$~is quasinuclear then the set~$\Part[Q] \cup \Part[R]$ is stable but not strongly connected.

The following observation is immediate from the definitions.

\begin{lemma}
\label{lem:cycle-desc}
Let $\Part$~be a type system on~$\AddrSet$ and suppose there is a cycle in $\Part$ based at some~$P$ in~$\Part$. Then $\Part[Q] = \set{P_{\zeta}}{\zeta \in \AddrSet}$ is child-closed and stable.
\end{lemma}

This set~$\Part[Q]$ above is the set of vertices in the graph~$\graph{\Part}$ that one can reach by following a path starting at~$P$.

\begin{lemma}
\label{lem:CycleInComplement}
Let $\Part$~be a simple type system on~$\AddrSet$.  If $\Part[S]$~is a non-empty subset of~$\Part$ that is child-closed and such that the complement~$\Part \setminus \Part[S]$ contains a cycle, then $\card{\Part[S]} = 1$.
\end{lemma}

\begin{proof}
Write $\sim$~for the equivalence relation on~$\AddrSet$ that determines~$\Part$.  Let $P \in \Part \setminus \Part[S]$ such that there is a cycle in $\Part \setminus \Part[S]$ based at~$P$.  We define a new relation~$\approx$ on~$\AddrSet$ by the rule $\alpha \approx \beta$ when
\begin{enumerate}
\item[$(\ast)$] for all but finitely many words~$\eta$,  either $\alpha\eta \sim \beta\eta$ or the $\Part$\nbd types of~$\alpha\eta$ and of~$\beta\eta$ both belong to~$\Part[S]$.
\end{enumerate}

We shall show that $\approx$~defines a type system on~$\AddrSet$ that is a quotient of~$\Part$.  It is immediate that this relation is reflexive and symmetric.  Suppose $\alpha \approx \beta$ and $\beta \approx \gamma$.  For all but finitely many words~$\eta$, the condition appearing in~$(\ast)$ applies to the pair $\alpha\eta$~and~$\beta\eta$ and to the pair $\beta\eta$~and~$\gamma\eta$.  If $\alpha\eta \sim \beta\eta \sim \gamma\eta$, then $\alpha\eta \sim \gamma\eta$ since $\sim$~is an equivalence relation.  Otherwise, it must be the case that the $\Part$\nbd types of all three of $\alpha\eta$,~$\beta\eta$ and~$\gamma\eta$ belong to~$\Part[S]$ (with two of these types coinciding if $\alpha\eta \sim \beta\eta$ or $\beta\eta \sim \gamma\eta$).  We conclude that $\alpha \approx \gamma$.  This shows that $\approx$~is an equivalence relation on~$\AddrSet$.  It is also straightforward to verify that $(\ast)$~defines a relation that is coherent and reduced.  Let $\Part'$~be the type system on~$\AddrSet$ determined by~$\approx$.  This is a quotient of~$\Part$ in view of the first part of the condition appearing in~$(\ast)$.

Suppose that $\approx$~is the universal relation on~$\Omega$.  Recall that there is a cycle in~$\Part \setminus \Part[S]$ based at~$P$.  Let $\alpha$~be an address of $\Part$\nbd type~$P$ and let $\gamma$~be a non-empty word that determines a cycle in~$\Part \setminus \Part[S]$ based at $P$.  Let $\beta$~be an address with $\Part$\nbd type in~$\Part[S]$.  By assumption, $\alpha \approx \beta$.  Hence, for sufficiently large~$k$, the condition in~$(\ast)$ holds when $\eta = \gamma^{k}$ and so either $\alpha\eta \sim \beta\eta$ or the $\Part$\nbd types of $\alpha\eta$~and~$\beta\eta$ both belong to~$\Part[S]$.  However $\alpha\eta$~has $\Part$\nbd type~$P$, which is not in~$\Part[S]$, since $P_{\gamma} = P$, while $\beta\eta$~has $\Part$\nbd type in~$\Part[S]$ as $\Part[S]$~is child-closed.  This is a contradiction.

Consequently, by simplicity of the type system~$\Part$, it must be the case that $\Part' = \Part$.  Now if $\alpha$~and~$\beta$ have $\Part$\nbd types in~$\Part[S]$, then the $\Part$\nbd types of $\alpha\eta$~and~$\beta\eta$ are in~$\Part[S]$ for all words~$\eta$ and hence $\alpha \approx \beta$.  It follows therefore that $\alpha \sim \beta$ and we conclude that $\card{\Part[S]} = 1$.
\end{proof}

The majority of the remaining work to classify the finite simple type systems is contained in the following result:

\begin{lemma}
\label{lem:stable+childclosed}
Let $\Part$~be a simple type system on~$\AddrSet$ which admits a cycle.  Then $\Part$~contains precisely one, two or three stable child-closed subsets.  Moreover,
\begin{enumerate}
\item if $\Part$~contains exactly two stable child-closed subsets then one has cardinality~$1$ and is a subset of the other;
\item if $\Part$~contains three stable child-closed subsets then two of them have cardinality~$1$ and the third is the union of the first two.
\end{enumerate}
\end{lemma}

\begin{proof}
Let $\Part$~be a simple type system that contains some cycle.  By Lemma~\ref{lem:cycle-desc}, there is at least one stable child-closed subset of~$\Part$.

Suppose first that $\Part[Q]$~and~$\Part[Q]'$ are two distinct stable child-closed subsets of~$\Part$.  We may assume that there exists $Q \in \Part[Q]' \setminus \Part[Q]$.  Then by stability there is some $P \in \Part[Q]'$  with a cycle~$C$ in~$\Part[Q]'$ based at $P$ and a (non-empty) word~$\gamma$ such that $\gamma$~determines a path through~$\Part[Q]'$ from~$P$ to~$Q$.  Since $\Part[Q]$~is child-closed, necessarily $P \notin \Part[Q]$ and the cycle $C$ based at~$P$ is disjoint from~$\Part[Q]$.  Hence we deduce that $\card{\Part[Q]} = 1$ by Lemma~\ref{lem:CycleInComplement}.  If it is not the case that $\Part[Q] \subseteq \Part[Q]'$, then there is some member of~$\Part[Q] \setminus \Part[Q]'$ and the same argument shows $\card{\Part[Q]'} = 1$.  As a consequence $\Part[Q]$~and~$\Part[Q]'$ are then disjoint and both singleton sets.  Note in this case that the union $\Part[Q] \cup \Part[Q]'$ is also child-closed and stable.

We have shown that if $\Part[Q]$~and~$\Part[Q']$ are distinct stable child-closed sets then either
\begin{enumerate}
\item[(1)] $\Part[Q] \subseteq \Part[Q]'$ and $\card{\Part[Q]} = 1$ (or \viceversa), or
\item[(2)] $\card{\Part[Q]} = \card{\Part[Q]'} = 1$ and $\Part[Q] \cup \Part[Q]'$~is also child-closed and stable.
\end{enumerate}
In particular, if there are exactly two stable child-closed subsets then their configuration is as described in part~(i) of the statement.

We now consider each of the above possibilities to consider how the stable child-closed subsets could relate to a third such subset~$\Part[Q]''$.  Suppose first that $\Part[Q] \subseteq \Part[Q]'$ so that $\card{\Part[Q]} = 1$ as in~(1) above.  Then since $\card{\Part[Q]'} > 1$ it must be the case that $\Part[Q]'' \subseteq \Part[Q]'$ also and $\card{\Part[Q]''} = 1$.  Then $\Part[Q] \cup \Part[Q]''$~is stable and child-closed, so since both $\Part[Q] \cup \Part[Q]''$ and~$\Part[Q]'$ contain more than one type it must be the case that $\Part[Q'] = \Part[Q] \cup \Part[Q]''$.  In particular, $\Part[Q]'' = \Part[Q]' \setminus \Part[Q]$.  Alternatively if $\Part[Q]$~and~$\Part[Q]'$ are as in~(2), then applying the above observation to the sets $\Part[Q] \cup \Part[Q]'$~and~$\Part[Q]''$ shows that the only possibility is that $\Part[Q]'' = \Part[Q] \cup \Part[Q]'$.  In conclusion, since in either case there is only one possibility for~$\Part[Q]''$, we deduce that $\Part$~cannot contain more than three stable child-closed subsets and, moreover, if there are three then they are arranged as specified in~(ii) in the statement.
\end{proof}

We give specialised names for the categories of type systems that arise as finite simple type systems.  Firstly, if a binuclear (or multinuclear) type system has the property that the cardinality of each of its nucleii is~$1$, then we say the type system is \emph{atomic binuclear} (or \emph{atomic multinuclear}, respectively).  Similarly, a quasinuclear type system such that its nucleus contains a single type will be called \emph{atomic quasinuclear}.

We can now state our main result of this section.

\FiniteTypes

\begin{proof}
Let $\Part$~be a finite simple type system on~$\AddrSet$.  The finiteness ensures that there must exist some cycle in~$\Part$.  We shall show that each of the options from Lemma~\ref{lem:stable+childclosed} correspond to the specified kinds of type systems.

Suppose first that $\Part$~contains precisely one stable child-closed subset~$\Part[Q]$.  Then, by Lemma~\ref{lem:cycle-desc}, every cycle in the graph~$\graph{\Part}$ must be contained within~$\Part[Q]$.  As $\Part$~is finite, there is a positive integer~$t$ such that after following a path of length~$t$ one has crossed a cycle.  Since $\Part[Q]$~is child-closed, it follows that if $\length{\alpha} \geq t$ then the type of~$\alpha$ belongs to~$\Part[Q]$.  Child-closure also ensures that if $P \in \Part[Q]$, then both $P_{0}$~and~$P_{1}$ belong to~$\Part[Q]$.  Now let $P,Q \in \Part[Q]$.  After following a path of sufficient length from~$P$, we reach a cycle and then, by Lemma~\ref{lem:cycle-desc}, the descendants of some point on the cycle is a stable child-closed subset and so equals~$\Part[Q]$.  In particular, there is a word~$\gamma$ such that $P_{\gamma} = Q$.  Hence $\Part$~is nuclear with $\Part[Q]$~as its nucleus.

Second consider the case when $\Part$~contains precisely two stable child-closed subsets, say $\Part[R]$~and~$\Part[S]$, where $\Part[R] \subseteq \Part[S]$ and $\card{\Part[R]} = 1$.  Let $R$~denote the type belonging to~$\Part[R]$ and set $\Part[Q] = \Part[S] \setminus \Part[R]$.  We shall show that $\Part$~is quasinuclear; that is, the conditions of Definition~\ref{def:nuketypes}\ref{i:quasi} hold.  As $\Part$~is finite, there is a positive integer~$t$ such that after following a path of length~$t$ one has crossed a cycle in~$\Part$.  Consequently, with use of Lemma~\ref{lem:cycle-desc}, the type of~$\alpha$ belongs to~$\Part[S]$ when $\length{\alpha} \geq t$, so Condition~\ref{i:q-depth} holds.  Conditions~\ref{i:q-nucleus} and~\ref{i:q-child} hold immediately since $\Part[R]$~and~$\Part[S]$ are both child-closed and $\card{\Part[R]} = 1$.  It remains to show that Conditions \ref{i:q-strong}~and~\ref{i:q-reach} hold.

Fix $P \in \Part[Q]$; that is, $P \in \Part[S]$ but $P \neq R$.  Consider the set $\Part[U] = \set{P_{\gamma}}{\gamma \in \AddrSet}$, which consists of the vertices in~$\graph{\Part}$ that one can reach by following a path beginning at~$P$.  Certainly $\Part[U]$~is child-closed.  If it were the case that $\Part[U] = \{P\}$, then $\Part[U]$~would also be stable but this would contradict the assumption that $\Part[R]$~and~$\Part[S]$ are the only stable child-closed subsets as both these contain~$R$.  Hence $\card{\Part[U]} > 1$.  Now as $\Part[S]$~is stable, there exists $P' \in \Part[S]$ such that there is a path in~$\graph{\Part}$ from~$P'$ to~$P$ and there is some cycle based at~$P'$.  If this cycle were not also contained in~$\Part[U]$, then no vertex on the cycle belongs to~$\Part[U]$ (as $\Part[U]$~is child-closed) and then Lemma~\ref{lem:CycleInComplement} tells us $\card{\Part[U]} = 1$, contrary to what we have established.  Hence the cycle, and in particular~$P'$, is contained in~$\Part[U]$.  It follows that $\Part[U]$~coincides with the set of all~$P'_{\gamma}$ for $\gamma \in \AddrSet$ and so $\Part[U]$~is stable by Lemma~\ref{lem:cycle-desc}.  Therefore $\Part[U] = \Part[S]$.  In particular, there is a path in~$\graph{\Part}$ from~$P$ to every vertex in~$\Part[Q]$ and to the vertex~$R$.  This shows that $\Part$~is indeed quasinuclear.

Finally if $\Part$~contains three stable child-closed subsets, then they have the form $\Part[Q]$,~$\Part[Q]'$ and $\Part[Q] \cup \Part[Q]'$ with $\card{\Part[Q]} = \card{\Part[Q]'} = 1$.  Again, by finiteness of~$\Part$, there is a positive integer~$t$ such that after following a path of length~$t$ one has crossed a cycle in~$\Part$.  Necessarily then the type of~$\alpha$ belongs to $\Part[Q]$ or~$\Part[Q]'$ when $\length{\alpha} \geq t$ and now we see that the conditions for~$\Part$ to be binuclear as in Definition~\ref{def:nuketypes} are satisfied.  This completes the classification of finite simple type systems.
\end{proof}

\section{Maximality of stabilizer subgroups}
\label{sec:stabismax}

We now turn our attention to showing that, for simple type systems which are nuclear, multinuclear, or atomic quasinuclear, the corresponding stabilizer is a maximal subgroup of Thompson's group~$V$ (see Theorems~\ref{thm:branchtype} and~\ref{thm:quasicycle}).  In particular, this includes all finite simple type systems as described in Theorem~\ref{thm:finitetypes} (see Theorem~\ref{thm:finitesimple}).

Simple type systems~$\Part$ that are atomic quasinuclear but not branching are dealt with in Theorem~\ref{thm:quasicycle} where we observe that $\Stab{V}{\Part}$~equals the stabilizer of a finite set of points in~$\Cant$ of the same tail class.  The majority of this section is therefore concerned with the multinuclear case and the atomic branching quasinuclear case.  We analyze these two cases in parallel establishing information that will enable us to prove Theorem~\ref{thm:branchtype}.  The essential strategy is as follows: If there is a subgroup~$H$ of~$V$ that strictly contains~$\Stab{V}{\Part}$ then we will use Proposition~\ref{prop:typequotient} to produce a type system that is a quotient of~$\Part$ while Lemma~\ref{lem:notStab} will be used to show that this is a \emph{proper} quotient.  Since $\Part$~is simple, we will use the definition of this type system to deduce that $H$~contains all transpositions in~$V$ and hence $H = V$ completing the proof.

We begin by establishing various facts about the type systems under consideration, starting with information about the associated semigroup~$\sgp{\Part}$ as defined in Definition~\ref{defn:semigroup}.  (See also Definition~\ref{def:s-type} for the meaning of s-type.)  In particular, Lemma~\ref{lem:multinuke-sgp} provides results about multinuclear type systems, while Lemma~\ref{lem:adjunct-sgp} gives analogous results about quasinuclear type systems that are both atomic and branching.  (Recall that a quasinuclear type system is atomic when its nucleus~$\Part[R]$ contains a single type.)  Note that nuclear and binuclear type systems are special cases of multinuclear type systems, so facts about multinuclear type systems apply to them.

\begin{lemma}
\label{lem:multinuke-sgp}
Let $\Part$~be a multinuclear type system on~$\AddrSet$ with $\Part[Q]^{(1)}$,~$\Part[Q]^{(2)}$, \dots,~$\Part[Q]^{(k)}$ its nucleii and $t$~a positive integer such that an address~$\alpha$ has type in some~$\Part[Q]^{(i)}$ when $\length{\alpha} \geq t$.  Then
\begin{enumerate}
\item \label{i:multinuke-gps}
each subsemigroup $G_{i} = \langle \, q \mid Q \in \Part[Q]^{(i)} \, \rangle$ of~$\sgp{\Part}$ has the structure of a group;
\item \label{i:nuke-group}
if $\Part$~is nuclear (that is, when $k = 1$), then $\sgp{\Part}$~is a group;
\item \label{i:multinuke-compltype}
if $\alpha$,~$\alpha'$, $\beta$ and~$\beta'$ are addresses of length at least~$t+2$ such that $\alpha \perp \beta$, \ $\alpha' \perp \beta'$, \ $\alpha \sim \alpha'$ and $\beta \sim \beta'$, then the complements $\Cant 
\setminus (\cone{\alpha} \cup \cone{\beta})$ and $\Cant \setminus (\cone{\alpha'} \cup \cone{\beta'})$ have the same s-type.
\end{enumerate}
\end{lemma}

\begin{proof}
\ref{i:multinuke-gps}~Fix an index~$i$, write $G_{i}$~for the subsemigroup generated by those generators corresponding to the types in~$\Part[Q]^{(i)}$ and also fix some $Q \in \Part[Q]^{(i)}$.  First follow the edge (with label~$0$) in the type graph~$\graph{\Part}$ from~$Q$ to~$Q_{0}$ and then, since the nucleus~$\Part[Q]^{(i)}$ is strongly connected, follow a path from~$Q_{0}$ back to~$Q$.  Note that, for every vertex on the path, both edges from it lead to vertices in~$\Part[Q]^{(i)}$ as this nucleus is child-closed.  Hence the generator~$q$ corresponding to~$Q$ satisfies $q = q + e$ for some $e \in G_{i}$.  Now if $P$~is any member of~$\Part[Q]^{(i)}$, there is a path from~$P$ to~$Q$ and this also only passes through vertices in~$\Part[Q]^{(i)}$.  Hence $p = q + x$ for some $x \in G_{i}$.  Then $p + e = x + q + e = x + q = p$.  This shows that $e$~is an identity element for the subsemigroup~$G_{i}$.

Let $R$~be a type in~$\Part[Q]^{(i)}$ such that the corresponding~$r$ is a summand of~$e$ when expressed as a sum from the generating set for~$G_{i}$.  Then $e = r + y$ for some $y \in G_{i}$.  If $P \in \Part[Q]^{(i)}$, there is a path from~$R$ to~$P$ and hence $r = p + z$ for some $z \in G_{i}$.  Now $e = p + z + y$ and we deduce that $p$~has an inverse in~$G_{i}$.  Since all its generators are invertible, we conclude that $G_{i}$~is a group.

\ref{i:nuke-group}~If $\Part$~is nuclear, say with $\Part[Q]$~as nucleus, set $G = \langle \, q \mid Q \in \Part[Q] \, \rangle$.  Now if $P \in \Part$, then $p = \sum p_{\gamma} \in G$ where we sum over all words~$\gamma$ of length~$t$.  Hence $\sgp{P} = G$ and this is a group by~\ref{i:multinuke-gps}.

\ref{i:multinuke-compltype}~After following a path of length~$t$ in the graph~$\graph{\Part}$, one arrives at a vertex in some nucleus~$\Part[Q]^{(i)}$ and then from that point all subsequent edges lead to a vertex in the same subset.  Hence, for each~$i$, there is a positive integer~$m_{i}$ and words $\gamma_{1}^{(i)}$,~$\gamma_{2}^{(i)}$, \dots,~$\gamma_{m_{i}}^{(i)}$ that are minimal amongst those that have type in~$\Part[Q]^{(i)}$.  These words have the properties that any distinct pair of words $\gamma_{j}^{(i)}$~and~$\gamma_{j'}^{(i')}$ are incomparable, all have length~$\leq t$, and that any word of length~$t$ or more has exactly one~$\gamma_{j}^{(i)}$ as a prefix.  Observe that
\[
\Cant = \bigcup_{i=1}^{k} \bigcup_{j=1}^{m_{i}} \cone{\gamma_{j}^{(i)}}
\]
expresses~$\Cant$ as a disjoint union of cones with the~$\gamma_{j}^{(i)}$ as their addresses.  Write $C_{i} = \bigcup_{j=1}^{m_{i}} \cone{\gamma_{j}^{(i)}}$ for~$i = 1$, $2$, \dots,~$k$.

Let $P$~and~$Q$ be the types of $\alpha$~and~$\beta$, respectively, and let $x$~and~$y$ be the respective s-types of the complements $U = \Cant \setminus ( \cone{\alpha} \cup \cone{\beta} )$ and $U' = \Cant \setminus ( \cone{\alpha'} \cup \cone{\beta'} )$.  We shall show that $x = y$ by considering possibilities for $P$~and~$Q$.  Note by assumption that $\alpha' \in P$ and $\beta' \in Q$.

\paragraph{Case~1:} $P$~and~$Q$ belong to the same nucleus~$\Part[Q]^{(i)}$.

As $\alpha$~and~$\beta$ have length~$\geq t+2$, at most two of the words $\gamma_{1}^{(i)}00$,~$\gamma_{1}^{(i)}01$, $\gamma_{1}^{(i)}10$ and~$\gamma_{1}^{(i)}11$ can occur as prefixes of either $\alpha$~or~$\beta$ and the other two are incomparable to both $\alpha$~and~$\beta$.  As a consequence, $C_{i} \setminus ( \cone{\alpha} \cup \cone{\beta} )$ is a union of cones indexed by addresses with some~$\gamma_{j}^{(i)}$, for $1 \leq j \leq m_{i}$, as prefix and this union involves at least some such cones.  Therefore $\stype{ C_{i} \setminus ( \cone{\alpha} \cup \cone{\beta} ) } = z$ is some element of the group~$G_{i}$.  Similarly $\stype{ C_{i} \setminus ( \cone{\alpha'} \cup \cone{\beta'} ) } = z' \in G_{i}$.  Now the equation $z + p + q = z' + p + q$ holds in the group~$G_{i}$, as both sides equal the s-type of~$C_{i}$, and so we deduce $z = z'$.  Since
\[
U = C_{1} \cup C_{2} \cup \dots \cup C_{i-1} \cup \bigl( C_{i} \setminus ( \cone{\alpha} \cup \cone{\beta} ) \bigr) \cup C_{i+1} \cup \dots \cup C_{k},
\]
it follows that $x = \sum_{\ell \neq i} \stype{C_{\ell}} + z$ and similarly $y = \sum_{\ell \neq i} \stype{C_{\ell}} + z'$.  We deduce that $x = y$ in this case.

\paragraph{Case~2:} $P \in \Part[Q]^{(i)}$ and $Q \in \Part[Q]^{(i')}$ with $i \neq i'$.

A similar argument to that used in Case~1 shows that $\stype{ C_{i} \setminus \cone{\alpha} } = \stype{ C_{i} \setminus \cone{\alpha'} }$ and $\stype{ C_{i'} \setminus \cone{\beta} } = \stype{ C_{i'} \setminus \cone{\beta'} }$.  We then deduce that $x = y$ since they are, respectively, the sums of the left- and right-hand sides of these two equations together with terms of the form~$\stype{C_{\ell}}$ for $\ell \neq i,i'$.  This establishes the final part of the lemma.
\end{proof}

If $\Part$~is a multinuclear type system on~$\Omega$, the associated semigroup~$\sgp{\Part}$ is generated by the groups~$G_{i}$ listed in Lemma~\ref{lem:multinuke-sgp} but it need not be the direct product of these groups.  The following simple example illustrates this and it will be important to bear this in mind when we use this semigroup in what follows.

\begin{example}
Let $\Part$~be the type system determined by the label diagram in Figure~\ref{fig:NotProduct} where $\textsf{A}$~is the type of the empty word.  This type system is multinuclear with two nuclei $\{\mathsf{P}\}$~and~$\{\mathsf{Q}\}$.  The semigroup~$\sgp{\Part}$ associated with~$\Part$ is the (additively-written) commutative semigroup generated by two elements $p$~and~$q$ subject to the additional relations $p = 2p$ and $q = 2q$.  Hence $\sgp{\Part} = \{p,q,p+q\}$ where all three elements are idempotents and the element~$p+q$ is a zero (in the semigroup sense); that is, $(p+q)+x = p+q$ for all $x \in \sgp{\Part}$.  The two groups, as described in Lemma~\ref{lem:multinuke-sgp}, associated to the nuclei are $G_{1} = \{p\}$ and $G_{2} = \{q\}$, which are both isomorphic to the trivial group.  For this type system, if $U$~is a non-empty clopen subset of~$\Cant$, then it has s-type equal to~$p$ when it is contained in the cone~$\cone{0}$, equal to~$q$ when it is contained in the cone~$\cone{1}$, and equal to~$p+q$ when it has non-empty intersection with both of these cones.
\begin{figure}
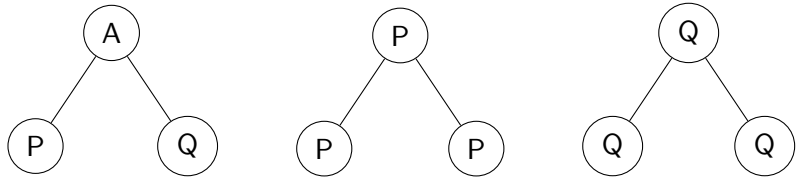

\begin{center}
\labelcaret{A}{P}{Q}
\qquad
\labelcaret{P}{P}{P}
\qquad
\labelcaret{Q}{Q}{Q}
\end{center}
\caption{Label diagram for a particular multinuclear type system}
\label{fig:NotProduct}
\end{figure}
\end{example}

\begin{lemma}
\label{lem:adjunct-sgp}
Let $\Part$~be a quasinuclear type system on~$\AddrSet$ that is both branching and atomic.  Let $t$,~$\Part[Q]$ and~$\Part[R]$ be as in Definition~\ref{def:nuketypes}\ref{i:quasi} and let $R$~be the unique type in~$\Part[R]$.  Then
\begin{enumerate}
\item \label{i:adjunct-id}
the element~$r$ (corresponding to the type~$R$) is an identity element for~$\sgp{\Part}$;
\item \label{i:adjunct-sgpgen}
$\sgp{\Part}$~is generated by~$r$ and the elements~$q$ for $Q \in \branchPartQ$;
\item \label{i:adjunct-ideal}
if $\Ideal$~is the subsemigroup generated by the~$q$ for $Q \in \branchPartQ$, then $I$~is an ideal of~$\sgp{\Part}$ and has the structure of a group.
\end{enumerate}
Moreover, there exists an integer~$N \geq 2$ such that
\begin{enumerate}
\setcounter{enumi}{3}
\item \label{i:adjunct-compltype1}
if $\alpha$~and~$\beta$ are incomparable addresses with $\length{\alpha}, \length{\beta} \geq N$, then $\stype{ \Cant \setminus ( \cone{\alpha} \cup \cone{\beta} ) } \in \Ideal$;
\item \label{i:adjunct-perpinQdag}
if $\alpha$~and~$\alpha'$ are addresses of length at least~$N$, then there exists an address~$\beta$ incomparable with both $\alpha$~and~$\alpha'$ such that $\length{\beta} \geq N$ and $\beta$~has type in~$\branchPartQ$;
\item \label{i:adjunct-compltype2}
if $\alpha$,~$\alpha'$, $\beta$ and~$\beta'$ are addresses of length~$\geq N$ such that $\alpha \perp \beta$, \ $\alpha' \perp \beta'$, \ $\alpha \sim \alpha'$ and $\beta \sim \beta'$, then the complements $\Cant \setminus ( \cone{\alpha} \cup \cone{\beta} )$ and $\Cant \setminus ( \cone{\alpha'} \cup \cone{\beta'} )$ have the same s-type.
\end{enumerate}
\end{lemma}

\begin{proof}
In this proof we refer to the five conditions~\ref{i:q-depth}--\ref{i:q-reach} appearing in Definition~\ref{def:nuketypes}\ref{i:quasi}.  By hypothesis, $\branchPartQ \neq \emptyset$.

\ref{i:adjunct-id}~Let $P \in \Part$.  By Conditions \ref{i:q-depth}~and~\ref{i:q-reach}, there is a path in~$\graph{\Part}$ from~$P$ to~$R$.  Hence we deduce $p = x + r$ for some $x \in \sgp{\Part}$.  Since $\Part[R] = \{R\}$ and this is child-closed by Condition~\ref{i:q-nucleus}, it follows that $r = r + r$ is one of the defining relations in~$\sgp{\Part}$.  Hence $p + r = x + r + r = p$.  As this holds for all generators~$p$, we conclude that $r$~is an identity element of~$\sgp{\Part}$.

\ref{i:adjunct-sgpgen}~If $P \in \Part$, then by the defining relations in~$\sgp{\Part}$, \ $p = \sum_{\gamma} p_{\gamma}$, where we sum over all words~$\gamma$ of length~$t$.  By Condition~\ref{i:q-depth}, each $P_{\gamma} \in \Part[Q] \cup \{R\}$ and hence $\sgp{\Part}$~is generated by~$r$ together with all~$q$ for $Q \in \Part[Q]$.  If $Q \in \Part[Q]$, then by Condition~\ref{i:q-strong} there is a path in~$\graph{\Part}$ from~$Q$ to some vertex in~$\branchPartQ$.  Choose~$\gamma$ to be the shortest word such that $Q_{\gamma} \in \branchPartQ$, say $\gamma = y_{1}y_{2}\dots y_{m}$ where each $y_{i} \in \letters$.  Then $Q_{y_{1}y_{2}\dots y_{i}} \in \Part[Q] \setminus \branchPartQ$ for $i < m$, so the corresponding defining relation is $q_{y_{1}y_{2}\dots y_{i}} = q_{y_{1}y_{2}\dots y_{i+1}} + r$ and this equals $q_{y_{1}y_{2}\dots y_{i+1}}$ by use of part~\ref{i:adjunct-id}.  Hence we deduce $q = q_{\gamma}$.  It follows that $\sgp{P}$~is generated by~$r$ together with the elements~$q$ for~$Q \in \branchPartQ$.

\ref{i:adjunct-ideal}~Since $r$~is the identity element for~$\sgp{\Part}$, it follows from part~\ref{i:adjunct-sgpgen} that $\sgp{P} = I \cup \{r\}$ and that $I$~is an ideal.  Note also that the argument from~\ref{i:adjunct-sgpgen} shows that if $Q \in \Part[Q]$ then the corresponding generator~$q$ is an element of~$I$.  It remains to show that $\Ideal$~is a group.

Fix some $P \in \branchPartQ$.  Then $P_{0},P_{1} \in \Part[Q]$ by definition of~$\branchPartQ$.  There is some path in~$\graph{\Part}$ from~$P_{0}$ to~$P$.  Hence we deduce that $p_{0} = p + x$ for some $x \in \sgp{\Part}$.  Let $e = p_{1} + x$.  Then $e \in I$ since $p_{1} \in I$ and $I$~is an ideal.  Furthermore, $p = p_{0} + p_{1} = p + e$.  There is also a path in~$\graph{\Part}$ from~$P_{1}$ to~$P$ and so $p_{1} = p + y$ for some $y \in \sgp{\Part}$.

Now if $Q$~is any type in~$\branchPartQ$, there is a path in~$\graph{\Part}$ from~$P$ to~$Q$ and so $q = p + z$ for some $z \in \sgp{\Part}$.  Then $q + e = p + z + e = q$.  This shows that $e$~is an identity element in~$\Ideal$.  There is also a path from~$P_{1}$ to~$Q$, so $p_{1} = q + w$ for some $w \in \sgp{\Part}$.  Then
\[
e = p_{1} + x = p + x + y = q + (p_{0} + w + x + y)
\]
and here $p_{0} + w + x + y \in \Ideal$ since $\Ideal$~is an ideal.  In conclusion, every~$q$ corresponding to $Q \in \branchPartQ$ has an inverse in~$\Ideal$.  Hence $\Ideal$~is indeed a group.

We shall now establish the existence of the integer~$N$ required for parts~\ref{i:adjunct-compltype1}--\ref{i:adjunct-compltype2}.  There exists some member in~$\branchPartQ$, say with representative~$\gamma \in \AddrSet$.  Then $\gamma0$~and~$\gamma1$ both have types in~$\Part[Q]$ and so there exist paths in~$\graph{\Part}$ from the corresponding vertices to some members of~$\branchPartQ$.  Hence there exist incomparable addresses $\delta_{1}$~and~$\delta_{2}$ that have types in~$\branchPartQ$.  Take $N = \max\{ \length{\delta_{1}}, \length{\delta_{2}}, t \} + 1$.  Then $N \geq 2$.  We shall now verify that the three required claims hold for such~$N$.

\ref{i:adjunct-compltype1}~Suppose that $\length{\alpha}, \length{\beta} \geq N$.  Then at least one (indeed, at least two) of the incomparable addresses $\delta_{1}0$,~$\delta_{1}1$, $\delta_{2}0$ and~$\delta_{2}1$ is not a prefix of~$\alpha$ or of~$\beta$.  Let $\delta_{i}x$, for $x \in \letters$, satisfy $\delta_{i}x \nprefix \alpha$ and $\delta_{i}x \nprefix \beta$.  Since $\delta_{i}x$~is shorter than these two addresses, we deduce $\cone{\delta_{i}x} \subseteq U = \Cant \setminus ( \cone{\alpha} \cup \cone{\beta} )$.  Write~$P$ for the type of~$\delta_{i}x$.  Then the corresponding element~$p$ is a summand of the s-type of~$U$ and, since $I$~is an ideal of~$\sgp{\Part}$, we deduce $\stype{U} \in I$, as claimed.

\ref{i:adjunct-perpinQdag}~Suppose $\length{\alpha}, \length{\alpha'} \geq N$.  Again, for some~$x \in \{0,1\}$, the address~$\delta_{i}x$ is not a prefix of $\alpha$~or~$\alpha'$, so that $\delta_{i}$~is incomparable with both these addresses.  By assumption, the type of~$\delta_{i}x$ is in~$\Part[Q]$.  Follow a sufficiently long path in~$\graph{\Part}$ from the corresponding vertex to some member of~$\branchPartQ$ to produce an address~$\beta$ with $\delta_{i} \prefix \beta$, \ $\length{\beta} \geq N$ and such that $\beta$~has type in~$\branchPartQ$.  Then $\beta$~is incomparable with both $\alpha$~and~$\alpha'$, as required.

\ref{i:adjunct-compltype2}~Let $\alpha$,~$\alpha'$, $\beta$ and~$\beta'$ be as in the statement.  Let $P$~and~$Q$ be the types of $\alpha$~and~$\beta$, respectively, and let $x$~and~$y$ be the respective s-types of the complements $\Cant \setminus ( \cone{\alpha} \cup \cone{\beta} )$ and $\Cant \setminus ( \cone{\alpha'} \cup \cone{\beta'} )$.  Then $p + q + x = p + q + y$ since both sides are the s-type of~$\Cant$.  If $P = Q = R$, then we deduce immediately that $x = y$ since $r$~is the identity element in~$\sgp{\Part}$.  Suppose then that either $P$~or~$Q$ is a member of~$\Part[Q]$.  Then $p + q \in \Ideal$, since $\Ideal$~is an ideal of~$\sgp{\Part}$.  We also know $x,y \in \Ideal$ by part~\ref{i:adjunct-compltype1}.  Since $\Ideal$~is a group, we deduce $x = y$ completing the proof of the lemma.
\end{proof}

We shall give a name to the integer parameter that appears in the above two lemmas.

\begin{defn}
We define the \emph{stable depth}~$\sd{\Part}$ of the type systems~$\Part$ under consideration as follows:
\begin{enumerate}
\item If $\Part$~is a multinuclear type system, its \emph{stable depth} is~$t+2$, where $t$~is the smallest integer such that an address~$\alpha$ has type in one of the nucleii whenever $\length{\alpha} \geq t$.
\item If $\Part$~is a quasinuclear type system that is both branching and atomic, its \emph{stable depth}~is the smallest integer~$N$ such that parts \ref{i:adjunct-compltype1}--\ref{i:adjunct-compltype2} of Lemma~\ref{lem:adjunct-sgp} hold.
\end{enumerate}
\end{defn}

The use of this same terminology for both classes of type systems will allow us to use Lemmas~\ref{lem:multinuke-sgp} and~\ref{lem:adjunct-sgp} in the same way.  Our definition ensures that $\sd{\Part} \geq 2$ in either case.  Hence if $\alpha$~and~$\beta$ are addresses of length at least the stable depth, then they necessarily have small support.

\begin{lemma}
\label{lem:nuke-incomparables}
Let $\Part$~be a type system on~$\AddrSet$, determined by the equivalence relation~$\sim$, that is either multinuclear or atomic branching quasinuclear.  Let $N \geq \sd{\Part}$ and $\alpha$,~$\alpha'$, $\beta_{1}$, $\beta_{2}$, \dots,~$\beta_{n}$ be addresses of length at least~$N$.  Then there exist pairwise incomparable addresses $\gamma_{1}$,~$\gamma_{2}$, \dots,~$\gamma_{n}$ such that, for each~$i$, \ $\beta_{i} \sim \gamma_{i}$, \ $\length{\gamma_{i}} \geq N$ and $\gamma_{i}$~is incomparable with both $\alpha$~and~$\alpha'$.
\end{lemma}

\begin{proof}
Consider first the case when $\Part$~is multinuclear and adopt the notation of Definition~\ref{def:nuketypes}\ref{i:multinuke}.  By definition, $N \geq t+2$.  Any address~$\theta$ of length at least~$N$ has type in one of the nucleii~$\Part[Q]^{(i)}$ and any descendent of~$\theta$ also has type belonging to~$\Part[Q]^{(i)}$.  As a consequence, it is sufficient to deal with the case when $\alpha$,~$\alpha'$, $\beta_{1}$, $\beta_{2}$, \dots,~$\beta_{n}$ all have types in the same nucleus~$\Part[Q]^{(i)}$, since incomparability follows immediately for addresses with types in different nucleii.  We proceed by induction on~$n$.  Let $\delta$~be a prefix of~$\alpha$ of length~$t$, so that $\delta$~also has type in~$\Part[Q]^{(i)}$.  Since $\length{\alpha}, \length{\alpha'} \geq t+2$, at least one of $\delta00$,~$\delta01$, $\delta10$ and~$\delta11$ is incomparable with both $\alpha$~and~$\alpha'$.  Write $\zeta$~for one of these four addresses satisfying this requirement.  Now follow a sufficiently long path in~$\graph{\Part}$ from the type of~$\zeta$ to that of~$\beta_{1}$ to produce an address~$\gamma_{1}$ that is incomparable with~$\alpha$ and satisfies $\gamma_{1} \sim \beta_{1}$ and $\length{\gamma_{1}} \geq N$.  This covers the base case~$n = 1$.

Now assume that $n > 1$ and there are incomparable addresses $\gamma_{1}$,~$\gamma_{2}$, \dots,~$\gamma_{n-1}$ such that $\beta_{i} \sim \gamma_{i}$, \ $\length{\gamma_{i}} \geq N$ and $\gamma_{i} \perp \alpha, \alpha'$ for $i = 1$,~$2$, \dots,~$n-1$.  Now $\gamma_{n-1}0$~and~$\gamma_{n-1}1$ are incomparable.  Follow paths in~$\graph{\Part}$ from the types of $\gamma_{n-1}0$~and~$\gamma_{n-1}1$ to those of $\beta_{n-1}$~and~$\beta_{n}$.  Replace~$\gamma_{n-1}$ by the first of the resulting addresses and define~$\gamma_{n}$ to be the second.  Then the new~$\gamma_{n-1}$ and~$\gamma_{n}$ are incomparable with the previous $\gamma_{1}$,~$\gamma_{2}$, \dots,~$\gamma_{n-2}$ and with both $\alpha$~and~$\alpha'$ and satisfy $\gamma_{n-1} \sim \beta_{n-1}$ and $\gamma_{n} \sim \beta_{n}$.  This completes the induction step in the multinuclear case.

Now suppose that $\Part$~is quasinuclear and both branching and atomic.  We use the notation of Definition~\ref{def:nuketypes}\ref{i:quasi} and write~$R$ for the unique type in~$\Part[R]$.  By Lemma~\ref{lem:adjunct-sgp}\ref{i:adjunct-perpinQdag}, there exists~$\zeta$ with type in~$\branchPartQ$ such that $\length{\zeta} \geq N$ and $\zeta$~is incomparable with both $\alpha$~and~$\alpha'$.  We claim that there exist pairwise incomparable addresses $\zeta_{1}$,~$\zeta_{2}$, \dots,~$\zeta_{n}$ that have $\zeta$~as prefix and that each has type in~$\branchPartQ$.  When $n = 1$, $\zeta_{1} = \zeta$ achieves the claim.  Suppose as inductive hypothesis that we have found such $\zeta_{1}$,~$\zeta_{2}$, \dots,~$\zeta_{n-1}$.  Then $\zeta_{n-1}0$~and~$\zeta_{n-1}1$ both have type in~$\Part[Q]$ and from these we can follow paths in~$\graph{\Part}$ to types in~$\branchPartQ$.  Replace~$\zeta_{n-1}$ by the resulting address determined by the first and take~$\zeta_{n}$ to be the second resulting address.  We have then completed the induction and established our claim.  Now for each~$i$, there is a path from the type~$\zeta_{i}$ to that of~$\beta_{i}$ (whether this is~$R$ or belongs to~$\Part[Q]$).  Hence we find an address~$\gamma_{i}$ with $\zeta_{i}$~as prefix and $\gamma_{i} \sim \beta_{i}$.  The resulting $\gamma_{1}$,~$\gamma_{2}$, \dots,~$\gamma_{n}$ then satisfy the claim of the lemma.
\end{proof}

\begin{rem}
For many of the applications of the above lemma, we shall not specify the addresses $\alpha$~and~$\alpha'$.  In these cases, all we seek are incomparable addresses~$\gamma_{i}$ that satisfy $\gamma_{i} \sim \beta_{i}$.
\end{rem}

The following lemma will be useful for showing that, for suitable type systems~$\Part$, if an overgroup~$H$ of~$\Fix{V}{\Part}$ contains a suitable transposition~$\swap{\alpha}{\beta}$ with $\alpha \nsim \beta$ then it contains many transpositions of this form.

\begin{lemma}
\label{lem:conjugator}
Let $\Part$~be a type system that is determined by the equivalence relation~$\sim$ on~$\AddrSet$ and assume that $\Part$~is either multinuclear or atomic branching quasinuclear.  Let $N \geq \sd{\Part}$.
\begin{enumerate}
\item \label{i:conjugator}
If $\alpha$,~$\alpha'$, $\beta$ and~$\beta'$ are addresses of length at least~$N$ such that $\alpha \sim \alpha'$, \ $\beta \sim \beta'$, \ $\alpha \perp \beta$ and $\alpha' \perp \beta'$, then there exists $g \in \Fix{V}{\Part}$ such that $\alpha^{g} = \alpha'$ and $\beta^{g} = \beta'$.

\item \label{i:allswaps}
Suppose that $H$~is a subgroup of\/~$V$ with $\Fix{V}{\Part} \leq H$ such that $\swap{\alpha}{\beta} \in H$ for some incomparable addresses $\alpha$~and~$\beta$ with $\length{\alpha}, \length{\beta} \geq N$.  Then $\swap{\alpha'}{\beta'} \in H$ for all incomparable addresses $\alpha'$~and~$\beta'$ with $\length{\alpha'}, \length{\beta'} \geq N$, \ $\alpha' \sim \alpha$ and $\beta' \sim \beta$.
\end{enumerate}
\end{lemma}

\begin{proof}
\ref{i:conjugator}~Let $U = \Cant \setminus ( \cone{\alpha} \cup \cone{\beta} )$ and $U' = \Cant \setminus ( \cone{\alpha'} \cup \cone{\beta'} )$.  By Lemma~\ref{lem:multinuke-sgp}\ref{i:multinuke-compltype} or by Lemma~\ref{lem:adjunct-sgp}\ref{i:adjunct-compltype2}, as appropriate, the s-types of $U$~and~$U'$ are equal.  Now apply Lemma~\ref{lem:clopen-type} to these complements and hence express
\[
\Cant = \cone{\alpha} \cup \cone{\beta} \cup \biggl( \bigcup_{i=1}^{n} \cone{\gamma_{i}} \biggr) = \cone{\alpha'} \cup \cone{\beta'} \cup \biggl( \bigcup_{i=1}^{n} \cone{\delta_{i}} \biggr)
\]
as disjoint unions of cones such that $\gamma_{i} \sim \delta_{i}$ for each~$i$.  We may therefore define an element~$g$ of~$V$ by the prefix substitutions $\alpha \mapsto \alpha'$, \ $\beta \mapsto \beta'$ and $\gamma_{i} \mapsto \delta_{i}$ for $i = 1$,~$2$, \dots,~$n$.  This element achieves what is claimed.

\ref{i:allswaps}~By part~\ref{i:conjugator}, there exists $g \in \Fix{V}{\Part}$ with $\alpha^{g} = \alpha'$ and $\beta^{g} = \beta'$.  Then $\swap{\alpha'}{\beta'} = \swap{\alpha}{\beta}^{g} \in H$.
\end{proof}

\begin{defn}
If $\Part$~is a multinuclear type system on~$\AddrSet$ with $\Part[Q]^{(1)}$,~$\Part[Q]^{(2)}$, \dots,~$\Part[Q]^{(k)}$ its nucleii, we denote by  $\addrs{\Part[Q]^{(i)}}$ the set of addresses~$\alpha$ with type in~$\Part[Q]^{(i)}$ and write
\[
\supt{\Part[Q]^{(i)}}= \bigcup \set{ \cone{\alpha} }{ \text{$\alpha$~has type in~$\Part[Q]^{(i)}$} }.
\]
\end{defn}

Since each nucleus~$\Part[Q]^{(i)}$ of a multinuclear type system is child-closed, if $\alpha$~and~$\beta$ have types in $\Part[Q]^{(i)}$~and~$\Part[Q]^{(j)}$, respectively, with $i \neq j$, then $\alpha \perp \beta$.  Consequently, the above multinuclear type system~$\Part$ produces a disjoint union decomposition
\[
\Cant = \bigcup_{i=1}^{k} \supt{\Part[Q]^{(i)}}
\]
associated to its nucleii.

\begin{lemma}
\label{lem:FixTransitiveProperClopeninSType}
Suppose $\Part$~is a multinuclear type system on~$\AddrSet$ with $\Part[Q]^{(1)}$,~$\Part[Q]^{(2)}$, \dots,~$\Part[Q]^{(k)}$ its nucleii.  Let $U$~and~$V$ be non-empty clopen subsets of~$\Cant$ such that
\begin{enumerate}
\item $\stype{U} = \stype{V}$,
\item $U \cap \supt{\Part[Q]^{(i)}} \neq \supt{\Part[Q]^{(i)}}$ for all~$i$, and
\item $V \cap \supt{\Part[Q]^{(i)}} \neq \supt{\Part[Q]^{(i)}}$ for all~$i$.
\end{enumerate}
Then there exists $g \in \Fix{V}{\Part}$ such that $Ug = V$.
\end{lemma}

\begin{proof}
Let $\sim$~be the equivalence relation that determines the type system~$\Part$.  For $i = 1$,~$2$, \dots,~$k$, define $S_{i} = \supt{\Part[Q]^{(i)}}$.  As $\stype{U} = \stype{V}$, Lemma~\ref{lem:clopen-type} produces decompositions
\[
U = \bigcup_{j=1}^{n} \cone{\gamma_{j}}
\AND
V = \bigcup_{j=1}^{n} \cone{\delta_{j}}
\]
into disjoint cones such that $\gamma_{j} \sim \delta_{j}$ for $j = 1$,~$2$, \dots,~$n$.  Refining these addresses if necessary, we may assume that each~$\gamma_{j}$ and each~$\delta_{j}$ has type in some nucleus~$\Part[Q]^{(i)}$.  Now fix $i \in \{1,2,\dots,k\}$ and, after relabelling for convenience of notation, suppose that $\gamma_{1}$,~$\gamma_{2}$, \dots,~$\gamma_{r}$ are the corresponding addresses for~$U$ that have type in~$\Part[Q]^{(i)}$.  Then $\delta_{1}$,~$\delta_{2}$, \dots,~$\delta_{r}$ are the addresses for~$V$ with type in~$\Part[Q]^{(i)}$.  Hence
\[
U \cap S_{i} = \bigcup_{j=1}^{r} \cone{\gamma_{j}}
\AND
V \cap S_{i} = \bigcup_{j=1}^{r} \cone{\delta_{j}}.
\]
Consequently, $U \cap S_{i}$~is empty if and only if $V \cap S_{i}$~is empty (namely, when $r = 0$).  If this were the case, then immediately the complements $S_{i} \setminus U$ and $S_{i} \setminus V$ has the same s-type as they both equal~$S_{i}$.  Let us suppose then that $r > 0$.  The above formulae for the intersections then ensure that $\stype{U \cap S_{i}} = \stype{V \cap S_{i}}$.  Now the latter is an element of the group~$G_{i}$ appearing in Lemma~\ref{lem:multinuke-sgp}\ref{i:multinuke-gps} and, since $S_{i} \setminus U$ and $S_{i} \setminus V$ are non-empty by hypothesis,
\[
\stype{S_{i}} = \stype{U \cap S_{i}} + \stype{S_{i} \setminus U} = \stype{V \cap S_{i}} + \stype{S_{i} \setminus V}
\]
is an equation in this group.  Therefore $\stype{S_{i} \setminus U} = \stype{S_{i} \setminus V}$ in this case also.  We may therefore apply Lemma~\ref{lem:clopen-type} to the clopen sets $S_{i} \setminus U$ and~$S_{i} \setminus V$ and hence express~$S_{i}$ as disjoint unions
\[
S_{i} =
\biggl( \bigcup_{j=1}^{r} \cone{\gamma_{j}} \biggr) \cup \biggl( \bigcup_{\ell=1}^{s} \cone{\gamma'_{\ell}} \biggr)
= \biggl( \bigcup_{j=1}^{r} \cone{\delta_{j}} \biggr) \cup \biggl( \bigcup_{\ell=1}^{s} \cone{\delta'_{\ell}} \biggr)
\]
where $\gamma'_{\ell} \sim \delta'_{\ell}$ for each~$\ell$.  Repeating this argument for each $i \in \{1,2,\dots,k\}$,  enables us to write
\[
\Cant =
\biggl( \bigcup_{j=1}^{n} \cone{\gamma_{j}} \biggr) \cup \biggl( \bigcup_{\ell=1}^{m} \gamma'_{\ell} \biggr)
= \biggl( \bigcup_{j=1}^{n} \cone{\delta_{j}} \biggr) \cup \biggl( \bigcup_{\ell=1}^{m} \delta'_{\ell} \biggr)
\]
where $\gamma'_{\ell} \sim \delta'_{\ell}$ for each~$\ell$.
%\cpb{(here, we have gathered together (through disjoint unions over the index $i$ enumerating the nucleii) each of the four labelled types of addresses into four sets (the $\gamma$'s, $\gamma'$'s, $\delta$'s and $\delta'$'s)).}
(As an aside, the value of~$n$ here is unchanged from that when we refined the addresses $\gamma_{j}$~and~$\delta_{j}$ to ensure they had type in the nucleii~$\Part[Q]^{(i)}$.  What we have done is find suitable partitions of the complements $\Cant\setminus U$ and~$\Cant\setminus V$.)
We may now define an element~$g$ of~$V$ by the prefix substitutions $\gamma_{j} \mapsto \delta_{j}$ for $1 \leq j \leq n$ and $\gamma'_{\ell} \mapsto \delta'_{\ell}$ for $1 \leq \ell \leq m$.  The result is an element~$g$ in~$\Fix{V}{\Part}$ that maps~$U$ to~$V$.
\end{proof}

\begin{prop}
\label{prop:typequotient}
Let $\Part$~be a type system that is determined by the equivalence relation~$\sim$ on~$\AddrSet$ and assume that $\Part$~is either multinuclear or atomic branching quasinuclear.  Let $H$~be a subgroup of~$V$ such that $\Fix{V}{\Part} \leq H$ and let $N \geq \sd{\Part}$.  Define a relation~$\approx$ on~$\AddrSet$ by the rule $\alpha \approx \beta$ when
\begin{quote}
for all words~$\eta$ with $\length{\eta} = N$, there exist incomparable addresses $\gamma$~and~$\delta$ with $\length{\gamma}, \length{\delta} \geq N$, \ $\gamma \sim \alpha\eta$, \ $\delta \sim \beta\eta$ and $\swap{\gamma}{\delta} \in H$.
\end{quote}
Then the partition determined by $\approx$~is a type system on~$\AddrSet$ that is a quotient of~$\Part$.
\end{prop}

\begin{proof}
Suppose first that $\alpha$~and~$\beta$ are addresses with $\alpha \sim \beta$.  Let $\eta$~be any word with $\length{\eta} = N$.  Since $\sim$~is coherent, necessarily $\alpha\eta \sim \beta\eta$.  By Lemma~\ref{lem:nuke-incomparables}, there exist incomparable addresses $\gamma$~and~$\delta$ with $\length{\gamma}, \length{\delta} \geq N$ and $\gamma \sim \alpha\eta \sim \beta\eta \sim \delta$.  Hence $\swap{\gamma}{\delta} \in \Fix{V}{\Part} \leq H$.  As this holds for all choices of~$\eta$, we conclude that $\alpha \approx \beta$.  It follows first that $\approx$~is reflexive (upon taking $\alpha = \beta$) and secondly that, once we know $\approx$~is a coherent and reduced equivalence relation on~$\AddrSet$, it will define a type system which is a quotient of~$\sim$.  The definition of~$\approx$ is certainly symmetric in $\alpha$~and~$\beta$.  The next step then is to show that $\approx$~is transitive.

Suppose that $\alpha$,~$\beta$ and~$\gamma$ are addresses satisfying $\alpha \approx \beta \approx \gamma$.  Let $\eta$~be any word of length~$N$.  Then, by definition of~$\approx$ and Lemma~\ref{lem:conjugator}\ref{i:allswaps}, \ $\swap{\alpha'}{\beta'} \in H$ for all incomparable addresses $\alpha'$~and~$\beta'$ of length~$\geq N$ with $\alpha' \sim \alpha\eta$ and $\beta' \sim \beta\eta$ and $\swap{\beta'}{\gamma'} \in H$ for all incomparable addresses $\beta'$~and~$\gamma'$ of length~$\geq N$ with $\beta' \sim \beta\eta$ and $\gamma' \sim \gamma\eta$.  By Lemma~\ref{lem:nuke-incomparables}, there exist a choice of incomparable addresses $\alpha'$,~$\beta'$ and~$\gamma'$ such that $\alpha' \sim \alpha\eta$, \ $\beta' \sim \beta\eta$ and $\gamma' \sim \gamma\eta$.  Then $\swap{\alpha'}{\beta'}, \swap{\beta'}{\gamma'} \in H$ for these addresses and so
\[
\swap{\alpha'}{\gamma'} = \swap{\alpha'}{\beta'}^{\swap{\beta'}{\gamma'}} \in H.
\]
This holds for all~$\eta$ of length~$N$ and so we conclude $\alpha \approx \gamma$.  This shows that $\approx$~is transitive.

We now show that the equivalence relation~$\approx$ is coherent.  Suppose that $\alpha \approx \beta$ for some addresses $\alpha$~and~$\beta$.  Let $\eta$~be any word of length~$N$ and write $\eta = \zeta y$ where $\length{\zeta} = n-1$ and $y \in \{0,1\}$.  If $x \in \{0,1\}$, then $x\zeta$~has length~$n$, so by the hypothesis combined with Lemma~\ref{lem:conjugator}\ref{i:allswaps}, \ $\swap{\gamma}{\delta} \in H$ for all incomparable addresses $\gamma$~and~$\delta$ of length~$\geq N$ with $\gamma \sim \alpha x\zeta$ and $\delta \sim \beta x\zeta$.  By Lemma~\ref{lem:nuke-incomparables}, there exist incomparable addresses $\gamma$,~$\delta$ and~$\delta'$ of length $\geq N$ such that $\gamma \sim \alpha x\zeta$ and $\delta \sim \delta' \sim \beta x\zeta$.  Since $\sim$~is coherent, $\delta y \sim \delta'y$ and therefore $\swap{\delta y}{\delta'y} \in \Fix{V}{\Part} \leq H$.  On the other hand, $\swap{\gamma}{\delta} \in H$ by assumption, so
\[
\swap{\gamma y}{\delta'y} = \swap{\delta y}{\delta'y}^{\swap{\gamma}{\delta}} \in H.
\]
Note $\alpha x\eta = \alpha x\zeta y \sim \gamma y$ and $\beta x\eta = \beta x\zeta y \sim \delta'y$.  Since $\eta$~is an arbitrary word of length~$N$, we conclude $\alpha x \approx \beta x$ for any $x \in \{0,1\}$.  This shows that $\approx$~is coherent.

Finally, suppose $\alpha$~and~$\beta$ are addresses such that $\alpha0 \approx \beta0$ and $\alpha1 \approx \beta 1$.  Let $x,y \in \{0,1\}$ and $\zeta$~be any word of length~$N-1$.  Then $\eta = x\zeta$ is an arbitrary word of length~$N$.  Let $\gamma$~and~$\delta$ be incomparable addresses of length~$\geq N$ such that $\gamma \sim \alpha\eta$ and $\delta \sim \beta \eta$.  Now $\gamma y \sim \alpha\eta y = \alpha x\zeta y$ and $\delta y \sim \beta\eta y = \beta x\zeta y$.  Since $\alpha x \approx \beta x$ by assumption, use of Lemma~\ref{lem:conjugator}\ref{i:allswaps} tells us that $\swap{\gamma y}{\delta y} \in H$.  This is true for both $y \in \{0,1\}$ and therefore
\[
\swap{\gamma}{\delta} = \swap{\gamma0}{\delta0} \, \swap{\gamma1}{\delta1} \in H.
\]
As $\eta$~is an arbitrary word of length~$N$, we conclude that $\alpha \approx \beta$.  This shows that $\approx$~is reduced and completes the proof of the proposition.
\end{proof}

\begin{lemma}
\label{lem:notStab}
Let $\Part$~be a type system that is determined by the equivalence relation~$\sim$ on~$\AddrSet$ and assume that $\Part$~is either multinuclear or atomic branching quasinuclear.  Let $N \geq \sd{\Part}$ and let $H$~be a subgroup of~$V$ such that $\Stab{V}{\Part} < H$.  Then there exist incomparable addresses $\alpha$~and~$\beta$ of length at least~$N$ and $h \in H$ such that $\alpha \sim \beta$ but $\alpha^{h} \nsim \beta^{h}$.
\end{lemma}

\begin{proof}
There exists some~$h \in H$ that does not stabilize the equivalence relation~$\sim$.  Replacing~$h$ by its inverse if necessary, we can assume there exist addresses $\alpha$~and~$\beta$ such that $\alpha^{h}$~and~$\beta^{h}$ are both defined, $\alpha \sim \beta$ but that $\alpha^{h} \nsim \beta^{h}$.  First note that if $\eta$~is any word of length~$N$, then $\alpha\eta \sim \beta\eta$ since $\sim$~is coherent, while $(\alpha\eta)^{h} = \alpha^{h}\eta$ and $(\beta\eta)^{h} = \beta^{h}\eta$.  Since $\sim$~is reduced, there must exist a word~$\eta$ of length~$N$ such that $\alpha^{h}\eta \nsim \beta^{h}\eta$.  Hence we may first replace $\alpha$~and~$\beta$ by the resulting addresses $\alpha\eta$~and~$\beta\eta$ so as to assume that $\length{\alpha}, \length{\beta} \geq N$.  We shall show that such addresses exist that are incomparable by considering the two possibilities for~$\Part$ in turn.

Suppose first that $\Part$~is multinuclear and that $\alpha$~has type in the nucleus~$\Part[Q]^{(i)}$ for some~$i$.  Then $\alpha0$~and~$\alpha1$ also have types in~$\Part[Q]^{(i)}$ and, again since $\sim$~is reduced, there is some $x \in \letters$ such that $(\alpha x)^{h} \nsim (\beta x)^{h}$.  If $\alpha x$~and~$\beta x$ are incomparable, then we can replace $\alpha$~and~$\beta$ by $\alpha x$~and~$\beta x$ to achieve what is sought.  Suppose instead that $\alpha x \prefix \beta x$.  Let $y$~be the other letter to~$x$ in~$\letters$.  There is a path indexed by the word~$\zeta$ in~$\graph{\Part}$ from the type of~$\alpha y$ to the type of~$\alpha x$.  Note that $\alpha y \zeta$~is incomparable with~$\alpha x$ and therefore also incomparable with~$\beta x$.  Take $\gamma_{0} = \alpha y \zeta$, \ $\gamma_{1} = \alpha x$ and $\gamma_{2} = \beta x$.  Then $\gamma_{0} \sim \gamma_{1} \sim \gamma_{2}$, \ $\gamma_{0}$~is incomparable with $\gamma_{1}$~and~$\gamma_{2}$, \ $\gamma_{0}^{h}$~is defined and $\gamma_{1}^{h} \nsim \gamma_{2}^{h}$.  Then either $\gamma_{0}^{h} \nsim \gamma_{1}^{h}$ or $\gamma_{0}^{h} \nsim \gamma_{1}^{h}$.  Hence we may replace $\alpha$~and~$\beta$ by~$\gamma_{0}$ and one of $\gamma_{1}$~or~$\gamma_{2}$ to achieve the claim.

Now consider the case when $\Part$~is atomic branching quasinuclear and we use the notation of Definition~\ref{def:nuketypes}\ref{i:quasi}, writing~$R$ for the unique type in its nucleus.  Then either the type of $\alpha$~and~$\beta$ is in~$\Part[Q]$ or they both have type~$R$.  In the former case, there is some shortest word~$\zeta$ (possibly empty) such that $\alpha\zeta$~has type in~$\branchPartQ$.  Let $\zeta_{1}$,~$\zeta_{2}$, \dots,~$\zeta_{k}$ be all the words of length equal to that of~$\zeta$.  If $\theta$~is a proper prefix of~$\zeta$, then $\alpha\theta$~has type in~$\Part[Q] \setminus \branchPartQ$, so one of $\alpha\theta0$~and~$\alpha\theta1$ has type~$R$ and the other has type in~$\Part[Q]$.  Moreover, it must be the latter that is a prefix of~$\alpha\zeta$.  As a consequence, precisely one of the addresses~$\alpha\zeta_{i}$ has type in~$\branchPartQ$ (namely~$\alpha\zeta$) and the others have type~$R$.  Using the same argument as in the first paragraph of the proof, we may replace $\alpha$~and~$\beta$ by $\alpha\zeta_{i}$~and~$\beta\zeta_{i}$ for some choice of~$i$ and hence assume that either $\alpha$~and~$\beta$ have type~$R$ or they both have type in~$\branchPartQ$.  In either case, we now apply a similar argument to the multinuclear case.  There is a letter $x \in \letters$ such that $(\alpha x)^{h} \nsim (\beta x)^{h}$.  If $\alpha x \perp \beta x$, then we use these addresses.  Otherwise, assume $\alpha x \prefix \beta x$ and let $y$~be the other letter in~$\letters$.  There is a path in~$\graph{\Part}$ from the type of~$\alpha y$ to that of~$\alpha x$.  Following this path produces an address~$\gamma_{0}$ with prefix~$\alpha y$ satisfying $\gamma_{0} \sim \alpha x$ and such that $\gamma_{0}^{h}$~is defined.  We then apply the same argument with $\gamma_{0}$, \ $\gamma_{1} = \alpha x$ and $\gamma_{2} = \beta x$ as before to complete the proof.
\end{proof}

We may now establish the first of our theorems about the maximality of the stabilizer of a type system.

\BranchType

\begin{proof}
As $\Part$ is a simple type system, the associated relation~$\sim$ is neither trivial nor universal.  Hence $\Stab{V}{\Part}$~is a proper subgroup of~$V$ by Lemma~\ref{lem:Stab-proper}.  Suppose that $H$~is a subgroup of~$V$ with $\Stab{V}{\Part} < H \leq V$, let $N \geq \sd{\Part}$ \ (the stable depth of~$\Part$) and let $\approx$~be the equivalence relation on~$\AddrSet$ with type system $\Part[Q]$ provided by Proposition~\ref{prop:typequotient}.  By Lemma~\ref{lem:notStab}, there exist $h \in H$ and incomparable addresses $\alpha$~and~$\beta$ with $\length{\alpha}, \length{\beta} \geq N$, \ $\alpha \sim \beta$ and $\alpha^{h} \nsim \beta^{h}$.  If $\eta$~is any word of length~$N$, then $\alpha\eta \sim \beta\eta$ by coherence, so $\swap{\alpha\eta}{\beta\eta} \in \Fix{V}{\Part} \leq H$.  Then
\[
\swap{\alpha^{h}\eta}{\beta^{h}\eta} = \swap{\alpha\eta}{\beta\eta}^{h} \in H.
\]
This holds for all~$\eta$ of length~$N$ and so $\alpha^{h} \approx \beta^{h}$.  Hence $\Part[Q]$~is a \emph{proper} quotient of~$\Part$ and so $\approx$ must be the universal relation by our assumption of simplicity.

Now suppose $\gamma$~and~$\delta$ are any incomparable addresses.  Then $\gamma \approx \delta$, so if $\eta$~is a word of length~$N$ then, with use of Lemma~\ref{lem:conjugator}\ref{i:allswaps}, we deduce $\swap{\gamma'}{\delta'} \in H$ for all incomparable addresses $\gamma'$~and~$\delta'$ of length~$\geq N$ with $\gamma' \sim \gamma\eta$ and $\delta' \sim \delta\eta$.  In particular, $\swap{\gamma\eta}{\delta\eta} \in H$.  This holds for all words~$\eta$ of length~$N$ and hence $\swap{\gamma}{\delta} \in H$ as it is the product of all transpositions~$\swap{\gamma\eta}{\delta\eta}$ as $\eta$~ranges over all words of length~$N$.  Since $V$~is generated by all such transpositions~$\swap{\gamma}{\delta}$, we conclude that $H = V$.  This establishes that $\Stab{V}{\Part}$ is a maximal subgroup of~$V$.
\end{proof}

%%%%%%%%%%%%%%%%%%%%%%%%%%%%%%%%%%%%%%%%%%%%%%%%%%
%% QUASINUCLEAR "cycle type"

The final case to consider is that of a simple type system that is atomic quasinuclear but not branching.

\begin{thm}
\label{thm:quasicycle}
Let $\Part$~be a simple type system on~$\AddrSet$ that is atomic quasinuclear but not branching.  Then there is a finite subset~$D$ of~$\Cant$ consisting of rational elements of the same tail class such that $\Stab{V}{\Part} = \Stab{V}{D}$ and this is a maximal subgroup of~$V$.
\end{thm}

The proof of this theorem will occupy the remainder of this section.  Let $\Part$ be a simple type system on~$\AddrSet$, defined by the equivalence relation~$\sim$, that is atomic quasinuclear.  Suppose that the positive integer~$t$ and the subsets $\Part[Q]$~and~$\Part[R]$ of~$\Part$ are as in Definition~\ref{def:nuketypes}\ref{i:quasi} and that $R$~is the unique type in~$\Part[R]$.  We assume that $\branchPartQ = \emptyset$.  This means that if $Q \in \Part[Q]$, then one edge from~$Q$ in the graph~$\graph{\Part}$ leads to another member of~$\Part[Q]$ and the other leads to the vertex~$R$.

Since every address of length~$\geq t$ either has type in~$\Part[Q]$ or has type~$R$ and since if $\alpha \in R$ then $\alpha0,\alpha1 \in R$ also, there exist finitely many distinct addresses $\beta_{1}$,~$\beta_{2}$, \dots,~$\beta_{m}$ that are minimal subject to having type that is in~$\Part[Q]$.  These addresses are consequently incomparable with each other.  Let $Q^{(1)}$~be the type of~$\beta_{1}$.  If we start at the vertex~$Q^{(1)}$ in the graph~$\graph{\Part}$, there is one edge leading to a member of~$\Part[Q]$.  Repeatedly following the unique edge to the member of~$\Part[Q]$ must eventually lead back to~$Q^{(1)}$ in view of Condition~\ref{i:q-reach} of the definition.  Equally, in following this path we must have passed through all members of~$\Part[Q]$, so there exist letters $x_{1}$,~$x_{2}$, \dots,~$x_{n}$ in~$\letters$ such that the distinct types in~$\Part[Q]$ are $Q^{(1)}$,~$Q^{(2)}$, \dots,~$Q^{(n)}$ where $Q^{(i)}$~is the type of the address $\beta_{1}x_{1}x_{2}\dots x_{i-1}$ for $i = 1$,~$2$, \dots,~$n$.  Furthermore, when we define $\zeta = x_{1}x_{2}\dots x_{n}$, the address~$\beta_{1}\zeta$ has the same type~$Q^{(1)}$ as~$\beta_{1}$.

\begin{lemma}
\label{lem:cycle-nopower}
Under the above hypothesis that $\Part$~is a simple type system that is atomic quasinuclear with equivalence relation~$\sim$ and for which $\branchPartQ = \emptyset$, the word~$\zeta$ labelling the cycle starting at the vertex~$Q^{(1)}$ cannot be expressed as a power of some word of shorter length.
\end{lemma}

\begin{proof}
Suppose that $\zeta$~is a power of a word of length~$\ell < n$.  This means $x_{i+\ell} = x_{i}$ for every value of~$i \geq 1$ with $i+\ell \leq n$.  Define a relation~$\approx$ on~$\AddrSet$ by the rule $\alpha \approx \beta$ when
\begin{quote}
for all words~$\eta$ with $\length{\eta} = t$, either $\alpha\eta \sim \beta\eta$, or $\alpha\eta \in Q^{(i)}$ and $\beta\eta \in Q^{(j)}$ where $i \equiv j \pmod{\ell}$.
\end{quote}
The goal of this definition is to glue together every $\ell$th equivalence class~$Q^{(i)}$ while ensuring that the result still defines a type system.  The fact that $\approx$~is reflexive and symmetric follows immediately, while it is straightforward to check that $\approx$~is transitive.  Suppose $\alpha \approx \beta$.  Let $\theta$~be any word of length~$t-1$ and let $x,y \in \letters$.  If $\alpha x\theta \sim \beta x\theta$, then it follows from coherence that $\alpha x\theta y \sim \beta x\theta y$ and since $\theta y$~is an arbitrary word of length~$t$, we deduce $\alpha x \approx \beta x$.  Otherwise, $\alpha x\theta \in Q^{(i)}$ and $\beta x\theta \in Q^{(j)}$ for some $i$~and~$j$ with $i \equiv j \pmod{\ell}$.  Our assumption on the word~$\zeta$ ensures that the letter~$y$ either labels the edges in~$\graph{\Part}$ from both $Q^{(i)}$~and~$Q^{(j)}$ to vertices in~$\Part[Q]$ or it labels the edges to the vertex~$R$.  Hence either $\alpha x\theta y \in Q^{(i+1)}$ and $\beta x\theta y \in Q^{(j+1)}$ (computed modulo~$n$) or they both have type~$R$, in which case $\alpha x\theta y \sim \beta x\theta y$.  Again, since $\theta y$~is an arbitrary word of length~$t$ we conclude $\alpha x \approx \beta x$ in this case.  This shows that $\approx$~is a coherent equivalence relation on~$\AddrSet$.

Suppose $\alpha0 \approx \beta0$ and $\alpha1 \approx \beta1$.  Let $\eta$~be an arbitrary word of length~$t$ and write $\eta = x\theta$ where $x \in \letters$ and $\length{\theta} = t-1$.  Now $\length{\alpha\eta},\length{\beta\eta} \geq t$, so each of these addresses either has type in one of the~$Q^{(i)}$ or type~$R$.  Suppose $\alpha\eta$~has type~$R$, so that $\alpha\eta0, \alpha\eta1 \in R$ also.  Since $\alpha x \approx \beta x$, it must be the case that $\alpha x\zeta y = \alpha\eta y$ and $\beta x\zeta y = \beta\eta y$ are related under~$\sim$ for both choices of $y \in \letters$, so $\beta\eta0, \beta\eta1 \in R$.  As $\sim$~is reduced, the only possibility is that $\beta\eta$~has type~$R$ and hence $\alpha\eta \sim \beta\eta$ in this case.  On the other hand, if $\alpha\eta \in Q^{(i)}$ for some~$i$, then $\alpha\eta y_{1} \in Q^{(i+1)}$ (computed modulo~$n$) and $\alpha\eta y_{2} \in R$ for the appropriate choice of $y_{1},y_{2} \in \letters$.  Since $\alpha x \approx \beta x$, we deduce that $\beta\eta y_{1} \in Q^{(j)}$ where $j \equiv i+1 \pmod{\ell}$ and $\beta\eta y_{2} \in R$.  The assumption on the word~$\zeta$ ensures that $y_{1}$~is the label on the edge from~$Q^{(j-1)}$ to~$Q^{(j)}$.  If $\gamma \in Q^{(j-1)}$, then $\gamma y_{1} \sim \beta\eta y_{1}$ (as both belong to~$Q^{(j)}$) and $\gamma y_{2} \sim \beta\eta y_{2}$ (as both belong to~$R$).  Hence $\gamma \sim \beta\eta$ since $\sim$~is reduced.  Thus $\beta\eta \in Q^{(j-1)}$ and here $j-1 \equiv i \pmod{m}$.  This shows that the condition for~$\approx$ is satisfied whether $\alpha\eta \in Q^{(i)}$ or~$R$; that is, $\alpha \approx \beta$.  Hence $\approx$~is also reduced.

Finally, the type system defined by~$\approx$ is a proper quotient of~$\Part$, since if $\alpha \sim \beta$ then $\alpha\eta \sim \beta\eta$ for any word~$\eta$ while some addresses are related under~$\approx$ that were not related under~$\sim$.  This contradicts the hypothesis that $\Part$~is simple.  Therefore, there can be no decomposition of~$\zeta$ as a power of some shorter word.
\end{proof}

Let us return to the distinct words $\beta_{1}$,~$\beta_{2}$, \dots,~$\beta_{m}$.  For each~$i$, there exists $k = k(i)$ such that $\beta_{i}$~has type~$Q^{(k)}$.  Note that $\beta_{i}x_{k(i)}x_{k(i)+1}\dots x_{n}$~has type~$Q^{(1)}$ for each~$i$.  Define $u_{i} = \beta_{j}x_{k(i)}x_{k(i)+1}\dots x_{m} \overline{\zeta} \in \Cant$ and set $D = \{ u_{1},u_{2},\dots,u_{n} \}$.  The elements of~$D$ then share the same tail class and are certainly rational.  An address with one of the types~$Q^{(j)}$ has some~$\beta_{i}$ as a prefix and the corresponding vertex in the graph~$\graph{\Part}$ is reached by first following the path labelled by~$\beta_{i}$ and then following a sufficient subsequent path around the cycle comprising $Q^{(1)}$,~$Q^{(2)}$, \dots,~$Q^{(n)}$ (starting at the vertex~$Q^{(k(i))}$).  Hence the set of all addresses that have type in any of the~$Q^{(j)}$ is given by
\[
\bigcup_{j=1}^{n} Q^{(j)} = \set{ \alpha }{ \text{$\beta_{i} \prefix \alpha \properprefix u_{i}$ for some~$i$} }.
\]

Now let $g \in \Stab{V}{D}$.  Fix some~$i \in \{1,2,\dots,m\}$ and set $u_{\ell} = u_{i}^{g}$.  By refining the cones in the domain of~$g$ appropriately, we can assume that there is an integer~$v \geq 0$ such that $g$~is given by a prefix substitution on the cone with address $\gamma_{i} = \beta_{i}x_{k(i)}x_{k(i)+1}\dots x_{n}\zeta^{v}$ (which has type~$Q^{(1)}$); say $\gamma_{i}^{g} = \delta$ for some address~$\delta$.  Now $u_{i} = \gamma_{i}\overline{\zeta}$, so $u_{\ell} = \delta\overline{\zeta}$.  It follows then by Lemma~\ref{lem:cycle-nopower} that $\delta$~has the same form as~$\gamma_{i}$: \ $\delta = \beta_{\ell}x_{k(\ell)}x_{k(\ell)+1}\dots x_{n}\zeta^{w}$ for some $w \geq 0$ and so $\delta$~also has type~$Q^{(1)}$.  Now if $\alpha$~is an address of some type~$Q^{(j)}$ with~$\gamma_{i}$ as a prefix, say $\alpha = \gamma_{i}\eta$ where $\eta$~is a prefix of some power of~$\zeta$, then $\alpha^{g} = \delta\eta \in Q^{(j)}$ also.  This observation applies for each $i \in \{1,2,\dots,n\}$ and so we conclude that, for each~$j$, \ $\alpha^{g} \in Q^{(j)}$ for all but finitely many~$\alpha \in Q^{(j)}$.  In particular, all sufficiently long addresses of type~$Q^{(j)}$ are images of addresses of the same type~$Q^{(j)}$ under the transformation~$g$.  Now if $\alpha \in R$ such that $\alpha^{g}$~is defined, then it cannot be the case that $\alpha^{g} \in Q^{(j)}$ as otherwise all addresses of the form~$\alpha^{g}\eta$ (for $\eta \in \AddrSet$) would include arbitrary long addresses of type~$Q^{(j)}$ but they would all be images of addresses of type~$R$.  In conclusion, but for finitely many exceptions, $g$~maps the addresses in each of the classes $Q^{(1)}$,~$Q^{(2)}$, \dots,~$Q^{(n)}$ and~$R$ back into the same class; that is, $g \in \Fix{V}{\Part}$.

By use of Lemma~\ref{lem:Stab-proper}, we now have the inclusions $\Stab{V}{D} \leq \Fix{V}{\Part} \leq \Stab{V}{\Part} < V$.  Since $\Stab{V}{D}$~is a maximal subgroup of~$V$ as observed in Corollary~\ref{cor:Stab-tailclass}, it follows that the first three subgroups are equal and we have completed the proof of Theorem~\ref{thm:quasicycle}. \qed

\spc

Finally, Theorem~\ref{thm:finitetypes} tells us that a simple type system with finitely many types is either one considered in Theorem~\ref{thm:branchtype} or in Theorem~\ref{thm:quasicycle}.  Hence we have established our main result about stabilizers of finite simple type systems:

\FiniteSimpleThm

\section{Description of the stabilizers of simple type systems}
\label{sec:descriptions}

As proven in Section~\ref{sec:simpletypes}, every finite simple type system is either nuclear, atomic binuclear, or atomic quasinuclear.  Furthermore, as observed in the previous section, the stabilizer subgroup associated to any finite simple type system is a maximal subgroup of~$V$. In this section we describe the maximal subgroups associated with each of the three types.

\subsection{Atomic multinuclear type systems}

As we shall see, any atomic multinuclear type system with two or more nuclei has maximal stabilizer, even though such a type system is simple if and only if it has exactly two nuclei.

\begin{prop}
\label{prop:AtomicMultinuclearPartition}
Let $\Part$~be an atomic multinuclear type system with $n$~nuclei where $n \geq 2$.  Then  there exists a partition\/ $\{E_{1},E_{2},\dots,E_{n}\}$ of\/~$\Cant$ into clopen sets such that $\Stab{V}{\Part} = \Stab{V}{\{E_{1},E_{2},\dots,E_{n}\}}$.
\end{prop}

\begin{proof}
Let $\{Q_{1}\}$,~$\{Q_{2}\}$, \dots,~$\{Q_{n}\}$ be the nuclei of~$\Part$ and, for each~$i$, let $E_{i}$~be the union of all cones $\cone{\alpha}$ for which $\alpha$~has type~$Q_{i}$.  Since each~$\{Q_{i}\}$ is child-closed and all but finitely many addresses have one of the types~$Q_{i}$, the collection $\{E_{1},E_{2},\dots,E_{n}\}$ forms a partition of~$\Cant$. Moreover, since each of these sets is open (being a union of cones) they must all be clopen.  Note then that, since the associated equivalence relation on~$\AddrSet$ is reduced, an address $\alpha \in \AddrSet$ has type~$Q_{i}$ if and only if $\cone{\alpha} \subseteq E_{i}$.

Now if an element $g\in G$ stabilizes this partition, then it permutes the sets $E_{1}$,~$E_{2}$, \dots,~$E_{n}$ according to some permutation $\sigma \in S_{n}$.  Then, for all $\alpha \in \AddrSet$, if $\alpha$~has one of the types~$Q_{i}$ then $\alpha^{g}$~is defined and has type~$Q_{i^{\sigma}}$, and it follows that $g \in \Stab{V}{\Part}$.  Conversely, if $g \in \Stab{V}{\Part}$, there is a permutation $\sigma \in S_{n}$ such that, for each~$i$, the element~$g$ maps all but finitely many addresses of type~$Q_{i}$ to addresses of type~$Q_{i^{\sigma}}$  Then $g$~maps all but finitely many cones in~$E_{i}$ into~$E_{i^{\sigma}}$, and it follows easily that $(E_{i})^{g}= E_{i^{\sigma}}$ for each~$i$ and hence $g$~stabilizes the partition.
\end{proof}

\begin{rem}
If $\{E_{1},E_{2},\dots,E_{n}\}$ is any partition of~$\Cant$ into clopen sets, then the stabilizer of $\{E_{1},E_{2},\dots,E_{n}\}$ is in fact the stabilizer of a uniquely determined atomic multinuclear type system~$\Part$.  Specifically, define an equivalence relation~$\sim$ on~$\AddrSet$ by $\alpha \sim \beta$ if and only if
\[
\set{\gamma \in \AddrSet}{\cone{\alpha\gamma} \subseteq E_{i}} = \set{\gamma \in \AddrSet}{\cone{\beta\gamma} \subseteq E_{i}}
\]
for each~$i$.  Then $\sim$~determines a type system~$\Part$ which is atomic multinuclear, and $\Stab{V}{\Part}$ is precisely the stabilizer of $\{E_{1},E_{2},\dots,E_{n}\}$.
\end{rem}

Since the stabilizer of a finite simple type system is maximal, it follows from Proposition~\ref{prop:AtomicMultinuclearPartition} that the stabilizer of a partition $\{E_{1},E_{2},\dots,E_{n}\}$ of the Cantor set into clopen subsets is maximal for~$n=2$.  The following proposition extends this result to arbitrary values of~$n$.

\begin{prop}
Let $\{E_{1},E_{2},\dots,E_{n}\}$ be a partition of\/~$\Cant$ into clopen sets, where $n\geq 2$.  Then the stabilizer $\Stab{V}{\{E_{1},E_{2},\dots,E_{n}\}}$ is a maximal subgroup of\/~$V$ isomorphic to the permutational wreath product $V\operatorname{wr} S_n\cong V^{n} \rtimes S_{n}$, where $S_{n}$~denotes the symmetric group of degree~$n$.
\end{prop}

\begin{proof}
Let $G$ be the stabilizer of the partition.  To prove that $G$ is maximal, let $H$ be a subgroup of $V$ that properly contains~$G$, and let $h \in H \setminus G$.  Then there exist $E_{i}$,~$E_{j}$ and~$E_{k}$ with $j \neq k$ such that $(E_i)^h$ intersects both $E_j$ and~$E_k$. Thus we can find disjoint cones $\cone{\alpha}, \cone{\beta} \subseteq E_i$ such that $\alpha^h$ and $\beta^h$ are both defined, $\cone{\alpha^h}$ is properly contained in~$E_j$, and $\cone{\beta^h}$ is properly contained in~$E_k$.  Since $\swap{\alpha}{\beta}$ stabilizes $\{E_{1},E_{2},\dots,E_{n}\}$, the transposition $\swap{\alpha^h}{\beta^h}=\swap{\alpha}{\beta}^h$ lies in $H$.  Conjugating further by elements of $G$, we find that $H$ contains all transpositions $\swap{\gamma}{\delta}$ for which $\cone{\gamma}$~is properly contained in some~$E_r$ and $\cone{\delta}$~is properly contained in some~$E_s$ with $r\neq s$.   Note that $H$~also contains all transpositions whose cones are contained in the same $E_{r}$.  It follows that $H$ contains all transpositions $\swap{\gamma}{\delta}$ for which $\gamma$ and $\delta$ are sufficiently long, and therefore $H = V$.  This proves that $G$~is maximal.

Finally, observe that for each $E_i$ the subgroup of elements of $V$ that are supported on $E_i$ is isomorphic to~$V$.  It follows that the subgroup of all elements of $V$ that map each $E_i$ to itself is isomorphic to $V^n$, so $G$ is isomorphic to the given semidirect product.
\end{proof}

\begin{example}
\label{ex:AtomicMultinuclearTypeSystem}
Figure~\ref{fig:AtomicMultinuclearTypeSystem} shows the label diagram for an atomic multinuclear type system with nuclei $\{\mathsf{Q}\}$, $\{\mathsf{R}\}$, and $\{\mathsf{S}\}$, where $\mathsf{A}$ is the label for the empty word.  Note that the word $1$ has label~$\mathsf{B}$, all words that begin with $0$ have label $\mathsf{Q}$, all words that begin with $10$ have label $\mathsf{R}$, and all words that begin with $11$ have label $\mathsf{S}$.  If $\Part$ is the corresponding type system, then $\Stab{V}{\Part}$ is the stabilizer in $V$ of the partition $\{0\Cant,10\Cant,11\Cant\}$.
\begin{figure}
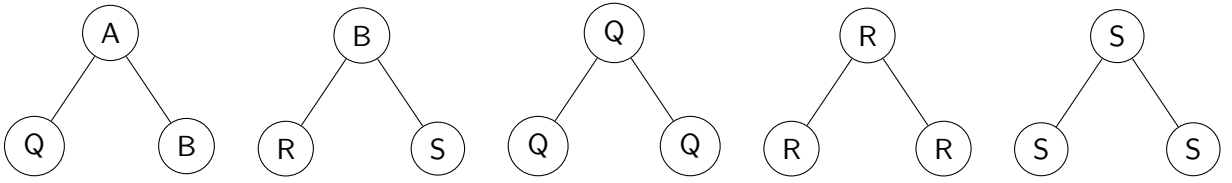

\begin{center}
  \labelcaret{A}{Q}{B}
  \hfill
  \labelcaret{B}{R}{S}
  \hfill
  \labelcaret{Q}{Q}{Q}
  \hfill
  \labelcaret{R}{R}{R}
  \hfill
  \labelcaret{S}{S}{S}
\end{center}
\caption{Label diagram for the partition $\{0\Cant,10\Cant,11\Cant\}$, which has stabilizer isomorphic to $V^3\rtimes S_3$.}
  \label{fig:AtomicMultinuclearTypeSystem}
\end{figure}
\end{example}

We close this subsection by demonstrating that distinct simple multinuclear type systems give rise to distinct maximal subgroups.  It is easy to use this lemma (by creating distinct infinite simple nuclear type systems) to show that there are uncountably-many maximal subgroups of $V$, however, in Section \ref{sec:uncountIso} we will obtain the stronger result that in fact there are uncountably many isomorphism classes of maximal subgroups of $V$.

\begin{lemma}
\label{lem:distinctStabs}
Let $\Part$~and~$\Part[Q]$ be distinct type systems on~$\AddrSet$ that are both multinuclear and simple.  Then $\Stab{V}{\Part} \neq \Stab{V}{\Part[Q]}$.
\end{lemma}

\begin{proof}
Let $\Part$~and~$\Part[Q]$ be two type systems on~$\AddrSet$ both of which are multinuclear and simple such that $\Stab{V}{\Part} = \Stab{V}{\Part[Q]}$.  Let $\sim$~and~$\approx$ be the equivalence relations on~$\AddrSet$ that define $\Part$~and~$\Part[Q]$, respectively, and take the integer~$N$ to be greater than the stable depths of both $\Part$~and~$\Part[Q]$.  Consider two addresses $\alpha$~and~$\beta$ with $\alpha \sim \beta$.  Let $\eta$~be a word of length~$N$.  Since $\sim$~is coherent, $\alpha\eta \sim \beta\eta$.  By Lemma~\ref{lem:nuke-incomparables}, there exists an address~$\gamma$ that is incomparable with both $\alpha\eta$~and~$\beta\eta$ and such that $\gamma \sim \alpha\eta \sim \beta\eta$ and $\length{\gamma} \geq N$.  Therefore $\swap{\alpha\eta}{\gamma}$ and $\swap{\beta\eta}{\gamma}$ both belong to~$\Fix{V}{\Part}$ and hence to $\Stab{V}{\Part[Q]}$.

Now by an application of Lemma~\ref{lem:nuke-incomparables}, this time to~$\approx$, there exists~$\delta$ that is incomparable with $\alpha\eta$ and~$\gamma$ and such that $\delta \approx \gamma$.  Since $g = \swap{\alpha\eta}{\gamma} \in \Stab{V}{\Part[Q]}$, we deduce $\delta^{g} \approx \gamma^{g}$; that is, $\delta \approx \alpha\eta$.  It follows that $\alpha\eta \approx \gamma$.  Similarly, one establishes $\beta\eta \approx \gamma$.  We deduce that $\alpha\eta \approx \beta\eta$ for all words~$\eta$ of length~$N$ and it follows that $\alpha \approx \beta$ since $\approx$~is reduced.  By symmetry, we now conclude that the type systems $\Part$~and~$\Part[Q]$ are the same.  This establishes the lemma.
\end{proof}

\subsection{Atomic quasinuclear type systems}

As we have seen, atomic quasinuclear type systems can be classified into non-branching and branching types.  As proven in Theorem~\ref{thm:quasicycle}, the stabilizer of a simple, non-branching, atomic quasinuclear type system is the same as the stabilizer of some finite set of points in $\Cant$ that have the same tail class.  The following example illustrates such type systems.

\begin{example}
\label{ex:NonBranchingAtomicQuasinuclearTypeSystem}
Figure~\ref{fig:NonBranchingAtomicQuasinuclearTypeSystem} shows the label diagram for a simple, non-branching, atomic quasinuclear type system with nucleus $\{\mathsf{R}\}$, where $\mathsf{A}$ is the label for the empty word.  Here:
\begin{itemize}
    \item The empty word $\emptyword$ is the only address with label~$\mathsf{A}$.
    \item Every address of the form $(01)^n0$ or $(10)^n$ for $n\geq 0$ has label~$\mathsf{B}$.
    \item Every address of the form $(01)^n$ or $(10)^n1$ for $n\geq 0$ has label~$\mathsf{C}$.
    \item Every other address has label $\mathsf{R}$.
\end{itemize}
If $\Part$ is the corresponding type system, then $\Stab{V}{\Part}$ is precisely the stabilizer in $V$ of the two-point set $\bigl\{\overline{01},\overline{10}\bigr\}$.
\begin{figure}
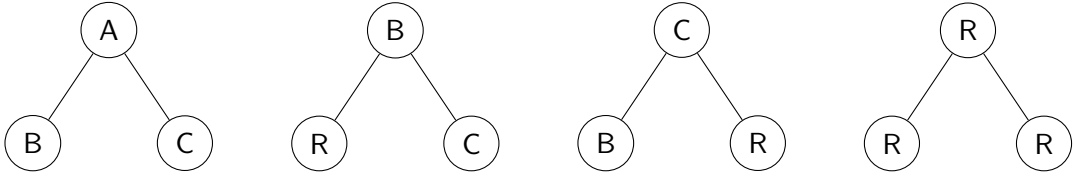

\begin{center}
  \labelcaret{A}{B}{C}
  \qquad
  \labelcaret{B}{R}{C}
  \qquad
  \labelcaret{C}{B}{R}
  \qquad
  \labelcaret{R}{R}{R}
\end{center}
\caption{Label diagram for a type system whose stabilizer is the stabilizer in $V$ of the two-point set $\bigl\{\overline{01},\overline{10}\bigr\}$.}
  \label{fig:NonBranchingAtomicQuasinuclearTypeSystem}
\end{figure}
\end{example}

The branching case is slightly more complicated.  For the following proposition, recall that an action of a group $G$ on compact metric space $X$ is \emph{minimal} if the $G$\nbd orbit of each point in~$X$ is dense in~$X$.    Equivalently (and explaining the name), an action is minimal if the  only $G$\nbd invariant closed subsets are~$\emptyset$ and $X$.

\begin{prop}\label{prop:BranchinQuasinuclearK}
Let $\Part$ be a simple, branching, atomic quasinuclear type system.  Then there exists a non-empty, nowhere dense, perfect subset $K$ of\/ $\Cant$ such that\/ $\Stab{V}{\Part}= \Stab{V}{K}$.  Furthermore, the action of\/ $\Stab{V}{K}$ on $K$ is minimal. 
\end{prop}
\begin{proof}
Let $\{R\}$ be the nucleus for $\Part$, let $U$ be the union of all cones $\cone{\alpha}$ for which $\alpha$~has type~$R$, and let $K = \Cant\setminus U$.  Note that $U$~is open since it is a union of cones, and thus $K$~is closed.  Furthermore, since every cone $\cone{\beta}$ contains a cone~$\cone{\alpha}$ for which $\alpha$~has type~$R$, the open set $U$~is dense in~$\Cant$, and therefore $K$~is nowhere dense.  Finally, since $\Part$ is branching there is a cone $\cone{\gamma}$ for which $\gamma$~does not have type~$R$, and any such cone contains at least two points from~$K$.  Therefore $K$ is non-empty and {\bf has no isolated points}.

To prove that $\Stab{V}{\Part}=\Stab{V}{K}$, note first that $R$ is the only terminal vertex in the type graph for~$\Part$.  Let $g \in \Stab{V}{\Part}$.  Since $\{R\}$~is child-closed and every address~$\alpha$ has some descendant~$\alpha\eta$ of type~$R$, it follows that $g$~maps all but finitely many addresses of type~$R$ to addresses of type~$R$.  We conclude that $U^g\subseteq U$, and indeed $U^g=U$ since the same statement holds for~$g^{-1}$. This proves that $g\in\Stab{V}{K}$, and hence $\Stab{V}{\Part}\leq \Stab{V}{K}$.  For the opposite conclusion, observe that $\Stab{V}{K}$ is a proper subgroup of $V$ since it does not contain any transpositions $\swap{\alpha}{\beta}$ for which $\cone{\alpha} \subseteq U$ and $\cone{\beta} \nsubseteq U$. Since $\Stab{V}{\Part}$ is maximal by Theorem~\ref{thm:branchtype}, we conclude that $\Stab{V}{\Part}=\Stab{V}{K}$.

Finally, to prove that the action of $\Stab{V}{K}$ on~$K$ is minimal, let $v,w \in K$ and let $W$ be a neighbourhood of~$w$. Choose a cone~$\cone{\alpha}$ containing~$v$ whose type lies in the terminal set $\Part[Q] \cup \{R\}$ for~$\Part$, and let $\cone{\beta}$~be any cone containing~$w$ which is contained in~$W$.  Then $\beta$~does not have type~$R$, so there exists a cone~$\cone{\gamma}$ contained in~$\cone{\beta}$ such that $\gamma$~has the same type as~$\alpha$.  Then the transposition $\swap{\alpha}{\gamma}$ lies in $\Stab{V}{\Part}$ and maps~$v$ into~$\cone{\beta}$ and hence into~$W$, which proves that the action of $\Stab{V}{K}$ on~$K$ is minimal.
\end{proof}

\begin{rem}
If $\Part$ is a simple atomic quasinuclear type system with nucleus $\{R\}$, then it is not hard to show that the corresponding relation $\sim$ on $\Omega$ must be defined by the formula
\[
\alpha\sim \beta \IFF
\set{\gamma \in \Omega}{\text{$\alpha\gamma$~has type~$R$}} = \set{\gamma \in \Omega}{\text{$\beta\gamma$~has type~$R$}}.
\]
Since any element of $\Stab{V}{\Part}$ maps all but finitely many addresses of type $R$ to addresses of type $R$, it follows that any element of $\Stab{V}{\Part}$ preserves the type of all but finitely many addresses, and hence $\Stab{V}{\Part}=\Fix{V}{\Part}$.
\end{rem}

\begin{example}\label{ex:BranchingTypeSystem}
Figure~\ref{fig:BranchingTypeSystem} shows the label diagram for a simple, branching, atomic quasinuclear type system with nucleus $\{\mathsf{R}\}$, where $\mathsf{A}$ is the label for the empty word.  In this case:
\begin{itemize}
    \item An address has label $\mathsf{A}$ if it lies in $\{00,1\}^*$, i.e., if every maximal string of consecutive $0$'s in the word has even length.
    \item An address has label $\mathsf{B}$ if it lies in $\{00,1\}^*0$, i.e., if it has the form $\alpha 0$ for some word $\alpha$ with label $\mathsf{A}$.
    \item An address has label $\mathsf{R}$ if it starts with $0^{2n+1}1$ or contains $10^{2n+1}1$ for some~$n$.
\end{itemize}
If $\Part$ is the corresponding type system, then $\Stab{V}{\Part}$ is equal to $\Stab{V}{K}$, where $K=\{00,1\}^\omega$ is the set of all infinite concatenations of $00$ and~$1$. That is, $K$ is the set of all infinite words in~$\Cant$ for which every maximal string of consecutive zeros is either infinite or has even length.
\begin{figure}
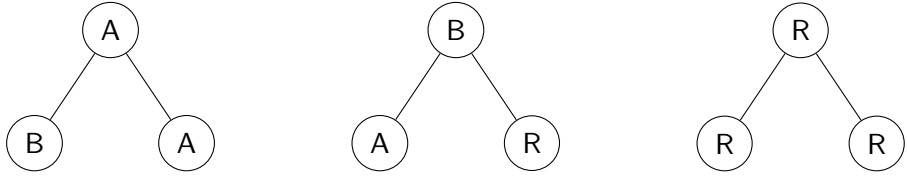

\begin{center}
  \labelcaret{A}{B}{A}
  \qquad\qquad
  \labelcaret{B}{A}{R}
  \qquad\qquad
  \labelcaret{R}{R}{R}
\end{center}
\caption{Label diagram for a type system whose stabilizer is the stabilizer in $V$ of the Cantor subset $K=\{00,1\}^\omega$.}
  \label{fig:BranchingTypeSystem}
\end{figure}
\end{example}

%%%%%%%%%%%%%%%%%%%%%%%%%%%%%%%%%%%%%%%%%%%%%%%%%%%%%%%%%%%
\subsection{Nuclear type systems}

We now turn to the maximal subgroups determined by a nuclear type system.  We shall observe that these subgroups are isomorphic to full groups of subshifts of finite type. This class of Thompson-like groups was introduced by Matsumoto in \cite{Matsumoto} and studied extensively by Matui in \cite{Matui}.

We begin by defining these groups.  Given a finite directed multigraph $\Gamma$ with edge set $\mathcal{E}(\Gamma)$, the corresponding \textit{(one-sided) subshift of finite type} is the pair $(X_\Gamma,\sigma)$, where
\[
X_{\Gamma} = \set{e_{1}e_{2}e_{3}\dots\in \mathcal{E}(\Gamma)^{\mathbb{N}}}{\text{$(e_{1},e_{2},e_{3},\dots)$ is an infinite directed path in $\Gamma$}}
\]
and $\sigma\colon X_{\Gamma} \to X_{\Gamma}$ is the shift map.  Note that $X_{\Gamma}$~inherits a topology as a closed subset of the infinite product~$\mathcal{E}(\Gamma)^{\mathbb{N}}$, where $\mathcal{E}(\Gamma)$ has the discrete topology.  With respect to this topology, $X_{\Gamma}$~is compact, Hausdorff, and totally disconnected.

Let $A_{\Gamma}$~denote the adjacency matrix for~$\Gamma$, i.e., the matrix whose $(i,j)$-th entry is the number of directed edges from vertex~$i$ to vertex~$j$.  The subshift $(X_{\Gamma},\sigma)$ is said to be \textit{irreducible} if $A_{\Gamma}$ is an irreducible matrix, or equivalently if $\Gamma$~is strongly connected.  An irreducible subshift $(X_{\Gamma},\sigma)$ is \textit{branching} if $\Gamma$~is not a directed cycle, or equivalently if the adjacency matrix~$A_{\Gamma}$ is not a permutation matrix.  If $(X_{\Gamma},\sigma)$ is irreducible and branching, then $X_{\Gamma}$~has no isolated points and is therefore a Cantor space.

Given a finite directed path $\alpha = (e_{1},e_{2},\dots,e_{n})$ in~$\Gamma$, the corresponding \emph{cone} $C(\alpha,X_{\Gamma})$ is the set of all points in~$X_{\Gamma}$ that have $e_{1}e_{2}\dots e_{n}$ as a prefix.  Note that $C(\alpha,X_{\Gamma})$~is a clopen subset of~$X_{\Gamma}$ and each point in $C(\alpha,X_{\Gamma})$ has the form~$\alpha\psi$ for some infinite directed path~$\psi$ in~$\Gamma$ whose initial vertex is the same as the terminal vertex of~$\alpha$.  If $\alpha$ and $\beta$ are two finite directed paths in~$\Gamma$ with the same terminal vertex, there is a \emph{prefix-replacement homeomorphism} $C(\alpha,X_{\Gamma}) \to C(\beta,X_{\Gamma})$ that maps each point~$\alpha \psi$ in $C(\alpha,X_{\Gamma})$ to the corresponding point~$\beta\psi$ in $C(\beta,X_{\Gamma})$.

\begin{defn}
Let $\Gamma$ be a finite directed multigraph and let $(X_{\Gamma},\sigma)$ be the resulting subshift of finite type.
\begin{enumerate}
\item A homeomorphism $f \colon U \to V$ between open subsets of~$X_\Gamma$ is \textit{Thompson-like} if for every point $p \in U$ there exists a cone $C(\alpha,X_{\Gamma}) \subseteq U$ containing~$p$ and a cone $C(\beta,X_{\Gamma})$ containing~$p^{f}$ such that~$f$ maps $C(\alpha,X_{\Gamma})$ to $C(\beta,X_{\Gamma})$ by a prefix-replacement homeomorphism.
\item The \emph{full group}~$V(\Gamma)$ over~$\Gamma$ is the group of all Thompson-like homeomorphisms $X_{\Gamma} \to X_{\Gamma}$.
\item More generally, if $E \subseteq X_{\Gamma}$ is a non-empty clopen set, the \emph{restricted full group}~$V(\Gamma,E)$ is the group of all Thompson-like homeomorphisms $E \to E$.
\end{enumerate}
\end{defn}

Since $X_{\Gamma}$~is compact, a homeomorphism $f \colon X_{\Gamma} \to X_{\Gamma}$ lies in~$V(\Gamma)$ if and only if there exist partitions $C(\alpha_{1},X_{\Gamma})$, $C(\alpha_{2},X_{\Gamma})$, \dots, $C(\alpha_{n},X_{\Gamma})$ and $C(\beta_{1},X_{\Gamma})$, $C(\beta_{2},X_{\Gamma})$, \dots, $C(\beta_{n},X_{\Gamma})$ of~$X_{\Gamma}$ into cones such that $f$~maps each $C(\alpha_{i},X_{\Gamma})$ to $C(\beta_{i},X_{\Gamma})$ by a prefix-replacement homeomorphism.  A similar statement holds for elements of $V(\Gamma,E)$ for any clopen subset~$E$ of~$X_{\Gamma}$.  Note also that $V(\Gamma,E)$ and $V(\Gamma,E')$ are isomorphic whenever there exists a Thompson-like homeomorphism $E \to E'$.

The theory of the groups $V(\Gamma)$ and $V(\Gamma,E)$ was developed by Matui in~\cite{Matui}, where $V(\Gamma)$~is viewed as the topological full group of the \'{e}tale groupoid associated to $(X_{\Gamma},\sigma)$ and $V(\Gamma,E)$~is the topological full group of the restriction of this \'{e}tale groupoid to~$E$.  Matui proves the following results about such groups.

\begin{thm}[Matui~\cite{Matui}]
\label{thm:Matui}
Let $(X_{\Gamma},\sigma)$ be an irreducible, branching subshift of finite type and let $G=V(\Gamma,E)$ for some non-empty clopen set $E \subseteq X_{\Gamma}$.  Then:
\begin{enumerate}
\item The derived subgroup~$[G,G]$ is simple and finitely generated.
\item The group~$G$ is finitely presented.  Indeed, it has type~$\mathrm{F}_{\infty}$.
\item \label{part:abelianization}
The abelianization $G/[G,G]$ is isomorphic to $(H_{0} \otimes \Zint_{2}) \oplus H_{1}$ where
\[
H_{0} = \coker (I - A_{\Gamma})
\AND
H_{1} = \ker (I - A_{\Gamma}).
\]
Here $I$~is the $n \times n$~identity matrix and we regard $I - A_{\Gamma}$ as the homomorphism $\Zint^{n} \to \Zint^{n}$ determined by the right action of the matrix on row vectors.
\end{enumerate}
\end{thm}

For part~\ref{part:abelianization} above, note that it follows from the Rank--Nullity Theorem that $H_{1}$~is isomorphic to the torsion-free part of~$H_{0}$.  Note also that $H_{0} \otimes \Zint_{2} \cong H_{0}/2H_{0}$ is trivial if and only if $H_{0}$~is finite of odd order.  We conclude that $G$~is simple if and only if $H_{0}$~is finite of odd order and $G$~is virtually simple if and only if $H_{0}$~is finite.

\begin{rem}
Our definitions and notation are very different from those used by Matui in \cite{Matui}, though the groups and conclusions are the same.  Matui uses a subshift of finite type to define an \'{e}tale groupoid, and then considers the topological full group associated to this \'etale groupoid.  In particular, our group $V(\Gamma)$ is denoted $\llbracket G\rrbracket$ in \cite{Matui}, our group $V(\Gamma,E)$ is denoted $\llbracket G|Y \rrbracket$, and the corresponding derived subgroups are denoted $D(\llbracket G \rrbracket)$ and $D(\llbracket G|Y \rrbracket)$, respectively.
\end{rem}

We now relate the restricted full groups $V(\Gamma,E)$ to nuclear type systems.

\begin{defn}
If $\Part$~is a nuclear type system, we define the \emph{nucleus graph} for~$\Part$ to be the induced subgraph of the type graph for~$\Part$ whose vertices correspond to the types in the nucleus. 
\end{defn}
By the definition of nuclear type systems, this nucleus graph is always strongly connected.

\begin{thm}
\label{thm:FixIsMatui}
Let $\Part$~be a finite nuclear type system, let\/ $\Delta$~be its nucleus graph and let\/ $(X_{\Delta},\sigma)$~be the corresponding subshift of finite type.  Then there exists a non-empty clopen set $E \subseteq X_{\Delta}$ such that\/ $\Fix{V}{\Part} \cong V(\Delta,E)$ and the corresponding $H_{0}$ is isomorphic to~$\sgp{\Part}$.
\end{thm}

\begin{proof}
As usual, let $\graph{\Part}$~be the type graph of~$\Part$, so that $\Delta \subseteq \graph{\Part}$, and let $A$~be the type of the empty string. Let $E_{A}$~be the clopen set in~$X_{\graph{\Part}}$ consisting of all infinite directed paths with initial vertex~$A$, and note each such path is determined by the $0$--$1$ labeling on its edges.  The $0$--$1$ labeling on the edges of~$\graph{\Part}$ defines a map $\pi \colon X_{\graph{\Part}} \to \Cant$ which associates to each directed path its sequence of binary labels.  The map~$\pi$ restricts to a homeomorphism $E_{A} \to \Cant$.  Note that $\pi$ maps each cone $C(\textbf{e},X_{\graph{\Part}})$ to the cone $\alpha\Cant$, where $\alpha$ is the label on the finite path $\textbf{e}$, and therefore $\pi$ conjugates $V(\graph{\Part},E_{A})$ to $\Fix{V}{\Part}$.  Thus $\Fix{V}{\Part}$ is isomorphic to $V(\graph{\Part},E_{A})$.

Now, we can view~$X_{\Delta}$ as the subspace of~$X_{\graph{\Part}}$ consisting of all infinite directed paths whose initial vertex lies in the nucleus.  Note that  $C(\alpha,X_{\Delta})=C(\alpha,X_{\graph{\Part}})$ for any finite path $\alpha$ starting at a vertex in the nucleus.  If $E\subseteq X_{\Delta}$ is a clopen set, it follows that Thompson-like homeomorphisms of $E$ with respect to $X_\Delta$ are the same as Thompson-like homeomorphisms of $E$ with respect to $X_{\graph{\Part}}$, and therefore $V(\Delta,E) = V(\graph{\Part},E)$.

We now show that there exists a clopen set $E \subseteq X_{\Delta}$ which maps to~$E_{A}$ by a Thompson-like homeomorphism. Since $\Part$~is nuclear, we can partition~$\Cant$ into finitely many cones $\cone{\alpha_{1}}$,~$\cone{\alpha_{2}}$, \dots,~$\cone{\alpha_{k}}$ such that the type of each~$\alpha_{i}$ lies in the nucleus.  Then each~$\alpha_{i}$ corresponds to a path~$\gamma_{i}$ in~$\graph{\Part}$ from~$A$ to some vertex~$v_{i}$ in~$\Delta$.  The list $v_{1}$,~$v_{2}$, \dots,~$v_{k}$ of these terminal vertices may include repetition, but since $\Delta$~is strongly connected and is not a cycle, there exists an $m > 0$ so that each vertex of~$\Delta$ is the endpoint of at least $k$~distinct directed paths in~$\Delta$ of length~$m$.  It follows that we can find $k$~distinct directed paths $\delta_{1}$,~$\delta_{2}$, \dots,~$\delta_{k}$ in~$\Delta$ of length~$m$ such that each~$\delta_{i}$ ends at the corresponding vertex~$v_{i}$.  Let
\[
E = C(\delta_{1},\Delta) \cup C(\delta_{2},\Delta) \cup \dots \cup C(\delta_{k},\Delta)
\]
and note that these cones are pairwise disjoint since $\delta_{1}$,~$\delta_{2}$, \dots,~$\delta_{k}$ are distinct and all have the same length.  Then the homeomorphism $f \colon E \to E_A$ that maps each $C(\delta_{i},\Delta)$ to $C(\gamma_{i},\graph{\Part})$ by prefix replacement is Thompson-like and therefore $V(\Delta,E) \cong V(\graph{\Part},E_{A}) \cong \Fix{V}{\Part}$.

Finally, we know from Lemma~\ref{lem:multinuke-sgp}\ref{i:multinuke-compltype} that $\sgp{\Part}$~is an abelian group.  Moreover, if $\Part[N]$~denotes the set of types in the nucleus of~$\Part$, it is easy to see that
\[
\sgp{\Part} \cong \presentation{ t \in \Part[N] }{\text{$t = t_{0} + t_{1}$ for each $t \in \Part[N]$}}.
\]
Let $A_{\Gamma}$~be the adjacency matrix for~$\Gamma$, and for each type~$t$ let $e[t]$~be the corresponding standard basis row vector.  Note that for each type~$t$, the corresponding row of~$A_{\Gamma}$ is $e[t_{0}] + e[t_{1}]$, so the corresponding row of $I - A_{\Gamma}$ is $e[t]-e[t_0]-e[t_1]$.  Then $H_{0} = \coker (I-A_{\Gamma})$ has precisely the same presentation as $\sgp{\Part}$, which proves that they are isomorphic.
\end{proof}

Combining this with Theorem~\ref{thm:Matui} and the fact that $\Fix{V}{\Part}$ has finite index in $\Stab{V}{\Part}$ when $\Part$ is finite (by Lemma~\ref{lem:P*-action}), we obtain the following:

\begin{cor}
\label{cor:ConclusionsFromMatui}
Let $\Part$~be a finite nuclear type system.  Then:
\begin{enumerate}
\item The derived subgroup of\/ $\Fix{V}{\Part}$ is simple and finitely generated.
\item The groups\/ $\Fix{V}{\Part}$ and\/ $\Stab{V}{\Part}$ are finitely presented.  Indeed, they have type\/~$\mathrm{F}_{\infty}$.
\item \label{i:Fix-simple}
$\Fix{V}{\Part}$~is simple if and only if $\sgp{\Part}$~is finite of odd order, and\/ $\Fix{V}{\Part}$ and\/ $\Stab{V}{\Part}$ are virtually simple if and only if $\sgp{\Part}$~is finite.
\end{enumerate}
\end{cor}

\begin{example}\label{ex:HigmanThompson}
It is well-known that the class of restricted full groups $V(\Gamma,E)$ for irreducible, branching subshifts of finite type includes the Higman--Thompson groups $V_{n,r}$ as special cases.  Indeed, it is not difficult to find simple nuclear type systems whose corresponding maximal subgroups are Higman--Thompson groups.

For example, given any $n\geq 3$, let $\Part$ be the type system with types labeled $1,\ldots n-1$, where:
\begin{itemize}
    \item $(i)_0=1$ and $(i)_1=i-1$ for all $i>1$.
    \item $(1)_0=n-1$ and $(1)_1=1$.
\end{itemize}
A label diagram for the $n=5$ case is shown in Figure~\ref{fig:HigmanThompson}.  Clearly $\Part$ is nuclear, and it is not hard to check that $\Part$ is simple, so $\Stab{V}{\Part}$ is a maximal subgroup of~$V$.  If the empty word has type~$r$, where $1\leq r\leq n-1$, then it turns out that $\Stab{V}{\Part}=\Fix{V}{\Part}\cong V_{n,r}$.

\begin{figure}[hbt!]
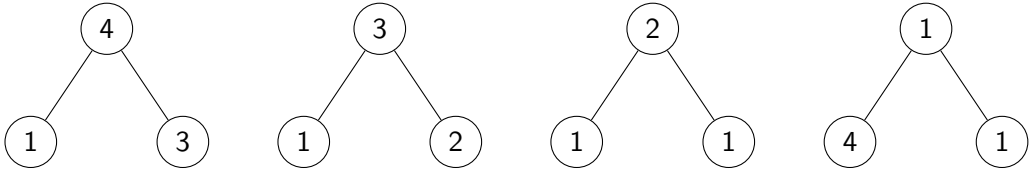

\begin{center}
  \labelcaret{4}{1}{3}\qquad
  \labelcaret{3}{1}{2}\qquad
  \labelcaret{2}{1}{1}\qquad
  \labelcaret{1}{4}{1}
\end{center}
  \caption{Label diagram for the $n=5$ case of the simple nuclear type system in Example~\ref{ex:HigmanThompson}.  Depending on the type of the empty word, the  corresponding maximal subgroup of $V$ is isomorphic to one of the Higman--Thompson groups $V_{5,4}$, $V_{5,3}$, $V_{5,2}$, or $V_{5,1}$.}
  \label{fig:HigmanThompson}
\end{figure}

One way to prove this assertion is to use a theorem of Matsumoto \cite{Matsumoto} on isomorphisms of the restricted full groups $V(\Gamma,E)$.  Note first that if $\Gamma$ is a directed graph with one vertex and $n$ loops at~$\Gamma$, then the full group $V(\Gamma)$ is isomorphic to $V_{n,1}$, and hence different clopen sets $E\subseteq X_\Gamma$ give all of the groups~$V_{n,r}$.  Matsumoto proved that if $(X_{\Gamma},\sigma)$ and $(X_\Delta,\sigma)$ are any two irreducible, branching subshifts of finite type such that $\coker(I-A_{\Gamma}) \cong \coker(I-A_\Delta)$ and $\det(I-A_\Gamma) = \det(I-A_\Delta)$, then each group $V(\Gamma,E)$ is isomorphic to $V(\Delta,E')$ for certain clopen sets~$E'$ that Matsumoto described.  If $\Gamma$ is the single-vertex graph for $V_{n,1}$ and $\Delta$ is the type graph for the type system defined above, then it is easy to check that $\coker(I-A_{\Gamma}) \cong \coker(I-A_\Delta) \cong \Zint_{n-1}$ and $\det(I-A_\Gamma) = \det(I-A_\Delta) = -(n-1)$, and applying Matsumoto's criterion for the appropriate choice of $E'$ yields the desired result.
\end{example}

\begin{example}
\label{ex:SimpleMaximal}
Here we give an example of a finitely presented, simple, maximal subgroup of~$V$ which is not isomorphic to the derived subgroup of any of the Higman--Thompson groups~$V_{n,r}$. Note that it is also not isomorphic to $T$, since it contains the symmetric group~$S_{n}$ for every $n$.

Let $\Part$~be the type system determined by the label diagram in Figure~\ref{fig:SimpleMaximal}, where the empty word has label~$\mathsf{A}$.  Note that Type~$\mathsf{A}$ is the only type that occurs exactly twice as a child of a type, so any symmetry of this type system must fix $\mathsf{A}$. It follows that this type system has no non-trivial symmetries, and therefore $\Fix{V}{\mathcal{P}} = \Stab{V}{\mathcal{P}}$. A straightforward exhaustive check shows that $\Part$ is simple, so $\Fix{V}{\Part}$ is a maximal subgroup of~$V$ by Theorem~\ref{thm:branchtype}.  The corresponding group
\[
\sgp{\Part} = \presentation{ A,B,C,D,E }{ A=B+C, \; B=D+C, \; C=B+E, \; D=C+A, \; E=A+B }
\]
is isomorphic to $\Zint_{3} \oplus \Zint_{3}$, so $\Fix{V}{\Part}$ is a simple group of type~$\mathrm{F}_\infty$ by Corollary~\ref{cor:ConclusionsFromMatui}\ref{i:Fix-simple}.  It follows from Theorem~3.10 in Matui~\cite{Matui} that if two groups $V(\Gamma,E)$ and $V(\Gamma',E')$ have isomorphic derived subgroups then their corresponding $H_{0}$~groups must be isomorphic as well.  Since the $H_{0}$~groups for the Higman--Thompson groups $V_{n,r}$ are cyclic of order~$n-1$ (since the corresponding $\Gamma$ has one vertex with $n$ loops, with different values of $r$ corresponding to different clopen sets $E\subseteq X_\Gamma$), it follows that $\Fix{V}{\Part}$ is not isomorphic to the derived subgroup $[V_{n,r},V_{n,r}]$ of any of the Higman--Thompson groups~$V_{n,r}$.

\begin{figure}[hbt!]
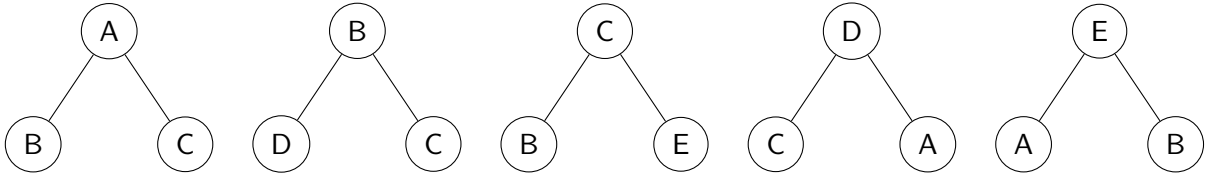
 
\begin{center}
  \labelcaret{A}{B}{C}\quad
  \labelcaret{B}{D}{C}\quad
  \labelcaret{C}{B}{E}\quad
  \labelcaret{D}{C}{A}\quad
  \labelcaret{E}{A}{B}
\end{center}
  \caption{Label diagram for the simple nuclear type system in Example~\ref{ex:SimpleMaximal}.  The corresponding maximal subgroup of $V$ is finitely presented and simple but is not isomorphic to the derived subgroup of any Higman--Thompson group~$V_{n,r}$.}
  \label{fig:SimpleMaximal}
\end{figure}

\end{example}

\begin{example}
\label{ex:InfiniteAbelianization}
Let $\Part$~be the type system determined by the label diagram in Figure~\ref{fig:InfiniteAbelianization}, where the empty word has label~$\mathsf{A}$.  Again, it is easy to observe that $\Part$~is a simple nuclear type system and so $\Fix{V}{\Part}$~is a maximal subgroup of~$V$.  In this case, the group
\[
\sgp{\Part} = \presentation{ A,B,C }{ A=A+B, \; B=A+C, \; C=B+C }
\]
is infinite cyclic, so the abelianization of~$\Fix{V}{\Part}$ is isomorphic to $\Zint_{2} \oplus \Zint$.  In particular, $\Fix{V}{\Part}$~is not virtually simple, though by Theorem~\ref{thm:Matui} its derived subgroup is simple and finitely generated.

\begin{figure}[hbt!]
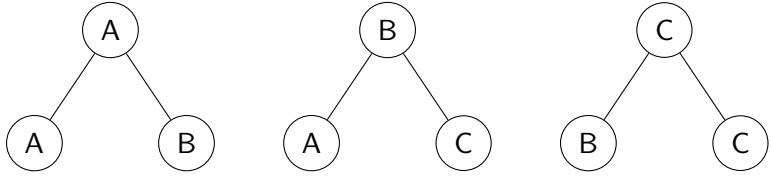
 
\begin{center}
  \labelcaret{A}{A}{B}\qquad
  \labelcaret{B}{A}{C}\qquad
  \labelcaret{C}{B}{C}
\end{center}
  \caption{Label diagram for the simple nuclear type system in Example~\ref{ex:InfiniteAbelianization}.  The corresponding maximal subgroup of $V$ has a simple derived subgroup of infinite index.}
  \label{fig:InfiniteAbelianization}
\end{figure}

\end{example}

\section{Uncountably many distinct maximal subgroups.}\label{sec:uncountIso}

In this section, we shall show that uncountable many maximal subgroups arise as the stabilizers of simple nuclear (albeit infinite) type systems, and with in fact uncountably many isomorphism types of these groups being represented. Furthermore, these constructed groups are not found as point stabilizers (e.g., as in \cite{Savchuk}); they do not even stabilize any finite set of points.  

\subsection{Uncountably many isomorphism types}

An essential ingredient for recognising when two simple multinuclear type systems $\Part_1$ and $\Part_2$ on~$\AddrSet$ induce isomorphic groups $\Fix{V}{\Part_1}$ and $\Fix{V}{\Part_2}$ is Rubin's Theorem~\cite{Rubin}.  We provide some background to this famous theorem and explain how it applies in our context.  We follow the presentation of Rubin's Theorem as it appears in~\cite{BEM}.

Let $X$ be a topological space and $G$~be a group acting on $X$.  For $g\in G$ define the \emph{support of~$g$} to be the set $\supt{g} = \set{x\in X}{xg\neq x}$.
If the action of $G$ on $X$ is faithful we say the action is a \emph{Rubin action} if
\begin{enumerate}[label=(R\arabic*)]
\item \label{pt:goodSpace} $X$ is locally compact, Hausdorff, and has no isolated points, and
\item \label{pt:locDense} for each open set $U \subseteq X$ and each $p \in U$, the closure of the orbit of $p$ under the group $G_U = \set{g \in G }{ \supt{g} \subseteq U}$ contains a neighbourhood of $p$.
\end{enumerate}
If the action satisfies Condition~\ref{pt:locDense} we say the action is \emph{locally dense} following the language of Brin from \cite{brinHigherV}.  This condition is equivalent to Rubin's original condition that no point of $U$ has a nowhere dense orbit under the action of $G_U$.

In any case, we can now state the formulation of Rubin's Theorem from \cite{BEM}.

\begin{thm}[Rubin's Theorem]
If a group $G$ has a Rubin action on two spaces $X$ and $Y$, then there is a $G$-equivariant homeomorphism $\phi:X\to Y$.
\end{thm}

We begin with a foundational observation.

\begin{lemma}
\label{lem:FixMultinuclearRubin}
If $\Part$ is a multinuclear type system on~$\AddrSet$ then $\Fix{V}{\Part}$ acts on $\Cant$ with a Rubin action.
\end{lemma}

\begin{proof}
Firstly, note that Cantor space satisfies the conditions of Condition~\ref{pt:goodSpace} in the definition of a Rubin action.

Secondly, we now show that the action of $G = \Fix{V}{\Part}$ satisfies Condition~\ref{pt:locDense} of the definition of a Rubin action.  Let $U$~be an open neighbourhood of some point $u \in \Cant$.  Let $\alpha$~be a sufficiently long prefix of~$u$ such that $\length{\alpha} \geq \sd{\Part}$ and $\cone{\alpha} \subseteq U$.  Then $\alpha$~has type in some nucleus~$\Part[Q]$ of~$\Part$.  Let $w \in \cone{\alpha}$ with $w \neq u$ and let $N$~be an arbitrary positive integer.  Take $\zeta$~to be a sufficiently long prefix of~$w$ such that $\alpha \prefix \zeta$, \ $\zeta$~is not a prefix of~$u$ and $\length{\zeta} \geq N$.  Then let $\beta$~be such that $\alpha \prefix \beta \prefix u$ with $\beta \perp \zeta$.  As both addresses $\beta$~and~$\zeta$ have $\alpha$~as prefix, they must both have type in the same nucleus~$\Part[Q]$.  Hence there exists a word~$\gamma$ such that $\zeta\gamma$~has the same type as~$\beta$.  Then $g = \swap{\beta}{\zeta\gamma} \in G$ and, as both addresses have $\alpha$~as a prefix, $g \in G_{\cone{\alpha}} \subseteq G_{U}$.  The image of~$u$ under~$g$ has~$\zeta$ as a prefix and thus the orbit of~$u$ under~$G_{U}$ contains an element that shares a prefix of length~$N$ with~$w$.  Since $N$~is arbitrary, we conclude that $w$~lies in the closure of the orbit of~$u$ under~$G_{U}$.  Hence the orbit of~$u$ is dense in~$\cone{\alpha}$.  This demonstrates that $G = \Fix{V}{\Part}$ has a Rubin action on~$\Cant$.
\end{proof}

\begin{cor}\label{cor:MultinuclearStabIsRubin}
If $\Part$ is a multinuclear type system on~$\AddrSet$ then the action of~$\Stab{V}{\Part}$ on~$\Cant$ is a Rubin action that stabilizes no non-empty finite subset of~$\Cant$.
\end{cor}

\begin{proof}
Since $\Fix{V}{\Part} \leq \Stab{V}{\Part}$, it follows by Lemma~\ref{lem:FixMultinuclearRubin} that $\Stab{V}{\Part}$ also has a Rubin action on~$\Cant$.  We then deduce immediately from Condition~\ref{pt:locDense} that no non-empty finite subset of~$\Cant$ is stabilized by~$\Stab{V}{\Part}$.
\end{proof}

We shall employ Rubin's theorem to separate isomorphism types of stabilizers of simple nuclear type systems, and show by construction that there are uncountably many such isomorphism types for such stabilizers. To do this, we note that the semigroups associated to such partitions are invariants of conjugacy under homeomorphisms of~$\Cant$.
%\MRQ{[It isn't clear to me what ``conjugacy'' means precisely here.]}\cpb{I think this is what we are proving in the next lemma using the Rubin conjugator.}  (The following lemma is also true for multinuclear type systems, but we only give the proof in the simpler case of nuclear type systems.)  \MRQ{[I've not worked through the details of the multinuclear case -- but why don't we include that instead of the special case if it similar?  If it is true, I don't expect it to be significantly more complicated than the version given here.]}
%\cpb{I now think we need to be a bit careful in the multinuclear case proof because we cannot just quote 5.8 (see my response to your comment after statement of 5.8).  These difficulties look surmountable, but we would need a little care.  The issue is that each $S_i$ has some proper clopen subset which acts as the identity.  So, in the direct product of groups we get over a multinuclear, there could be some clopen set supported in $S_1$ but not $S_2$ which has the trivial $s-type$ in the semigroup, and some other clopen set supported in $S_2$ but not $S_1$.  These sets will pick out the same element of the semigroup (as a direct product of groups) but there is not sum like in 5.8).}

\begin{lemma}
\label{lem:nonIsomorphicStabs}
Let $\Part$~and~$\Part[Q]$ be multinuclear type systems on~$\AddrSet$.  If $\Fix{V}{\Part}\cong \Fix{V}{\Part[Q]}$, then $\sgp{\Part}\cong \sgp{\Part[Q]}$.
\end{lemma}

%\MRQ{[The following is my attempt at the proof.  Which is preferred between this and the one below?
%
%One additional comment: The following proof makes explicit use of the homomorphism $\sgp{\Cant} \to \sgp{\Part}$ that we introduced in Section~3.  However, we don't see to have an explicit lemma that refers to it.  Perhaps that should be introduced?]

\begin{proof}
First observe that Lemma~\ref{lem:FixMultinuclearRubin} tells us that the groups $\Fix{V}{\Part}$ and~$\Fix{V}{\Part[Q]}$ have Rubin actions on~$\Cant$.  Hence, if $\theta \colon \Fix{V}{\Part} \to \Fix{V}{\Part[Q]}$ is an isomorphism, then there is a homeomorphism $\phi \colon \Cant \to \Cant$ such that $g\theta = \phi^{-1}g\phi$ for all $g \in \Fix{V}{\Part}$.  We shall use~$\phi$ to construct an isomorphism between the semigroups $\sgp{\Part}$ and~$\sgp{\Part[Q]}$.

Let $U_{1}$~and~$U_{2}$ be non-empty proper clopen subsets of~$\Cant$ with the same s-type in~$\sgp{\Part}$.  Hence, by Lemma~\ref{lem:FixTransitiveProperClopeninSType}, there exists $g \in \Fix{V}{\Part}$ such that $U_{1}g = U_{2}$.  Thus $(U_{1}\phi)(g\theta) = U_{2}\phi$.  Express $U_{1}\phi = \bigcup_{i=1}^{n} \cone{\gamma_{i}}$ as a disjoint union of cones in~$\Cant$.  Since $g\theta \in \Fix{V}{\Part[Q]}$, we may refine the addresses~$\gamma_{i}$ and so assume that each~$\gamma_{i}$ is sufficiently long that $\gamma_{i}^{g\theta} \sim \gamma_{i}$.  Hence
\[
\stypep{\Part[Q]}{U_{2}\phi} = \operatorname{s-type}_{\Part[Q]} \biggl( \bigcup_{i=1}^{n} \cone{\gamma_{i}^{g\theta} } \biggr) = \stypep{\Part[Q]}{U_{1}\phi}
\]
since the addresses $\gamma_{i}$~and~$\gamma_{i}^{g\psi}$ have the same type in~$\Part[Q]$.

As observed in Definition~\ref{def:s-type}, there is a homomorphisms $\psi \colon \sgp{\Cant} \to \sgp{\Part}$ that maps the characteristic function~$\chi_{U}$ of each non-empty clopen subset~$U$ to its s-type in the semigroup~$\sgp{\Part}$.  Furthermore, if $\alpha$~is any address, then there is a disjoint union decomposition $(\cone{\alpha})\phi = (\cone{\alpha0})\phi \cup (\cone{\alpha1})\phi$ and hence $\chi_{(\cone{\alpha})\phi} = \chi_{(\cone{\alpha0})\phi} + \chi_{(\cone{\alpha1})\phi}$.  Therefore $\stypep{\Part[Q]}{(\cone{\alpha})\phi} = \stypep{\Part[Q]}{(\cone{\alpha0})\phi} + \stypep{\Part[Q]}{(\cone{\alpha1})\phi}$.  This shows that the map $\chi_{\cone{\alpha}} \mapsto \stypep{\Part[Q]}{(\cone{\alpha})\phi}$ preserves the defining relations listed in Lemma~\ref{lem:S(C)-pres} and hence induces a homomorphism $\sgp{\Cant} \to \sgp{\Part[Q]}$.  Putting these two homomorphism together with the observation made in the previous paragraph, we deduce that there is an induced homomorphism $\theta_{\ast} \colon \sgp{\Part} \to \sgp{\Part[Q]}$ given by $\stypep{\Part}{U} \mapsto \stypep{\Part[Q]}{U}$ for each non-empty clopen subset~$U$ of~$\Cant$.  Finally, the inverse $\theta^{-1} \colon \Fix{V}{\Part[Q]} \to \Fix{V}{\Part}$ corresponds to the homeomorphism~$\phi^{-1}$ and hence the corresponding induced semigroup homomorphism $(\theta^{-1})_{\ast} \colon \sgp{\Part[Q]} \to \sgp{\Part}$ is given by $\stypep{\Part[Q]}{U} \mapsto \stypep{\Part}{U\phi^{-1}}$ when $U$~is a non-empty proper clopen subset of~$\Cant$.  Consequently, $(\theta^{-1})_{\ast}$~is the inverse of~$\theta_{\ast}$ and we conclude that $\theta_{\ast}$~is an isomorphism between the semigroups $\sgp{\Part}$~and~$\sgp{\Part[Q]}$.
\end{proof}

\UncountablyManyIso

\begin{proof}
Let $\{p_n\}$ be a strictly increasing sequence of prime numbers.  We shall construct a simple 
nuclear type system $\Part$ such that $\Stab{V}{\Part}=\Fix{V}{\Part}$ and $\sgp{\Part}\cong 
\Zint\oplus \bigoplus_{n\geq 1}\Zint/p_n\Zint$.  By Theorem~\ref{thm:branchtype} the corresponding 
stabilizer subgroups are maximal in~$V$. Denote by $\mathcal{M}$ the set of maximal 
subgroups corresponding to different increasing sequences of primes.  By construction $\mathcal{M}$ is uncountable, and by Lemma~\ref{lem:nonIsomorphicStabs} the elements of $\mathcal{M}$ are pairwise non-isomorphic.  Finally, observe that for a given such system $\Part$, by 
Corollary \ref{cor:MultinuclearStabIsRubin}, the group $\Stab{V}{\Part}$ does not stabilize any finite set of points.

The type system $\Part$ has types $X$, $Y$, and $Z^{(n,k)}$ for $n \geq 0$ and $0 \leq k \leq p_{n+1}$, which obey the following rules:
\begin{itemize}
    \item The empty word $\emptyword$ has type $X$;
    \item $X_0=Z^{(0,0)}$ and $X_1=Y$;
    \item $Y_0 = X$ and $Y_1=Y$;
    \item $\bigl(Z^{(n-1,k)}\bigr)_{0} = Z^{(n,0)}$ and $\bigl(Z^{(n-1,k)}\bigr)_{1} = Z^{(n-1,k+1)}$ for $0\leq k\leq p_{n}-1$;
    \item $\bigl(Z^{(n-1,p_{n})}\bigr)_{0} = X$ and $\bigl(Z^{(n-1,p_{n})}\bigr)_{1}=Z^{(n-1,0)}$.
\end{itemize}
Note that these rules uniquely define a type system $\Part$ and that $\Part$ is nuclear since every type has type $X$ as a descendant. 
 From this it also follows that for any two incomparable addresses $\alpha,\beta\in \AddrSet$ there are words $\gamma,\delta\in\AddrSet$ so that $\swap{\alpha\gamma}{\beta\delta}\in  \Fix{V}{\Part}$ and thus we see that $\Stab{V}{\Part}$ does not preserve any circular order on a dense set of points in $\Cant$.    
 
 The semigroup $\sgp{\Part}$ has presentation
\[
\presentation{ x,y,z^{(0)},z^{(1)},\ldots }{ x=z^{(0)}+y,\; y=x+y,\; \text{$z^{(n-1)} = p_n z^{(n)}+x+z^{(n-1)}$ for $n\geq 1$} }
\]
where the generator~$z^{(n)}$ corresponds to the type~$Z^{(n,0)}$ for each $n \geq 1$.
Since $\Part$ is nuclear, we know from Lemma~\ref{lem:multinuke-sgp} that $\sgp{\Part}$ is an abelian group.  The relation $y=x+y$ implies that $x$ is the identity element, so the presentation simplifies to
\[
\presentation{ y,z^{(0)},z^{(1)},\ldots }{ z^{(0)} = -y,\; \text{$p_n z^{(n)}=0$ for $n\geq 1$} },
\]
which is a presentation for $\Zint\oplus \bigoplus_{n\geq 1}\Zint/p_n\Zint$.

By Lemma \ref{lem:P*-action}, $\Stab{V}{\Part}$~acts on the set $\infPart$ of infinite types (which in this case is the set of types as all types in~$\Part$ are infinite) with the kernel of this action being~$\Fix{V}{\Part}$.  However, any such permutation of types must respect the cycle structure of the type-graph, which immediately forces the type~$Y$ to map to itself (as the only type with a loop on it) and then an induction argument shows all types are fixed.  Hence $\Stab{V}{\Part} = \Fix{V}{\Part}$.

All that remains is to prove that $\Part$ is simple.  Suppose $\Part[Q]$ is a proper quotient of $\Part$ and let $\approx$ be the associated equivalence relation on~$\Part$.  We must prove that $\approx$ is the universal relation.  Since $\Part[Q]$ is a proper quotient, we know that $\approx$ identifies at least two distinct types in~$\Part$.

Suppose first that $X\approx Z^{(0,0)}$.  Observe that if $X \approx Z^{(n-1,0)}$ for some~$n$ then $X \approx Z^{(0,0)} = X_{0} \approx (Z^{n-1,0})_{0} = Z^{(n,0)}$.  Hence $X \approx Z^{(n,0)}$ for all $n \geq 0$.  Then $Z^{(0,0)} = \bigl(Z^{(0,0)}\bigr)_{1^{p_{1}+1}} \approx X_{1^{p_{1}+1}} = Y$ and consequently $Z^{(n,k)} = \bigl( Z^{(n,0)} \bigr)_{1^{k}} \approx Y_{1^{k}} = Y$ for $1 \leq k \leq p_{n+1}$.  This shows that $\approx$~is the universal relation in this case.

Now if $X \approx Y$ then $Z^{(0,0)}=X_0 \approx Y_0 = X$ and we are done by the previous paragraph.  If $X \approx Z^{(n-1,k)}$ for some $n \geq 1$ and $0 \leq k \leq p_{n}$ then $X \approx Z^{(n-1,k)} = \bigl( Z^{(n-1,k)} \bigr)_{1^{p_{n}+1}} \approx X_{1^{p_{n}+1}} = Y$ and again we are done.  In conclusion, if $X$~is identified with any other type then the resulting relation~$\approx$ is universal.

Now, since $Y_{0} = Y_{10} = X$ and each type $T \neq Y$ has the property that either $T_{0} \neq X$ or $T_{10} \neq X$, identifying~$Y$ with any other type also identifies~$X$ with another type, thus yielding the universal relation.

Finally, if $Z' = Z^{(m,j)}$ and $Z''= Z^{(n,k)}$ are distinct then there exists~$\ell$ such that $(Z')_{1^{\ell}0} = Z^{(m+1,0)}$ and $(Z'')_{1^{\ell}0} = X$.  When $m \neq n$, this is established using the fact that $p_{m}$~and~$p_{n}$ are distinct.  Consequently, $Z'\approx Z''$ implies $Z^{(m+1,0)}\approx X$ and this completes the verification that $\Part$~is simple and also establishes the theorem.
\end{proof}

\begin{comment}
It's not clear that they're simple, but oh well.

The idea is to make uncountably many nuclear maximal subgroups whose associated semigroups (which are actually abelian groups) are pairwise non-isomorphic.  Then we need to show that the isomorphism type of the semigroup is a conjugacy invariant and use Rubin's theorem.

For the type system, let $p_1, p_2, p_3, \ldots$ be any sequence of distinct primes.  We have types $x_0, x_1, x_2, \ldots$ along the right vine, where $x_0$ is the identity element of the semigroup.  As a general rule, we make sure everything has an $x_0$ below it, which guarantees that the type system is nuclear.

Along the right vine we have $x_n = y_n + x_{n+1}$ for each $n$ for some types $y_n$.  Below $y_0$ we put $y_0 = y_0 + x_0$.  For every other $y_n$ we arrange for a tree below $y_n$ to give the relation $y_n = p_n(z_n) + y_n$ for some type $z_n$.  Note then that $z_n$ ends up having order $p_n$.  We also put $z_n = x_0 + z_n$.  The semigroup turns out to be the direct sum of the infinite-rank free abelian group generated by the $y_n$'s and the finite cyclic groups of prime order generated by the $z_n$'s.  Since any collection of primes is possible, there are uncountably many isomorphism types.

Again, we might need to fiddle with these a bit to make sure they're all simple.  I think we do prove that infinite, simple nuclear type systems give maximal subgroups.
\end{comment}

\section{Conditions which prevent maximality}
\label{sec:nonproper}

In the interest of better understanding maximal subgroups of Thompson's group~$V$, it is useful to be able to determine when a subgroup is in fact a proper subgroup.  In view of the results of previous sections, we are particularly interested in the case of subgroups that do not stabilize a non-trivial type system~$\Part$; that is, in some sense are ``primitive'' in their partial action on the addresses in~$\AddrSet$.  In this section, we study two conditions that are related to this primitivity and we discover that any subgroup satisfying either of these conditions is actually~$V$.

We start by proving the surprising fact that $V$~has no proper subgroups whose partial action on~$\AddrSet$ is two-fold transitive.  Recall first that for an address~$\alpha$ in~$\AddrSet$ and an element~$g$ in~$V$, we say that $\alpha^{g}=\beta \in \AddrSet$ if $\cone{\alpha}$~is taken to~$\cone{\beta}$ under~$g$ via a prefix replacement.  Otherwise we say $\alpha^{g}$~is undefined.  This definition extends naturally, for any positive integer~$k$, to the set
\[
\smalladdr{k} = \biggl\{ (\alpha_{1},\alpha_{2},\dots,\alpha_{k}) \in \AddrSet^{k} \biggm| \text{$\cone{\alpha_{i}} \cap \cone{\alpha_{j}} = \emptyset$ for all $i \neq j$ and $\bigcup_{i=1}^{k} \cone{\alpha_{i}} \neq \Cant$} \biggr\}.
\]
Indeed, we define
\[
(\alpha_{1},\alpha_{2},\dots,\alpha_{k})^{g} = (\beta_{1},\beta_{2},\dots,\beta_{k})
\]
if $\alpha_{i}^{g} = \beta_{i}$ for all~$i$.  Otherwise, we say $(\alpha_{1},\alpha_{2},\dots,\alpha_{k})^{g}$ is undefined.

\paragraph{Remark}
When $k > 1$, it will often be the case that for a given~$g$ there will be an infinite set of sequences $(\alpha_{1},\alpha_{2},\dots,\alpha_{k})$ for which $(\alpha_{1},\alpha_{2},\dots,\alpha_{k})^{g}$ is undefined.  Nevertheless, we still refer to this as a partial action.

\begin{defn}
A subgroup~$G$ of~$V$ will be said to act \emph{$k$\nbd fold transitively on proper cones} when the induced partial action of~$G$ on~$\smalladdr{k}$ is
transitive; that is, if for all $(\alpha_{1},\alpha_{2},\dots,\alpha_{k})$, $(\beta_{1},\beta_{2},\dots,\beta_{k}) \in \smalladdr{k}$, there exists $g \in G$ such that $\alpha_{i}^{g} = \beta_{i}$ for $i = 1$,~$2$, \dots,~$k$.
\end{defn}

\begin{lemma}\label{lem:trick}
Let $G$~be a subgroup of~$V$ that acts $2$\nbd fold transitively on proper cones.  Let $g$~be an element of~$G$ such that $\cone{0}$~is invariant under~$g$.  Then there exists a $h \in G$ that agrees with~$g$ on~$\cone{0}$ and is the identity on~$\cone{10}$.
\end{lemma}

\begin{proof}
Let $\alpha$~be any address such that $\cone{\alpha}$~is a proper subcone of~$\cone{1}$ and such that $\alpha^{g}$~is defined.  Then $\alpha^{g} = \beta$ for some address~$\beta$ and consequently $\cone{\beta}$~is also a proper subcone of~$\cone{1}$.  Since $G$~is $2$\nbd fold transitive, there exist elements $k, k' \in G$ such that
\[
0^{k} = 0, \quad (10)^{k} = \alpha, \quad 0^{k'} = 0, \quad \beta^{k'} = 10.
\]
Then $h = kgk'$ agrees with~$g$ on~$\cone{0}$ and $(10)^{h} = 10$.
\end{proof}

\begin{lemma}
\label{lem:doubletransposition}
Let $G$~be a subgroup of~$V$ that acts $2$\nbd fold transitively on proper cones.  Then $\swap{\alpha0}{\alpha1} \, \swap{\beta0}{\beta1} \in G$ for all $(\alpha,\beta) \in \smalladdr{2}$.
\end{lemma}

\begin{proof}
We shall first show that $\swap{000}{001} \, \swap{010}{011}$ is an element of~$G$.  Since $G$~acts $2$\nbd fold transitively on proper cones, there exists an element of~$G$ that agrees with~$\swap{00}{01}$ on~$\cone{0}$.  By Lemma~\ref{lem:trick}, there exists an element $h \in G$ that agrees with~$\swap{00}{01}$ on $\cone{0} \cup \cone{10}$.  Since $G$~acts $2$\nbd fold transitively on proper cones, there exists an element $k_{1} \in G$ such that $0^{k_{1}} = 00$ and $(10)^{k_{1}} = 01$.  Then $k_{1}^{-1}hk_{1}$~agrees with~$\swap{000}{001}$ on~$\cone{0}$.  Again by Lemma~\ref{lem:trick} there exists $h' \in G$ that agrees with~$\swap{000}{001}$ on $\cone{0} \cup \cone{10}$.  There is also some $k_{2} \in G$ such that $0^{k_{2}} = 0$ and $(10)^{k_{2}} = 11$.  Then $g = k_{2}^{-1}h'k_{2}$ agrees with~$\swap{000}{001}$ on $\cone{0} \cup \cone{11}$.  Since $h$~is the identity on~$\cone{10}$ and $g$~is the identity on~$\cone{11}$, the commutator~$[h,g]$ is the identity on~$\cone{1}$, so
\[
[h,g] = [ \swap{00}{01}, \swap{000}{001} ] = \swap{000}{001} \, \swap{010}{011}.
\]
This completes the first step of the proof.  Now if $\alpha$~and~$\beta$ are any incomparable addresses of small support, there exists~$g' \in G$ such that $(00)^{g'} = \alpha$ and $(01)^{g'} = \beta$.  Then $G$~contains
\[
\bigl( \swap{000}{001} \, \swap{010}{011} \bigr)^{g'} = \swap{\alpha0}{\alpha1} \, \swap{\beta0}{\beta1},
\]
as claimed.
\end{proof}

\TwoFoldTrans

\begin{proof}
Let $\alpha$,~$\beta$ and~$\gamma$ be three incomparable addresses with small support.  Since $G$~acts $2$\nbd fold transitively on proper cones, there exists $g \in G$ such that $(00)^{g} = \alpha$ and $(01)^{g} = \beta$.  Similarly there exists $h \in G$ such that $(00)^{h} = \beta$ and $(01)^{h} = \gamma$.  Since $\cone{\alpha} \cup \cone{\beta} \cup \cone{\gamma} \neq \Cant$, there is an address~$\delta$ that is incomparable with each of $\alpha$,~$\beta$ and~$\gamma$ and such that $\zeta = \delta^{g^{-1}}$ and $\eta = \delta^{h^{-1}}$ are both defined.  As $\cone{\delta}$~is disjoint from $\cone{\alpha} \cup \cone{\beta}$, it follows that $\cone{\zeta} = (\cone{\delta})^{g^{-1}}$ is disjoint from $(\cone{\alpha} \cup \cone{\beta})^{g^{-1}} = \cone{0}$.  Similarly $\cone{\eta}$~is disjoint from~$\cone{0}$.
  
Now use Lemma~\ref{lem:doubletransposition} to tell us that $\swap{00}{01} \, \swap{\zeta0}{\zeta1}$ and $\swap{00}{01} \, \swap{\eta0}{\eta1}$ are elements of~$G$.  We conjugate the first by~$g$ and the second by~$h$ to deduce that $G$~contains
\[
\swap{\alpha}{\beta} \, \swap{\delta0}{\delta1}
\AND
\swap{\beta}{\gamma} \, \swap{\delta0}{\delta1}.
\]
We conclude that the product of these two elements, that is, the element~$\threecycle{\alpha}{\gamma}{\beta}$ also belongs to~$G$ for all choices of incomparable addresses $\alpha$,~$\beta$ and~$\gamma$ of small support.

Now let  $\alpha$~and~$\beta$ are two incomparable addresses with with small support.  Using what has just been established we deduce that $G$~contains the two elements $\threecycle{\alpha0}{\beta0}{\alpha1}$ and $\threecycle{\alpha0}{\beta0}{\beta1}$.  Hence it also contains
\begin{align*}
\bigl[ \threecycle{\alpha0}{\beta0}{\alpha1} , (\threecycle{\alpha0}{\beta0}{\beta1} \bigr]
&= \threecycle{\alpha0}{\alpha1}{\beta0} \, \threecycle{\alpha0}{\beta0}{\alpha1}^{\threecycle{\alpha0}{\beta0}{\beta1}} \\
&= \threecycle{\alpha0}{\alpha1}{\beta0} \, \threecycle{\beta0}{\beta1}{\alpha1} \\
&= \swap{\alpha0}{\beta0} \, \swap{\alpha1}{\beta1} \\
&= \swap{\alpha}{\beta}.
\end{align*}
Hence $G$~contains all transpositions~$\swap{\alpha}{\beta}$ where $\alpha$~and~$\beta$ are incomparable addresses of small support.  These transpositions generate~$V$ and we conclude that $G = V$, as claimed.
\end{proof}

Theorem~\ref{thm:2foldtrans} says that a maximal subgroup of~$V$ can only act either zero- or one-fold transitively on proper cones.  We now consider a different condition related to primitivity.

We shall use the term \emph{small transposition} for a transposition~$\swap{\alpha}{\beta}$ where the set of addresses $\{\alpha,\beta\}$ has small support.

\begin{defn}
Let $G$~be a subgroup of~$V$.  Let $\gamma$,~$\delta$, $\gamma_{0}$,~$\gamma_{1}$, \dots,~$\gamma_{n}$ be addresses in~$\AddrSet$.  We say the sequence $(\gamma_{0}, \gamma_{1}, \dots, \gamma_{n})$ is a \emph{$G$\nbd chain} connecting~$\gamma$ to~$\delta$ if 
\begin{enumerate}
\item $\gamma = \gamma_{0}$ and $\delta = \gamma_{n}$, and
\item for all~$i$ with $0 \leq i \leq n-1$, \ $\swap{\gamma_{i}}{\gamma_{i+1}}$ is a small transposition that belongs to~$G$.
\end{enumerate}
\end{defn}

We say that the \emph{length} of the $G$\nbd chain $(\gamma_{0}, \gamma_{1}, \dots, \gamma_{n})$ is~$n$ (as there are $n$~small transpositions involved).  We also will refer to any given address~$\gamma_{i}$ appearing in this sequence as a \emph{link (in the chain)}.  If $C = (\gamma_{0}, \gamma_{1}, \dots, \gamma_{n})$ is a $G$\nbd chain from~$\gamma$ to~$\delta$, then any subchain $(\gamma_{i}, \gamma_{i+1}, \dots, \gamma_{j})$ of~$C$ is also a $G$\nbd chain (namely from~$\gamma_{i}$ to~$\gamma_{j}$).

\begin{defn}
Let $G$~be a subgroup of~$V$.  We say that $G$~is \emph{swap-primitive} if for any pair of non-empty addresses $\gamma,\delta \in \AddrSet$, there is a $G$-chain connecting~$\gamma$ to~$\delta$; that is, there is a sequence $(\gamma_{0}, \gamma_{1}, \dots, \gamma_{n})$ with $\gamma_{0} = \gamma$, \ $\gamma_{n} = \delta$ and such that $\swap{\gamma_{i}}{\gamma_{i+1}}$~is a small transposition in~$G$ for each~$i$.
\end{defn}

\begin{rem}
This definition of swap-primitive is related to a criterion for primitivity of actions.  Specifically, an action of a group $G$ on a set $X$ is primitive if and only if for every pair $\{x_1,x_2\}$ of distinct elements of~$X$, the $G$-orbit of the (undirected) edge $\{x_1,x_2\}$ is a connected graph with vertex set~$X$ (see Proposition~\ref{prop:DM-prim}).  A subgroup $G$ of $V$ is swap-primitive if the graph whose edges correspond to transpositions in $G$ of small support is connected.  Thus swap-primitivity follows from  the hypothesis that the partial action of $G$ on $\AddrSet$ is primitive in the graph-theoretic sense together with the hypothesis that $G$ has at least one transposition with small support.
\end{rem}

We will prove in Theorem~\ref{thm:swap-primitiveisV} that the only subgroup of~$V$ that acts swap-primitively is~$V$ itself.  We shall need several technical lemmas and start with some basic observations about properties of chains of minimal length.

\begin{lemma}
\label{lem:chainCompression}
Let $G$~be a subgroup of\/~$V$.  If $\gamma,\delta\in\AddrSet$ and there is a $G$\nbd chain $(\gamma_{0},\gamma_{1},\dots,\gamma_{n})$ connecting~$\gamma$ to~$\delta$ such that $\gamma_{i}$~is incomparable to~$\gamma_{i+2}$ for some index~$i$, then the sequence $(\gamma_{0},\gamma_1,\dots,\gamma_{i},\gamma_{i+2},\dots\gamma_{n})$ is also a $G$\nbd chain connecting~$\gamma$ to~$\delta$.
\end{lemma}

\begin{proof}
Note that for every index~$j$, the addresses $\gamma_{j}$~and~$\gamma_{j+1}$ are incomparable by assumption that $\swap{\gamma_{j}}{\gamma_{j+1}}$ is a transposition.  Under the hypothesis that $\gamma_{i} \perp \gamma_{i+2}$ for some index~$i$, observe that $G$~also contains $\swap{\gamma_{i}}{\gamma_{i+2}} = \swap{\gamma_{i}}{\gamma_{i+1}}^{\swap{\gamma_{i+1}}{\gamma_{i+2}}}$ and this is a small transposition.  Consequently the link~$\gamma_{i+1}$ may be omitted and we still have a $G$\nbd chain connecting~$\gamma$ to~$\delta$.
\end{proof}

\begin{lemma}
\label{lem:disjointChainHalves}
Let $G$~be a subgroup of\/~$V$ and $\gamma$~and~$\delta$ be addresses such that $C = (\gamma_{0},\gamma_{1},\dots,\gamma_{n})$ is a $G$\nbd chain of minimal length connecting~$\gamma$ to~$\delta$.  Set
\begin{align*}
E_{C} &= \set{\gamma_{2i}}{i \in \Zint, \; 0 \leq 2i \leq n} \\[5pt]
O_{C} &= \set{\gamma_{2i+1}}{i \in \Zint, \; 0 \leq 2i+1 \leq n}.
\end{align*}
Let $\lambda_{e}$~be the shortest element of~$E_{C}$ and $\lambda_{o}$~be the shortest element of~$O_{C}$.  Then
\begin{enumerate}
\item every element of~$E_{C}$ has $\lambda_{e}$~as a prefix,
\item \label{i:Odd-prefix}
every element of~$O_{C}$ has $\lambda_{o}$~as a prefix, and
\item $\lambda_{e}$~is incomparable to~$\lambda_{o}$.
\end{enumerate}
\end{lemma}

\begin{proof}
First note that, for every index $i\leq n-2$, the address~$\gamma_{i}$ is comparable to~$\gamma_{i+2}$ by Lemma~\ref{lem:chainCompression}.  Write $E_{C} = \{ \gamma_{0}, \gamma_{2}, \dots, \gamma_{2m} \}$ and let $\lambda_{e}$~be its shortest member.  We shall show, by induction on~$m$, that $\lambda_{e}$~is a prefix of every member of~$E_{C}$.  If $m = 0$, then the result is trivial.  We first assume that $\lambda_{e}$~is an element of $\{ \gamma_{0}, \gamma_{2}, \dots, \gamma_{2m-2} \}$.  By induction, $\lambda_{e}$~is a prefix of each of these $m-1$~elements.  Now $\gamma_{2m-2}$~and~$\gamma_{2m}$ are comparable.  If $\gamma_{2m-2} \prefix \gamma_{2m}$ then it follows $\lambda_{e} \prefix \gamma_{2m}$.  If $\gamma_{2m} \prefix \gamma_{2m-2}$, then they share~$\lambda_{e}$ as a prefix (since both are at least as long as~$\lambda_{e}$).  This establishes the claim in this case.  The case when $\lambda_{e} = \gamma_{2m}$ follows similarly by applying induction to the set $\{ \gamma_{2}, \gamma_{4}, \dots, \gamma_{2m} \}$ instead.  Hence by induction $\lambda_{e}$~is indeed a prefix of each member of~$E_{C}$.  Part~\ref{i:Odd-prefix} follows similarly.

Suppose that $\lambda_{e}$~and~$\lambda_{o}$ are comparable.  We consider the case that $\lambda_{e} \prefix \lambda_{o}$.  Suppose $\gamma_{2j} = \lambda_{e}$.  If $j > 0$, consider the address $\gamma_{2j-1} \in O_{C}$.  This has $\lambda_{o}$~as a prefix and hence $\gamma_{2j} \prefix \gamma_{2j-1}$ by our hypothesis.  This is impossible since $\swap{\gamma_{2j-1}}{\gamma_{2j}}$~is assumed to be a valid transposition.  If $j = 0$, the same argument applies using the transposition~$\swap{\gamma_{0}}{\gamma_{1}}$ instead, while the case when $\lambda_{o} \prefix \lambda_{e}$ is also similar.  We conclude that $\lambda_{e}$~and~$\lambda_{o}$ are incomparable, as claimed.
\end{proof}

\begin{cor}
\label{cor:chainsRiverHop}
Let $G$~be a subgroup of\/~$V$ that is swap-primitive.  Suppose further that $\gamma$~and~$\delta$ are incomparable addresses of small support such that $C =(\gamma_{0}, \gamma_{1}, \dots, \gamma_{n})$ is a $G$\nbd chain of minimal length connecting~$\gamma$ to~$\delta$ where $n \geq 2$.  Let $E_{C}$~and~$O_{C}$ denote the members of the chain of even and odd index, respectively.  Then $n \geq 3$ and all of the elements of~$E_{C}$ begin with one of the letters in~$\letters$ and all those of~$O_{C}$ begin with the other letter.
\end{cor}

\begin{proof}
First note that if $n = 2$, then a $G$\nbd chain $(\gamma_{0},\gamma_{1},\gamma_{2})$ connecting~$\gamma$ to~$\delta$ cannot be minimal since $\gamma_{0}$~and~$\gamma_{2}$ are incomparable and we may use Lemma~\ref{lem:chainCompression} to reduce to a $G$\nbd chain of length~$1$.  Hence it must be the case that $n \geq 3$.  By Lemma~\ref{lem:disjointChainHalves}, it remains to show that $\gamma_{0}$~and~$\gamma_{1}$ begin with different letters.

Suppose that $0$~is the first letter of both $\gamma_{0}$~and~$\gamma_{1}$.  Let $D = (\delta_{0}, \delta_{1}, \dots, \delta_{m})$ be a $G$\nbd chain of minimal length connecting~$\gamma_{0}$ to~$1$.  Applying Lemma~\ref{lem:disjointChainHalves} to~$C$ establishes that all~$\gamma_{i}$ have~$0$ as their first letter, while upon applying it to~$D$, we observe that all terms~$\delta_{i}$ with even index have~$0$ as first letter and that $1$~is the shortest term with odd index.  Consequently $\delta_{1}$~has~$1$ as its first letter.  Therefore $\delta_{1}$~is incomparable to all the addresses~$\gamma_{i}$.

Since $D$~is a $G$\nbd chain with $\delta_{0} = \gamma_{0}$, it is the case that $\swap{\gamma_{0}}{\delta_{1}} \in G$.  Suppose that $\swap{\gamma_{i}}{\delta_{1}} \in G$ for some~$i$ with $0 \leq i < n$.  Then $\swap{\gamma_{i+1}}{\delta_{1}} = \swap{\gamma_{i}}{\gamma_{i+1}}^{\swap{\gamma_{i}}{\delta_{1}}} \in G$ also.  Hence we conclude that $\swap{\gamma_{n}}{\delta_{1}} \in G$ and therefore $\swap{\gamma}{\delta} = \swap{\gamma_{0}}{\gamma_{n}} = \swap{\gamma_{0}}{\delta_{1}}^{\swap{\gamma_{n}}{\delta_{1}}} \in G$.  This means that $\gamma$~and~$\delta$ can actually be connected by a $G$\nbd chain of length~$1$ contrary to assumption.

A symmetric argument shows that it cannot be the case that both $\gamma_{0}$~and~$\gamma_{1}$ have $1$~as their first letter.  Hence one of them starts with the letter~$0$ and the other starts with~$1$.
\end{proof}

\begin{lemma}
\label{lem:DiamondLemma}
Let $G$~be a subgroup of\/~$V$. Suppose that $(\gamma_{0},\gamma_{1},\gamma_{2},\gamma_{3})$ are consecutive terms in some $G$\nbd chain~$C$ and that there is a prefix~$\delta$ of~$\gamma_{2}$ and an address~$\zeta$ incomparable with each~$\gamma_{i}$ such that $\swap{\delta}{\zeta} \in G$. Then $C$~is not a $G$\nbd chain of minimal length.
\end{lemma}

\begin{proof}
Suppose that $C$~is a $G$\nbd chain of minimal length.  Then, by Corollary~\ref{cor:chainsRiverHop}, the first letters of the addresses in~$C$ alternate and we may assume that $\gamma_{0}$~and~$\gamma_{2}$ have first letter~$0$ while $\gamma_{1}$~and~$\gamma_{3}$ have first letter~$1$.  Write $\gamma_{2} = \delta\eta$.  Then $G$~contains
\[
\swap{\gamma_{1}}{\gamma_{2}}^{\swap{\delta}{\zeta}} = \swap{\gamma_{1}}{\zeta\eta}
\AND
\swap{\gamma_{2}}{\gamma_{3}}^{\swap{\delta}{\zeta}} = \swap{\zeta\eta}{\gamma_{3}}.
\]
Hence we can replace~$\gamma_{2}$ in the original $G$\nbd chain by~$\zeta\eta$.  Since all three of $\gamma_{0}$,~$\gamma_{1}$ and~$\zeta\eta$ are incomparable, we can shorten the $G$\nbd chain with use of Lemma~\ref{lem:chainCompression}.  This contradicts our assumption concerning~$C$ and establishes that, indeed, it is not of minimal length.
\end{proof}

\begin{lemma}
\label{lem:Appending}
Let $G$~be a subgroup of\/~$V$ that is swap-primitive.  If\/ $\swap{\gamma}{\delta} \in G$ and $\gamma$~and~$\delta$ start with the same letter, then $\swap{\gamma\zeta}{\delta\zeta} \in G$ for all words~$\zeta \in \AddrSet$.
\end{lemma}

\begin{proof}
Suppose without loss of generality that $\gamma$~and~$\delta$ begin with~$0$.  Since there exists a $G$\nbd chain from $\gamma \zeta$ to~$1$, we know by Lemma~\ref{lem:disjointChainHalves} that $\swap{\gamma \zeta}{1\eta} \in G$ for some non-empty address~$\eta$.  Then $G$~contains $\swap{\gamma\zeta}{1\eta}^{\swap{\gamma}{\delta}} = \swap{\delta \zeta}{1\eta}$ and hence contains $\swap{\gamma\zeta}{1\eta}^{\swap{\delta \zeta}{1\eta}} = \swap{\gamma \zeta}{\delta \zeta}$.
\end{proof}

\begin{lemma}
\label{lem:sameincones}
Let\/ $G$~be a subgroup of\/~$V$ that is swap-primitive suppose and that\/ $\swap{\gamma\zeta}{\gamma\eta} \in G$ for some non-empty addresses $\gamma,\zeta,\eta \in \AddrSet$.  Then, for all non-empty addresses~$\delta$, \ $G$~contains the transposition~$\swap{\delta\zeta}{\delta\eta}$.
\end{lemma}

\begin{proof}
Let $(\gamma_{0},\gamma_{1},\dots,\gamma_{n})$ be a $G$\nbd chain connecting~$\gamma$ to~$\delta$.  Then the conjugate of~$\swap{\gamma\zeta}{\gamma\eta}$ by the product $\swap{\gamma_{0}}{\gamma_{1}} \swap{\gamma_{1}}{\gamma_{2}} \dots \swap{\gamma_{n-1}}{\gamma_{n}}$ equals~$\swap{\delta\zeta}{\delta\eta}$ and so this also is in~$G$.
\end{proof}

Our strategy is to show that swap-primitivity of a subgroup~$G$ implies that it contains all small transpositions.  We begin by observing that certain transpositions are already guaranteed to belong to~$G$.

\begin{lemma}
\label{lem:swap-primitiveShortSwapExistence}
Let $G$~be a subgroup of\/~$V$ that is swap-primitive.  Then there are words $\gamma, \delta \in \AddrSet$ such that
\begin{enumerate}
\item one of the transpositions $\swap{0}{10\gamma}$~or~$\swap{0}{11\gamma}$ lies in~$G$, and
\item one of the transpositions $\swap{1}{00\delta}$~or~$\swap{1}{01\delta}$ lies in~$G$.
\end{enumerate}
Moreover,
\begin{enumerate}
\setcounter{enumi}{2}
\item \label{i:PP-exist3}
if\/ $\swap{0}{10\gamma}\in G$ then $\swap{11}{10\mu}\in G$ for some~$\mu$,
\item if\/ $\swap{0}{11\gamma}\in G$ then $\swap{10}{11\mu}\in G$ for some~$\mu$,
\item if\/ $\swap{1}{01\delta}\in G$ then $\swap{00}{01\nu}\in G$ for some~$\nu$, and
\item \label{i:PP-exist6}
if\/ $\swap{1}{00\delta}\in G$ then $\swap{01}{00\nu}\in G$ for some~$\nu$.
\end{enumerate}
\end{lemma}

\begin{proof}
Since $G$~is swap-primitive there is a $G$\nbd chain connecting~$0$ to~$11$.  The first link in this $G$\nbd chain is $0$ and so the second link must be an address of length at least~$2$ that begins with the letter~$1$ as a pair of links correspond to a small transposition in~$G$.  This establishes the first claim and the second has a similar proof.

For~\ref{i:PP-exist3}, assume that $\swap{0}{10\gamma} \in G$ for some $\gamma \in \AddrSet$.  There is a $G$\nbd chain $C = (\gamma_{0},\gamma_{1},\dots,\gamma_{m})$  of minimal length connecting~$11$ to~$0$.  By Lemma~\ref{lem:disjointChainHalves}, the links~$\gamma_{i}$ with odd index must have~$0$ as prefix.  Hence $\gamma_{1} = 0\zeta$ for some word~$\zeta$.  Then $G$~contains
\[
\swap{\gamma_{0}}{\gamma_{1}}^{\swap{0}{10\gamma}} = \swap{11}{0\zeta}^{\swap{0}{10\gamma}} = \swap{11}{10\gamma\zeta}
\]
and now setting $\mu = \gamma\zeta$ we have established~\ref{i:PP-exist3}.  The remaining assertions are established in similar fashion.
\end{proof}

The next lemma provides a sufficient condition for establishing that a subgroup is in fact equal to~$V$.

\begin{lemma}
\label{lem:wincondition}
Suppose that $G$~is swap-primitive subgroup of\/~$V$ that contains one of the transpositions from~$\{\swap{00}{01\gamma}, \swap{10}{11\delta} \}$ and one of the transpositions from~$\{\swap{01}{00\zeta}, \swap{11}{10\eta}\}$ for some words $\gamma, \delta, \zeta, \eta \in \AddrSet$. Then $G = V$.
\end{lemma}

\begin{proof}
Observe first that, by Lemma~\ref{lem:sameincones}, if $G$~contains one of the transpositions from $\{\swap{00}{01\gamma}, \swap{10}{11\delta} \}$ and one of the transpositions from  $\{\swap{01}{00\zeta}, \swap{11}{10\eta}\}$, then in fact $G$ contains a transposition of all four forms.

We claim that if $C$~is a $G$\nbd chain of length~$3$ connecting two incomparable addresses $\alpha$~and~$\beta$ of small support, then this is not a minimal length $G$\nbd chain connecting the addresses.  Suppose, for a contradiction, that $C$~is of minimal length.  We may assume that $\length{\alpha} \geq 2$, since otherwise we may simply reverse the entries in~$C$ and reverse the roles of $\alpha$~and~$\beta$.  To fix notation, assume that $\alpha = 00\mu_{0}$ for some word~$\mu_{0}$.  Then Corollary~\ref{cor:chainsRiverHop} tells us that $C$~has the form $(00\mu_{0}, 1\mu_{1}, 0\mu_{2}, 1\mu_{3})$.  If $\mu_{2}$~were the empty word, then necessarily  $\mu_{1}$~and~$\mu_{3}$ are non-empty, so $\beta = 1\mu_{3}$ has length at least~$2$ also and reversing~$C$ gives a $G$\nbd chain $(1\mu_{2},0,1\mu_{1},0\mu_{0})$ with the claimed form.  Hence we may assume that $\mu_{2}$~is non-empty.  Furthermore, by Lemma~\ref{lem:chainCompression}, $00\mu_{0}$~and~$0\mu_{2}$ must be comparable, so $\mu_{2}$~also has initial letter~$0$.  Hence $C = (00\mu_{0},1\mu_{1},00\nu,1\mu_{3})$ for some word~$\nu$.  Finally, by hypothesis, $\swap{00}{01\gamma} \in G$ for some $\gamma \in \AddrSet$.  Then $01\gamma$~is incomparable with each entry of~$C$ and so, by Lemma~\ref{lem:DiamondLemma}, $C$~is not of minimal length.  Symmetric arguments apply when $\alpha$~begins with other choices of first two letters.

Now consider a pair of incomparable addresses $\alpha$~and~$\beta$ of small support and let $C = (\gamma_{0},\gamma_{1},\dots,\gamma_{n})$ be a $G$\nbd chain of minimal length connecting them.  We have just observed that $n \neq 3$.  Assume $n > 3$.  Corollary~\ref{cor:chainsRiverHop} shows that the first letters of the~$\gamma_{i}$ alternate.  In particular, $\gamma_{0} \perp \gamma_{3}$ and $\gamma_{1} \perp \gamma_{4}$.  One of $\gamma_{0}$~and~$\gamma_{1}$ has length at least~$2$.  Hence the previous paragraph applies either to $(\gamma_{0},\gamma_{1},\gamma_{2},\gamma_{3})$ or to $(\gamma_{1},\gamma_{2},\gamma_{3},\gamma_{4})$ and shows it is not of minimal length.  This contradicts the assumption that $C$~is a $G$\nbd chain of minimal length connecting~$\alpha$ to~$\beta$.

It is not possible that $n = 2$ since then all three of $\gamma_{0}$,~$\gamma_{1}$ and~$\gamma_{2}$ are incomparable and we would deduce $\swap{\gamma_{0}}{\gamma_{2}} = \swap{\gamma_{0}}{\gamma_{1}}^{\swap{\gamma_{1}}{\gamma_{2}}} \in G$.  Hence $n = 1$, which shows that $\swap{\alpha}{\beta} \in G$.  In conclusion, all small transpositions belong to~$G$ and so $G = V$ since these elements generate~$V$.  This completes the proof of the lemma.
\end{proof}

\begin{lemma}
\label{lem:connectedto0}
Let $G$~be a subgroup of\/~$V$ that is swap-primitive.
\begin{enumerate}
\item \label{i:01-to-0}
If\/ $\swap{1}{00\delta} \in G$, then $01$~is connected to~$0$ by a $G$\nbd chain of length~$2$.
\item If\/ $\swap{1}{01\delta} \in G$, then $00$~is connected to~$0$ by a $G$\nbd chain of length~$2$.
\end{enumerate}
\end{lemma}

\begin{proof}
\ref{i:01-to-0}~Suppose that $\swap{1}{00\delta} \in G$ but that $01$~is not connected to~$0$ by a $G$\nbd chain of length~$2$.  Let $(\gamma_{0},\gamma_{1},\dots,\gamma_{n})$ be a $G$\nbd chain of minimal length connecting~$01$ to~$00$ and suppose $n \geq 2$.  Here $\gamma_{0} = 01$, \ $\gamma_{1} = 1\zeta$ for some $\zeta \in \AddrSet$ by Corollary~\ref{cor:chainsRiverHop}, while $\gamma_{2} \neq 0$ by hypothesis.  By Lemma~\ref{lem:chainCompression}, $\gamma_{2}$~must be comparable with~$01$, so $\gamma_{2} = 01\eta$ for some non-empty word~$\eta$.  Now $G$~contains
\[
\swap{01}{1\zeta}^{\swap{1}{00\delta}} = \swap{01}{00\delta\zeta}
\AND
\swap{1\zeta}{01\eta}^{\swap{1}{00\delta}} = \swap{00\delta\zeta}{01\eta}
\]
and so there is also a $G$\nbd chain $(01,00\delta\zeta,01\eta,\dots,\gamma_{n})$ from~$00$ to~$01$.  Corollary~\ref{cor:chainsRiverHop} tells us that this cannot be of minimal length and so we have a contradiction to our original choice of $G$\nbd chain between $01$~and~$00$.  Hence $n = 1$ and so $\swap{00}{01} \in G$.  Finally, by Lemma~\ref{lem:swap-primitiveShortSwapExistence}\ref{i:PP-exist6}, $\swap{01}{00\nu} \in G$ for some~$\nu$.  Lemma~\ref{lem:wincondition} tells us that $G = V$.  This is then a contradiction to the assumption that $01$~is not connected to~$0$ by a $G$\nbd chain of length~$2$ and completes the proof.  The second part of the lemma is established by a symmetric argument.
\end{proof}

\swapPrimitiveIsV

\begin{proof}
By Lemma~\ref{lem:swap-primitiveShortSwapExistence}, there exists a word~$\delta$ such that one of the transpositions $\swap{1}{00\delta}$ or~$\swap{1}{01\delta}$ lies in $G$.  Let us assume that $\swap{1}{00\delta} \in G$.  By Lemma~\ref{lem:swap-primitiveShortSwapExistence}\ref{i:PP-exist6}, there is a word~$\nu$ such that $\swap{01}{00\nu}\in G$.  It follows by Lemma~\ref{lem:Appending} that $\swap{010}{00\nu0} \in G$ also.  By Lemma~\ref{lem:connectedto0}, there is a $G$\nbd chain of length~$2$ connecting~$01$ to~$0$ and this has the form $(01,1\zeta,0)$ for some $\zeta \in \AddrSet$ and hence $\swap{01}{1\zeta}, \swap{0}{1\zeta} \in G$.  We now compute
\[
\swap{010}{00\nu0}^{\swap{01}{1\zeta} \swap{010}{00\nu0} \swap{0}{1\zeta} \swap{01}{1\zeta}} = \swap{00}{0110}
\]
to deduce $\swap{00}{0110} \in G$.  Taken together with $\swap{01}{00\nu} \in G$, we conclude that $G = V$ by Lemma~\ref{lem:wincondition}.  The case when $\swap{1}{01\delta} \in G$ is similar.
\end{proof}

\section{Questions}\label{sec:questions}

At this time, we are aware of the following types of maximal subgroups of Thompson's group~$V$: stabilizers of certain type systems (both finite and infinite) and the images of the group~$T$ under automorphisms of~$V$.  The stabilizers of a finite set of elements in~$\Cant$ of the same tail class considered in Corollary~\ref{cor:Stab-tailclass} can be expressed as stabilizers of type systems.  In the case that the tail class is irrational, the corresponding type system will be infinite, while for a rational tail class the type system is atomic quasinuclear type system as discussed in Theorem~\ref{thm:quasicycle}.  It was pointed out to us by James Hyde that the centralizer of the transposition~$\swap{0}{1}$ is also a maximal subgroup of~$V$.  For this we can also construct an infinite type system that is stabilized by the subgroup.  

In the present article, we have primarily concentrated our attention on finite type systems and a subsequent one will focus on the maximal subgroups arising as stabilizers of infinite type systems.  The questions that follow are intended to explore the maximal subgroups at the boundaries of what can be achieved with our techniques.

\QMain*

If the answer to Question~\ref{ques:main} were ``Yes'' then any proper subgroup of $V$ that contains a transposition of small support would need to preserve a type system.  The following two questions arose from attempts to prove this statement.

\begin{qu}
Suppose $G$ is a subgroup of $V$ that acts transitively on proper cones, does not preserve any type system, and contains at least one transposition of small support. Is $G$ swap-primitive on cones (and thus equal to $V$)?
\end{qu}

\begin{qu}
Suppose $G$ is a proper subgroup of $V$ that contains at least one transposition of small support. Define an equivalence relation on cones via $\delta \sim \gamma$ if $\delta$ is connected to $\gamma$ by a $G$-chain.  Is this equivalence relation a type system? 
\end{qu}

If the answer to Question \ref{ques:main} were ``No'' then it might be possible to find a maximal subgroup of $V$ which preserves a circular order on the dyadics, but which was not an automorphic image of $T$.  For this reason, we pose the following:

\begin{qu}
\label{ques:AutT}
If $G$ is a maximal subgroup of $V$ which preserves a circular order on the dyadics, is $G$ an automorphic image of $T$?
\end{qu}

Finally, in view of Theorem~\ref{thm:finitetypes}:

\begin{qu}
Is there a useful classification of infinite simple type systems?
\end{qu}

\bibliographystyle{alpha}

\begin{thebibliography}{BBGL{\etalchar{+}}13}

\bibitem[AN21a]{Aiello-Nangnibeda21a21b}
Valeriano Aiello and Tatiana Nagnibeda.
\newblock On the 3-colorable subgroup $\mathcal{F}$ and maximal subgroups of
  {T}hompson's group {$F$}.
\newblock submitted. arXiv:1912.04730, 2021.

\bibitem[AN21b]{Aiello-Nangnibeda21a21a}
Valeriano Aiello and Tatiana Nagnibeda.
\newblock On the oriented {T}hompson subgroup $\vec{F}_3$ and its relatives in
  higher {B}rown--{T}hompson groups.
\newblock {\em J. Algebra Appl.}, pages 1--20, 2021.
\newblock To appear. arXiv:1912.04730.

\bibitem[AS85]{AS85}
M.~Aschbacher and L.~Scott.
\newblock Maximal subgroups of finite groups.
\newblock {\em J. Algebra}, 92(1):44--80, 1985.

\bibitem[BBGL{\etalchar{+}}13]{BBGGHMS}
Collin Bleak, Hannah Bowman, Alison Gordon~Lynch, Garrett Graham, Jacob Hughes,
  Francesco Matucci, and Eugenia Sapir.
\newblock Centralizers in the {R}. {T}hompson group {$V_n$}.
\newblock {\em Groups Geom. Dyn.}, 7(4):821--865, 2013.

\bibitem[BCR18]{BCR18}
Jos\'{e} Burillo, Sean Cleary, and Claas~E. R\"{o}ver.
\newblock Obstructions for subgroups of {T}hompson's group {$V$}.
\newblock In {\em Geometric and cohomological group theory}, volume 444 of {\em
  London Math. Soc. Lecture Note Ser.}, pages 1--4. Cambridge Univ. Press,
  Cambridge, 2018.

\bibitem[BEM23]{BEM}
James Belk, Luke Elliott, and Francesco Matucci.
\newblock A short proof of {R}ubin's theorem.
\newblock {\em Israel J. of Math., to appear}, pages 1--10, 2023.

\bibitem[BG85]{browngeoghegan1}
Kenneth~S. Brown and Ross Geoghegan.
\newblock Cohomology with free coefficients of the fundamental group of a graph
  of groups.
\newblock {\em Comment. Math. Helv.}, 60(1):31--45, 1985.

\bibitem[BM97]{BurgerMozes}
Marc Burger and Shahar Mozes.
\newblock Finitely presented simple groups and products of trees.
\newblock {\em C. R. Acad. Sci. Paris S\'{e}r. I Math.}, 324(7):747--752, 1997.

\bibitem[BMN16]{BMN16}
Collin Bleak, Francesco Matucci, and Max Neunh\"{o}ffer.
\newblock Embeddings into {T}hompson's group {$V$} and {$co\mathcal{CF}$}
  groups.
\newblock {\em J. Lond. Math. Soc. (2)}, 94(2):583--597, 2016.

\bibitem[BQ17]{BQ17}
Collin Bleak and Martyn Quick.
\newblock The infinite simple group {$V$} of {R}ichard {J}. {T}hompson:
  presentations by permutations.
\newblock {\em Groups Geom. Dyn.}, 11(4):1401--1436, 2017.

\bibitem[Bri04]{brinHigherV}
Matthew~G. Brin.
\newblock Higher dimensional {T}hompson groups.
\newblock {\em Geom. Dedicata}, 108:163--192, 2004.

\bibitem[Bro87]{brown3}
Kenneth~S. Brown.
\newblock Finiteness properties of groups.
\newblock {\em Journal of Pure and Applied Algebra}, 44(1):45--75, 1987.

\bibitem[Bro92]{brown4}
Kenneth~S. Brown.
\newblock The geometry of finitely presented infinite simple groups.
\newblock In {\em Algorithms and classification in combinatorial group theory
  (Berkeley, CA, 1989)}, volume~23 of {\em Math. Sci. Res. Inst. Publ.}, pages
  121--136. Springer, New York, 1992.

\bibitem[BSD13]{bleak-salazar1}
Collin Bleak and Olga Salazar-D{\'{\i}}az.
\newblock Free products in {R}. {T}hompson's group {$V$}.
\newblock {\em Trans. Amer. Math. Soc.}, 365(11):5967--5997, 2013.

\bibitem[CFP96]{CFP}
J.~W. Cannon, W.~J. Floyd, and W.~R. Parry.
\newblock Introductory notes on {R}ichard {T}hompson's groups.
\newblock {\em Enseign. Math. (2)}, 42(3-4):215--256, 1996.

\bibitem[Coh93]{Coh93}
Daniel~E. Cohen.
\newblock String rewriting---a survey for group theorists.
\newblock In {\em Geometric group theory, {V}ol. 1 ({S}ussex, 1991)}, volume
  181 of {\em London Math. Soc. Lecture Note Ser.}, pages 37--47. Cambridge
  Univ. Press, Cambridge, 1993.

\bibitem[DM96]{DixonMort}
John~D. Dixon and Brian Mortimer.
\newblock {\em Permutation groups}, volume 163 of {\em Graduate Texts in
  Mathematics}.
\newblock Springer-Verlag, New York, 1996.

\bibitem[Dyd77a]{Dydak1}
Jerzy Dydak.
\newblock 1-movable continua need not be pointed 1-movable.
\newblock {\em Bull. Acad. Polon. Sci. S\'{e}r. Sci. Math. Astronom. Phys.},
  25(6):559--562, 1977.

\bibitem[Dyd77b]{Dydak2}
Jerzy Dydak.
\newblock A simple proof that pointed {FANR}-spaces are regular fundamental
  retracts of {ANR}'s.
\newblock {\em Bull. Acad. Polon. Sci. S\'{e}r. Sci. Math. Astronom. Phys.},
  25(1):55--62, 1977.

\bibitem[Far18]{FarleyFSS}
Daniel Farley.
\newblock Local similarity groups with context-free co-word problem.
\newblock In {\em Topological methods in group theory}, volume 451 of {\em
  London Math. Soc. Lecture Note Ser.}, pages 67--91. Cambridge Univ. Press,
  Cambridge, 2018.

\bibitem[FH93]{FreydHeller}
Peter Freyd and Alex Heller.
\newblock Splitting homotopy idempotents. {II}.
\newblock {\em J. Pure Appl. Algebra}, 89(1-2):93--106, 1993.

\bibitem[GLU21]{glu21}
Anthony Genevois, Anne Lonjou, and Christian Urech.
\newblock Asymptotically rigid mapping class groups ii: strand diagrams and
  nonpositive curvature.
\newblock submitted. arXiv:2110.06721, 2021.

\bibitem[GS87]{GhysSergiescu}
{\'E}tienne Ghys and Vlad Sergiescu.
\newblock Sur un groupe remarquable de diff\'eomorphismes du cercle.
\newblock {\em Comment. Math. Helv.}, 62(2):185--239, 1987.

\bibitem[GS97]{GS97}
Victor Guba and Mark Sapir.
\newblock Diagram groups.
\newblock {\em Mem. Amer. Math. Soc.}, 130(620):viii+117, 1997.

\bibitem[GS17a]{Golan-Sapir17a}
G.~Golan and M.~Sapir.
\newblock On the stabilizers of finite sets of numbers in the {R}. {T}hompson
  group {$F$}.
\newblock {\em Algebra i Analiz}, 29(1):70--110, 2017.

\bibitem[GS17b]{Golan-Sapir17b}
Gili Golan and Mark Sapir.
\newblock On {J}ones' subgroup of {R}. {T}hompson group {$F$}.
\newblock {\em J. Algebra}, 470:122--159, 2017.

\bibitem[GS17c]{Golan-Sapir17c}
Gili Golan and Mark Sapir.
\newblock On subgroups of {R}. {T}hompson's group {$F$}.
\newblock {\em Trans. Amer. Math. Soc.}, 369(12):8857--8878, 2017.

\bibitem[Leh08]{lehnertthesis}
J.~Lehnert.
\newblock {\em Gruppen von quasi-Automorphismen}.
\newblock PhD thesis, Goethe Universit\"at Frankfurt am Main, 2008.
\newblock
  \url{http://publikationen.ub.uni-frankfurt.de/frontdoor/index/index/docId/5795}.

\bibitem[Mat15a]{Matsumoto}
Kengo Matsumoto.
\newblock Full groups of one-sided topological {M}arkov shifts.
\newblock {\em Israel J. Math.}, 205(1):1--33, 2015.

\bibitem[Mat15b]{Matui}
Hiroki Matui.
\newblock Topological full groups of one-sided shifts of finite type.
\newblock {\em J. Reine Angew. Math.}, 705:35--84, 2015.

\bibitem[MT73]{thompsonmckenzie}
Ralph McKenzie and Richard~J. Thompson.
\newblock An elementary construction of unsolvable word problems in group
  theory.
\newblock In {\em Word problems: decision problems and the Burnside problem in
  group theory (Conf., Univ. California, Irvine, Calif. 1969; dedicated to
  Hanna Neumann)}, volume~71 of {\em Studies in Logic and the Foundations of
  Math.}, pages 457--478. North-Holland, Amsterdam, 1973.

\bibitem[Rub96]{Rubin}
Matatyahu Rubin.
\newblock Locally moving groups and reconstruction problems.
\newblock In {\em Ordered groups and infinite permutation groups}, volume 354
  of {\em Math. Appl.}, pages 121--157. Kluwer Acad. Publ., Dordrecht, 1996.

\bibitem[Sav10]{Savchuk}
Dmytro Savchuk.
\newblock Some graphs related to {T}hompson's group {$F$}.
\newblock In {\em Combinatorial and geometric group theory}, Trends Math.,
  pages 279--296. Birkh\"{a}user/Springer Basel AG, Basel, 2010.

\bibitem[Tho65]{Thompson}
Richard~J. Thompson.
\newblock Handwritten widely circulated notes.
\newblock Unpublished, 1965.

\bibitem[Tho80]{Thompson76}
Richard~J. Thompson.
\newblock Embeddings into finitely generated simple groups which preserve the
  word problem.
\newblock In {\em Word problems, {II} ({C}onf. on {D}ecision {P}roblems in
  {A}lgebra, {O}xford, 1976)}, volume~95 of {\em Studies in Logic and the
  Foundations of Mathematics}, pages 401--441. North-Holland, Amsterdam-New
  York, 1980.

\end{thebibliography}
\newcommand{\etalchar}[1]{$^{#1}$}
\def\cprime{$'$}

\noindent
James Belk, \texttt{jim.belk@gmail.com}\\
School of Mathematics \& Statistics, University of Glasgow, Glasgow, UK\\[5pt]
Collin Bleak, \texttt{collin.bleak@st-andrews.ac.uk}\\
School of Mathematics \& Statistics, University of St Andrews, St Andrews, UK\\[5pt]
Martyn Quick, \texttt{mq3@st-andrews.ac.uk}\\
School of Mathematics \& Statistics, University of St Andrews, St Andrews, UK\\[5pt]
Rachel Skipper, \texttt{rachel.skipper@utah.edu}\\
Department of Mathematics, University of Utah, Salt Lake City, Utah, USA

\end{document}